\documentclass{amsart}
\usepackage[numbers]{natbib}
\usepackage{amsfonts,amsmath,amssymb,amsthm}
\usepackage[hmargin=1.3in,vmargin=1.3in]{geometry}
\usepackage{xcolor}
\usepackage{comment}
\usepackage{import}
\usepackage{dsfont}
\usepackage{graphicx}
\usepackage{enumitem}

\usepackage{caption}
\usepackage{subcaption}

\newcommand{\E}{\mathbb{E}}
\newcommand{\W}{\overline{\mathcal{W}}}
\newcommand{\iW}{\widehat{\mathcal{W}}}
\renewcommand{\P}{\mathbb{P}}
\renewcommand{\tilde}{\widetilde}

\newtheorem{rem}[]{Remark}
\newtheorem{ex}[]{Example}
\newtheorem{thm}[]{Theorem}

\newtheorem{lemma}[]{Lemma}
\newtheorem{corollary}[]{Corollary}

\newcommand{\R}{\mathbb{R}}                        %
\renewcommand{\P}{\mathbb{P}}                      %
\newcommand{\Q}{\mathbb{Q}}                        %
\newcommand{\norm}[1]{\lVert#1\rVert}              %
\usepackage{tikz}
\usetikzlibrary{decorations.markings}
\usetikzlibrary{shapes.geometric}

\pgfdeclarelayer{edgelayer}
\pgfdeclarelayer{nodelayer}
\pgfsetlayers{edgelayer,nodelayer,main}

\tikzstyle{none}=[inner sep=0pt]
\tikzset{new/.style={thick}}

\usepackage{dsfont}

\usepackage{caption}
\usepackage{subcaption}
 
\usepackage{graphicx} 

\begin{document}

\title[Out-of-sample prediction error of the $\sqrt{\text{LASSO}}$ et. al.]{The out-of-sample prediction error of the $\sqrt{\text{LASSO}}$ and related estimators}

\author{Jos\'e Luis Montiel Olea}
\address{Jos\'e Luis Montiel Olea\newline
\mbox{}\quad\ Department of Economics, Cornell University}
\email{montiel.olea@gmail.com}

\author{Cynthia Rush} 
\address{Cynthia Rush\newline
\mbox{}\quad\ Department of Statistics, Columbia University}
\email{cynthia.rush@columbia.edu}

\author{Amilcar Velez}
\address{Amilcar Velez\newline
\mbox{}\quad\ Department of Economics, Northwestern University}
\email{amilcare@u.northwestern.edu}

\author{Johannes Wiesel}
\address{Johannes Wiesel\newline
\mbox{}\quad\ Department of Mathematics, Carnegie Mellon University}
\email{wiesel@cmu.edu}

\maketitle

\begin{abstract}
We study the classical problem of predicting an outcome variable, $Y$, using a linear combination of a $d$-dimensional covariate vector, $\mathbf{X}$. We are interested in linear predictors whose coefficients solve:
\begin{align*}
\inf_{\boldsymbol{\beta} \in \mathbb{R}^d} \Big( \mathbb{E}_{\mathbb{P}_n} \big[ \big|Y-\mathbf{X}^{\top} \boldsymbol{\beta} \big|^r \big] \Big)^{1/r} +\delta \, \rho\left(\boldsymbol{\beta}\right),
\end{align*}
where $\delta>0$ is a regularization parameter, $\rho:\R^d\to \R_+$ is a convex penalty function, $\mathbb{P}_n$ is the empirical distribution of the data, and $r\geq 1$. Our main contribution is a new bound on the out-of-sample prediction error of such estimators. 

The new bound is obtained by combining three new sets of results. First, we provide conditions under which linear predictors based on these estimators solve a \emph{distributionally robust optimization} problem: they minimize the worst-case prediction error over distributions that are close to each other in a type of \emph{max-sliced Wasserstein metric}. Second, we provide a detailed finite-sample and asymptotic analysis of the statistical properties of the balls of distributions over which the worst-case prediction error is analyzed. Third, we present an oracle recommendation for the choice of regularization parameter, $\delta$, that guarantees good out-of-sample prediction error.   
\end{abstract}

\section{Introduction}

The extent to which prediction algorithms can perform well not just on \emph{training} data, but also on new, unseen, \emph{testing} inputs is a central concern in machine learning. 
In fact, reducing a predictor’s testing error—or equivalently, improving its ``out-of-sample’’ performance or ``generalization error’’—possibly at the expense of increased training error, is a typical informal motivation for introducing regularization strategies in statistical estimation; see, for example,  \cite[Chapter 7]{ESL2017} and \cite[Chapter 7]{DLbook:2016}. More generally, the study of issues related to problems in which training and testing environments differ from one another is the subject of several recent, rapidly growing areas of research at the intersection of machine learning and statistics: transfer learning \citep{kpotufe2021marginal}, distributional shifts \cite{duchi2021learning,sugiyama2007covariate,adjaho2022externally}, domain adaptation \citep{mansour2009domain, ben2010theory}, adversarial attacks \cite{kurakin2016adversarial, goodfellow2014explaining}, learning under biased sampling \cite{sahoo2022learning} and cross-domain transfer performance \cite{andrews2022transfer} are some relevant examples. 

In this paper, we study the classical problem of predicting an outcome variable, $Y$, using a linear combination of a $d$-dimensional covariate vector, $\mathbf{X}$. We focus on linear predictors whose coefficients, $\widehat{\boldsymbol{\beta}}$,  solve the problem:
\begin{equation}
\arg \inf_{\boldsymbol{\beta} \in \mathbb{R}^d} \Big( \mathbb{E}_{\mathbb{P}_n} \big[ \big|Y-\mathbf{X}^{\top}\boldsymbol{\beta} \big|^r \big] \Big)^{1/r} +\delta \, \rho(\boldsymbol{\beta}),
\label{eq:penalized}
\end{equation}
where $\delta>0$ is a regularization parameter, $\rho:\R^d\to \R_+$ is a convex penalty function, $\mathbb{P}_n$ is the empirical distribution of the data, and $r\geq 1$. We assume that both $\rho$ and $r$ have been determined by the statistician, and make no attempt to provide normative statements regarding their selection. The square-root LASSO (henceforth, $\sqrt{\text{LASSO}}$) \cite{belloni2011square}, the square-root group LASSO \cite{bunea2013group}, the square-root sorted $\ell_1$ penalized estimator (SLOPE) \cite{stucky2017sharp}, and the $\ell_1$-penalized least absolute deviation estimator \cite{wang2007robust} provide examples of estimators obtained by solving \eqref{eq:penalized}. 

We are interested in studying the out-of-sample prediction error associated to such estimators; namely 
\begin{equation} \label{eq:gen_intro}
\mathbb{E}_{\mathbb{Q}} \big[ \big|Y-\mathbf{X}^{\top}\widehat{\boldsymbol{\beta}} \big|^r \big].
\end{equation}
The expectation above is computed by fixing the estimated $\widehat{\boldsymbol{\beta}}$, and then drawing new covariates and outcomes according to some joint distribution $\mathbb{Q}$. The distribution $\mathbb{Q}$ is similar, but not necessarily equal to, the true data generating process, $\mathbb{P}$, or the empirical distribution of the data, $\mathbb{P}_n$. 

Our main result is the following upper bound on the out-of-sample prediction error (see \eqref{eq:generalization_Q} and Theorem \ref{cor:generalization} for a more formal statement). If $\delta$ is chosen appropriately, then, with high probability, the objective function in \eqref{eq:penalized} constitutes, up to some adjustment terms, an upper bound for the out-of-sample prediction error evaluated at any $\boldsymbol{\beta}$ and $\Q$:
\begin{align} 
\label{eq:gen_loss}
\E_{\Q} \big[\big|Y-\mathbf{X}^\top {\boldsymbol \beta}\big|^r\big]^{1/r} \le  \E_{{\P_n}} \big[\big|Y-\mathbf{X}^\top {\boldsymbol \beta}\big|^r\big]^{1/r} + \delta \big(1 + \rho\left({\boldsymbol \beta}\right) )+ \iW_r(\P,\Q)  (1 + \rho\left({\boldsymbol \beta}\right) ),
\end{align}
where $\iW_r$ denotes a type of \textit{max-sliced Wasserstein metric}. Consequently, linear predictors whose coefficients solve \eqref{eq:penalized}, for an appropriately chosen $\delta$, have good out-of-sample performance at the true, unknown distribution of the data $\P$, and, also, at \emph{testing} distributions $\Q$ that are close to $\P$ in terms of $\iW_r$. 

We present a formal definition of this metric in  \eqref{eq:Wass} below, and explain how distributions that are close in this metric are required to have similar prediction errors (in a sense we make precise).
The proof of the above is based on three intermediate results, which bring together ideas related to \emph{distributionally robust optimization} (DRO), finite sample analysis of the max-sliced Wasserstein metric, and empirical process theory. We believe that the three steps used to prove \eqref{eq:gen_loss} provide results that are interesting in their own right, and in what follows, we discuss each of these steps in more detail.

First, we show that estimators constructed using \eqref{eq:penalized} are equivalent to those that solve a  DRO problem based on a $\iW_r$-ball
 around $\mathbb{P}_n$ (Theorem \ref{thm:penalized}, Section \ref{sec:dro}). 
 The DRO representation naturally yields finite-sample bounds for \eqref{eq:gen_intro} in terms of \eqref{eq:penalized}, provided that distributions $\mathbb{Q}$ are close to $\mathbb{P}_n$ in terms of our suggested metric (Section \ref{sec:examples} provides examples of distributions contained in our balls). Thus, our first result provides theoretical support for the claim that predictors based on estimators obtained via \eqref{eq:penalized} (such as the $\sqrt{\text{LASSO}}$ and related estimators) have good out-of-sample performance.

Second, we provide a detailed statistical analysis of the balls of distributions based on our suggested metric. More precisely, we determine the required size of a ball centered on $\mathbb{P}_n$ to guarantee that it contains $\mathbb{P}$ with high probability. We present both finite-sample results (Theorem \ref{thm:1} and Theorem \ref{thm:rates} in Section \ref{sec:rates}) and large-sample approximations (Theorem \ref{thm:ass_bdd} and \ref{thm:ass} in Section \ref{sec:asymptotic}). Our analysis suggests that our balls are \emph{statistically larger} than those based on the standard Wasserstein metric (Remark \ref{rem:Wass_ball}). Because the balls we consider are statistically larger, their radii can shrink to zero faster than order $n^{-1/d}$ (the usual rates for Wasserstein balls), and still contain $\mathbb{P}$ (see Figure \ref{fig:1}). 

\definecolor{darkpastelgreen}{rgb}{0.01, 0.75, 0.24}
\definecolor{azure}{rgb}{0.0, 0.5, 1.0}

\begin{figure}[h!]
\begin{tikzpicture}[scale=.6]
    \draw[fill= darkpastelgreen] (-7,0) -- (0,7) -- (7,0)-- (0,-7)-- cycle;   
    \draw[fill=azure ] (0,0) circle [radius =3];
    \draw[fill= yellow ] (-3,0) -- (0,3) -- (3,0)-- (0,-3)-- cycle;   
    \draw[fill=blue] (0,0) circle [radius=0.1];
    \draw[fill=orange] (-3,0) circle [radius=0.1];
     \node[align=left, color = orange] at (-3.5,0) {$\P^*$};
    \node[below, color = blue] at (0,0) {$\hat{\P}_n$};
    \node[above, align=right, color=red] at (1.5, 1.8) {$\P$};
    \draw[color=red, fill=red] (1.5,1.8) circle [radius=0.1];
    \draw[dashed] (0,0) -- (-3.5, 3.5);
    \node[above, align=right] at (-0.4,1.1) {$n^{-1/d}$};
    \draw[dashed] (0,0) -- (-3, 0);
    \node[above, align=right] at (1.6,-1.1) {$n^{-1/2}$};
    \draw[dashed] (0,0) -- (2.12, -2.12);
     \node[above, align=right] at (-1.2,-0.9) {$n^{-1/2}$};
\end{tikzpicture}
\caption{$\rho$-max-sliced Wasserstein ball of radius $n^{-1/2}$ (blue) vs. $d$-dimensional Wasserstein ball of radius $n^{-1/2}$ (yellow) and $n^{-1/d}$ (green). The measure $\P^*$ (orange) is the optimal perturbation in the DRO formulation.}
\label{fig:1}
\end{figure}
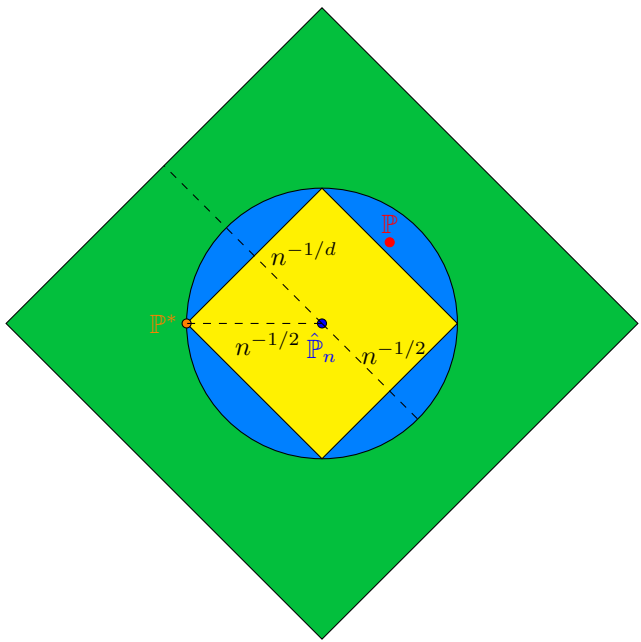

Third, we use the DRO representation of \eqref{eq:penalized} and the statistical analysis of our max-sliced Wasserstein balls to i) derive oracle recommendations for the penalization parameter $\delta$ (Section \ref{sec:recommendation}) that guarantee good out-of-sample prediction error (Theorem \ref{cor:generalization} in Section \ref{sec:recommendation_finite}); and ii) present a test statistic to rank the out-of-sample performance of two different linear estimators (Section \ref{sec:ranking_estimators}).  In Section \ref{sec:simulation} we present a small-scale simulation to illustrate the performance of predictions based on the $\sqrt{\text{LASSO}}$ but using our recommended parameter $\delta$.

None of our results rely on sparsity assumptions about the true data generating process; thus, they broaden the scope of use of the $\sqrt{\text{LASSO}}$ and related estimators in prediction problems.

We now provide an overview of the technical details of our main results. 

\subsection{Main Contributions}

\subsubsection{DRO formulation}

Our first result shows that linear predictors whose coefficients solve \eqref{eq:penalized} minimize the worst-case, out-of-sample prediction error attained over a ball of distributions centered around $\mathbb{P}_n$. This ball is defined by what we call the \emph{$\rho$-max-sliced Wasserstein ($\rho$-MSW) metric:}\footnote{Sliced Wasserstein distances \cite{rabin2011wasserstein, bonneel2015sliced}---i.e.,  distances between probability distributions that consider the average or maximum of standard Wasserstein distances between one-dimensional projections---have been the subject of recent research in statistics and machine learning; see, for example, \cite{kolouri2019generalized}, \cite{kengo12022statistical} and the references therein. As we discuss later in the paper, the max-sliced Wasserstein distance has been studied recently in \cite{niles2019estimation, lin2021projection,bartl2022structure}. Its use in the analysis of out-of-sample prediction error of the $\sqrt{\text{LASSO}}$ and related estimators yields a hitherto unexplored connection to the field of statistical optimal transport, which we hope to be attractive from a methodological point of view.}
\begin{equation} 
\iW_{r,\rho,\sigma}(\P,\tilde \P) :=\sup_{\boldsymbol{\gamma} \in \mathbb{R}^d} \Big(
 \inf_{\substack{\pi \in \Pi(\mathbb{P}, \tilde{\mathbb{P}}) }}
  \frac{1}{\sigma+\rho(\boldsymbol{\gamma})} \Big( \mathbb{E}_{\pi}\big[\big| (Y  - \mathbf{X}^{\top}  \boldsymbol\gamma) - (\tilde{Y}-\widetilde{\mathbf{X}}^{\top} \boldsymbol{\gamma} )  \big|^r \big] \Big)^{1/r} \Big).
\label{eq:Wass}
\end{equation}
Here, for arbitrary distributions $\mathbb{P}$ and $\widetilde{\mathbb{P}}$, the set $\Pi(\mathbb{P}, \widetilde{\mathbb{P}})$ denotes
the collection of probability distributions over random vectors $((\mathbf{X}^{\top}, Y), \, (\widetilde{\mathbf{X}}^{\top}, \widetilde{Y}))$, with marginal distributions $(\mathbb{P}, \widetilde{\mathbb{P}})$ (that is, the set $\Pi(\mathbb{P}, \widetilde{\mathbb{P}})$ is the collection of \emph{couplings} of $\mathbb{P}$ and $\widetilde{\mathbb{P}}$). We assume $r, \sigma\ge 1$, refer to $r$ as the Wasserstein exponent, and take $\sigma$ to be an auxiliary  hyperparameter.\footnote{For all the results in the paper, with the exception of Theorem \ref{thm:pivotal}, the hyperparameter $\sigma$ can be treated as an arbitrary positive constant (and in fact, without any loss of generality, it can be chosen to equal one). In Section \ref{sec:rates} we explain that introducing the hyperparameter $\sigma > 0$ is needed to allow for $\mathbb{P}$ and $\tilde{\mathbb{P}}$ to have different marginal distributions for the outcome variable. In the same section we argue that setting $\sigma \geq 1$ allows us to focus our statistical analysis on the usual MSW metric.} 

Intuitively, $\mathbb{P}$ and $\widetilde{\mathbb{P}}$ are close in the $\rho$-MSW metric, with Wasserstein exponent $r$, if for any $\boldsymbol \gamma$ there exists a coupling of $\mathbb{P}$ and $\widetilde{\mathbb{P}}$ that makes the $r$-th norm of the \emph{difference of their prediction errors} small, relative to $\rho(\boldsymbol \gamma)$.\footnote{Lemma \ref{lem:metric} shows that the $\rho$-MSW metric is indeed a metric.}

Formally, Theorem \ref{thm:penalized} in Section \ref{sec:dro} shows that $\widehat{\boldsymbol{\beta}}$ solves \eqref{eq:penalized} if and only if it solves the distributionally robust optimization problem
\begin{equation}
 \inf_{\boldsymbol{\beta} \in \mathbb{R}^d} \Big(  \sup_{\: \widetilde{\mathbb{P}} \,\in \, B^{r,\rho,\sigma}_{\delta}(\mathbb{P}_n)} \E_{\tilde\P}\big[ \big|Y-\mathbf{X}^\top \boldsymbol{\beta} \big|^r \big] \Big),
\label{eq:min2}
\end{equation}
where $B^{r,\rho,\sigma}_{\delta}(\P_n)$ is defined as the ball centered around the empirical distribution of the data, $\P_n$, collecting all the distributions $\tilde{\mathbb{P}}$ for which $\iW_{r,\rho,\sigma}(\P_n, \tilde{\P})$
is smaller than $\delta$. By construction, the minimax problem \eqref{eq:min2} provides robustness in situations where i) the trained procedure will be evaluated on test data from a distribution $\widetilde{\mathbb{P}}$ that is close to that of the training data, $\mathbb{P}$, but may be different \cite{ben2010theory}; ii) where there are covariate shifts \cite{shimodaira2000improving, quinonero2008dataset, sugiyama2007covariate, sugiyama2005input, agarwal2011linear, wen2014robust, reddi2015doubly, chen2016robust}; or iii) when there is an adversarial attack \cite{kurakin2016adversarial, goodfellow2014explaining}.

It is useful to compare the $\rho$-MSW metric to the $d$-dimensional Wasserstein metric with cost $\|\cdot\|$, defined by
\begin{equation}
\mathcal{W}_r(\P,\tilde\P) := \inf_{\substack{\pi  \, \in \, \Pi(\mathbb{P}, \widetilde{\mathbb{P}})}}
 \Big( \mathbb{E}_{\pi}\big[ \|(\widetilde{\mathbf{X}}^\top,\tilde{Y})-(\mathbf{X}^\top,Y)\|^r \big] \Big)^{1/r},
 \label{eq:Wass2}
 \end{equation}
 where $\|\cdot\|$ is an arbitrary metric in $\mathbb{R}^{d+1}$. Remark \ref{rem:Wass_ball} in Section \ref{sec:dro} shows that for a large class of penalty functions $\rho$, the balls based on \eqref{eq:Wass} will typically be larger than those based on \eqref{eq:Wass2}. 

It is also useful to note that our $\rho$-MSW metric is a slight generalization of the \emph{max-sliced Wasserstein metric} (MSW), first considered in \cite{deshpande2019max,kolouri2019generalized,paty2019subspace,niles2019estimation}. Broadly speaking, the MSW distance over probability distributions $\mathbb{P}$ and $\tilde{\mathbb{P}}$ on $\mathbb{R}^{d+1}$ is defined as 
\begin{equation} \label{eqn:MSW_intro}
\overline{\mathcal{W}}_{r}(\mathbb{P},\tilde{\mathbb{P}}) := \sup_{ \tilde{\boldsymbol{\gamma}} \in \mathbb{R}^{d+1} : \|\tilde{\boldsymbol{\gamma}}\|_{2} = 1 }   \mathcal{W}_r( \tilde{\boldsymbol{\gamma}}_{*} \mathbb{P} ~, \tilde{\boldsymbol{\gamma}}_{*} \tilde{\mathbb{P}} )~,
\end{equation}
where  $\mathcal{W}_r$ is the one-dimensional Wasserstein metric, and $\tilde{\boldsymbol{\gamma}}_{*} \P$ denotes the pushforward probability of $\mathbb{P}$ with respect to the linear map $\mathbf{z} \in \mathbb{R}^{d+1} \mapsto \mathbf{z}^{\top} \tilde{\boldsymbol{\gamma}} \in \mathbb{R} $. The supremum is defined over all linear, one-dimensional projections generated by the vectors in the unit sphere. We provide further details about the MSW metric in the discussion following Equation \eqref{eq:bar_W}.

To further illustrate the similarities between the MSW and our $\rho$-MSW metric, it is convenient to assume, for the moment, that $\rho(\cdot)$ is a norm on $\mathbb{R}^d$. For an arbitrary scalar $\sigma > 0$, define the function  $\|\cdot\|_{\rho,\sigma}$ on $\mathbb{R}^{d+1}$ via $\|\tilde{\boldsymbol{\gamma}}\|_{\rho,\sigma} = \sigma |y| + \rho(\boldsymbol{\gamma})$, where $\tilde{\boldsymbol{\gamma}} = (\boldsymbol{\gamma},y)$, $\boldsymbol{\gamma} \in \mathbb{R}^d$ and $y \in \mathbb{R}$. Note that $\|\tilde{\boldsymbol{\gamma}}\|_{\rho,\sigma}$ is a norm. Then, Lemma \ref{lem:metric_new} in the Supplementary Material \cite{MRVW2024} shows that our $\rho$-MSW metric can be written as
            $$ \widehat{W}_{r,\rho,\sigma}(\mathbb{P},\tilde{\mathbb{P}}) = \sup_{ \tilde{\boldsymbol{\gamma}} \in \mathbb{R}^{d+1} : \|\tilde{\boldsymbol{\gamma}}\|_{\rho,\sigma} = 1 }   \mathcal{W}_r( \tilde{\boldsymbol{\gamma}}_{*} \mathbb{P} ~, \tilde{\boldsymbol{\gamma}}_{*} \tilde{\mathbb{P}} )~.$$
Thus, when $\rho(\cdot)$ is a norm, the only difference between our $\rho$-MSW metric and the usual MSW metric is the set of one-dimensional projections that are used to define each metric. For example, when $\rho(\cdot) = \| \cdot \|_1$ and $\sigma = 1$, the $\rho$-MSW metric considers the supremum on the unit sphere defined by the $\ell_1$-norm, while the MSW metric considers the supremum on the unit sphere defined by the $\ell_2$-norm. In both cases, restricting the norm of the one-dimensional projections is needed to guarantee that these metrics are finite. Thus, one can think of $\rho(\cdot)$ as normalizing the linear, one-dimensional projections (and consequently, balancing out the numerator in Equation \eqref{eq:Wass}).

To the best of our knowledge, the result showing that linear predictors whose coefficients solve \eqref{eq:penalized} equivalently minimize the worst-case, out-of-sample prediction error attained over a neighborhood $B_\delta^{r,\rho, \sigma}(\P_n)$ based on the $\rho$-MSW metric is new. The connection between \eqref{eq:penalized} and \eqref{eq:min2} for the case $r=2$, penalty $\rho(\cdot)=\|\cdot\|_p$, $p\ge 1$, and $\iW_r$ replaced by $\mathcal{W}_r$, the Wasserstein metric, was first established in \cite[Theorem 1]{blanchet2019robust}, using optimal transport (OT) duality \cite{gao2023distributionally, blanchet2019quantifying}. These results have recently been extended to more general penalty functions \cite{Chuetal,wu2022generalization}. Our balls are different to the ones considered in these papers: we focus on convex penalty functions, and also our proofs do not rely on duality arguments. Instead, we explicitly identify a worst-case measure $\P^*\in B_\delta^{r,\rho,\sigma}(\P_n)$ for \eqref{eq:min2}. Our results show that the measure $\P^*$ is given by an additive perturbation of $\P_n$ (see Corollary \ref{cor:1}). In this sense, our proof can be seen as a natural extension of the seminal results in \cite[Theorem 1]{bertsimas2018characterization}. Moreover, we believe that $\iW_{r,\rho,\sigma}$ is the natural metric to assess out-of-sample performance, as it allows for the construction of neighborhoods containing a general class of \emph{testing} distributions that are only required to generate similar prediction errors as the \emph{training} distribution $\mathbb{P}_n$. To see this, note that the out-of-sample prediction error, $\mathbb{Q} \mapsto \mathbb{E}_\mathbb{Q}[| \Tilde{Y} - \Tilde{\mathbf{X}}^{\top} \boldsymbol{\boldsymbol{\beta}} |^r  ]^{1/r}$, is a Lipschitz continuous function under the $\rho$-MSW metric for any given $\boldsymbol{\beta}$ and penalty function $\rho$. Specifically, for any two distributions $\mathbb{Q}$ and $\mathbb{P}$, the definition of $\rho$-MSW implies
$$ \mid \mathbb{E}_\mathbb{Q}[ \mid \Tilde{Y} - \Tilde{\mathbf{X}}^{\top} \boldsymbol{\boldsymbol{\beta}} \mid^r ]^{1/r} - \mathbb{E}_\mathbb{P}[ \mid Y - \mathbf{X}^{\top} \boldsymbol{\beta} \mid^r ]^{1/r} \mid \le \left(\sigma + \rho(\boldsymbol{\boldsymbol{\beta}}) \right) \widehat{\mathcal{W}}_{r, \rho, \sigma}(\mathbb{P},\mathbb{Q}).$$
Consequently, $\sigma + \rho(\boldsymbol{\beta})$ can be interpreted as a Lipschitz constant that varies with $\boldsymbol{\beta}$. Thus, for fixed $\boldsymbol{\beta}$, any two distributions that are close under the $\rho$-MSW metric have similar prediction errors. The standard $d$-dimensional Wasserstein metric puts additional restrictions on the testing distributions considered. This means that two distributions can have similar prediction errors, but their Wasserstein distance could be very large, especially in high dimensions (making the associated bounds not very useful in practice). In Section \ref{subsection:gaussian_example} of the Supplementary Material \cite{MRVW2024} we further provide an example of two Gaussian distributions for which the difference in prediction errors is small, but the the standard Wasserstein metric is large.

\subsubsection{Statistical Analysis of $\iW_{r,\rho,\sigma}(\P_n, \P)$}

Our second set of results provide a detailed analysis of the statistical properties of \eqref{eq:Wass}. For simplicity in the exposition, we focus on the case when $\rho$ satisfies the inequality
\begin{align}\label{eq:new}
c_d \|(\boldsymbol{\gamma}, -1)\|\le \rho(\boldsymbol{\gamma})+1,
\end{align}
for a constant $c_d>0$ and a norm  $\|\cdot\|$ on $\R^{d+1}$. We also define its dual norm $\|\cdot\|_*= \sup_{\|\mathbf{x}\|=1} \langle \mathbf{x}, \mathbf{y}\rangle$; e.g.\ for $\|\cdot\|= \|\cdot\|_p$ for some $p\ge 1$ we have $\|\cdot\|_*=\|\cdot\|_q$, where $q=p/(p-1)$. Theorem \ref{thm:rates} in Section \ref{sec:rates} shows that if
\begin{equation*}
\Gamma:= \E_{\P}\big[   \|(\mathbf{X}^\top,Y ) \|_*^{s}\big]<\infty, \qquad \text{for some }s>2r,
\end{equation*}
then with a probability greater than $1-\alpha$ we have that 
\begin{equation}
\label{eq:W_bound_q_2_intro}
\iW_{r,\rho,\sigma}(\mathbb{P}_n, \P)^r \le \left( \max \left\{\frac{1}{c_d}, \, 1 \right\} \right)^r \frac{C\,\log(2n+1)^{r/s}}{\sqrt{n}} ~,
\end{equation}
where $C$ is the constant in \eqref{eq:constant_thm3} and is a function of the parameters $(r, d, \alpha, \Gamma, s)$. Furthermore, Theorem \ref{thm:ass} in Section \ref{sec:asymptotic} shows there exists a constant $C :=C(r,s,d)$, such that for all $x\ge 0$,
\begin{align}
\label{eq:W_bound_q_2_intro2}
\limsup_{n\to \infty} \, \P\left( \sqrt{n} \,  \iW_{r,\rho,\sigma}\left(\P_{n},\P\right)^r \ge x \right) \le  \P\left(\, \sup_{f \in \mathcal{F}} |G_f|\ge \frac{x}{C\sqrt{\Gamma} } \,\right),
\end{align}
where $(G_f)_{f\in \mathcal{F}}$ is a zero-mean Gaussian process specified in Theorem \ref{thm:ass}. 

The proofs of these results are based on a novel connection between an upper bound for the Wasserstein distance in $d=1$, and classical bounds from empirical process theory for self-normalized processes. Its relative simplicity enables us to find the explicit constants above.

\subsubsection{Applications}

\emph{Choosing $\delta$:} Our statistical analysis provides a concrete \emph{oracle} recommendation to select the regularization parameter $\delta_{n,r}$ in \eqref{eq:penalized} to be equal to the $(1/r)$-$th$ power of the right-hand side of \eqref{eq:W_bound_q_2_intro}; see Section \ref{sec:recommendation_finite}. Of course, the oracle recommendation for $\delta_{n,r}$ is typically not feasible as it depends on the unknown parameter $\Gamma$. In Section \ref{sec:normalization}, we present a simple strategy to normalize the sample covariates that guarantees that both \eqref{eq:W_bound_q_2_intro} and \eqref{eq:generalization_Q} hold with $\Gamma=2^s$ and $\sigma = \max\{ \E_{\P}[|Y|^s]^{1/s}, 1\}$. This means that we can turn our oracle recommendation into a simple formula that only depends on the true and unknown distribution of the data through the $s$-th moment of the outcome (which is typically easy to estimate), while also guaranteeing robustness to perturbations of the test dataset distribution from that of the training data in the form of \eqref{eq:generalization_Q} below. This \emph{pivotality} was the original motivation for the use of the $\sqrt{\text{LASSO}}$ and related estimators. %
The theoretical results in this paper (as well as the numerical simulations reported in Section \ref{sec:additional_sims} of the Supplementary Material \cite{MRVW2024}) suggest that, beyond pivotality, there are benefits---in terms of out-of-sample performance under a variety of testing distributions--- of using the $\sqrt{\textrm{LASSO}}$ and related estimators. 

We also note that our recommendation for the selection of regularization parameter does not rely on any sparsity assumption. We think this is an important point, as recent work \cite{Domenico2021,Ulrich2021} has argued that sparsity might not always be a compelling starting point in applications. 

Finally, while we are able to provide a recommendation for $\delta$, our current results do not allow us to say anything concrete about the selection of penalty function, $\rho$. This is in part due to the fact that, in our framework, we have too much flexibility making this choice. To illustrate this point, suppose that we wanted to pick the penalty function to optimize the out-of-sample prediction error of a linear predictor based on \eqref{eq:penalized} at a known distribution $\mathbb{Q}$. If $n$ denotes the sample size, we could always pick the convex penalty function
$$\rho^n(\boldsymbol{\beta}): = n \left( \mathrm{E}_{\mathbb{Q}}\left[\left| Y- \mathbf{X}^{\top} \boldsymbol{\beta} \right|^r \right] \right) ^{1/r}.$$
As $n$ grows to infinity, the relevance of the penalty function increases, and thus the solution of \eqref{eq:penalized} converges to
$$ \arg \inf_{\boldsymbol{\beta} \in \mathbb{R}^d}  \left( \mathbb{E}_{\Q} \left[ \left|Y-\mathbf{X}^{\top}\boldsymbol{\beta} \right|^r \right] \right)^{1/r}.$$
This just formalizes the obvious point that if we know the testing distribution $\mathbb{Q}$ at which we would like to have good performance, then it is better to use the best predictor based on such a distribution.

\emph{Bounds on out-of-sample performance:} When the available sample consists of independent and identical (i.i.d.) draws from a distribution $\mathbb{P}$ for which $\Gamma < \infty$ we can show that the objective function in \eqref{eq:penalized} provides explicit bounds on the out-of-sample performance of any linear estimator. In particular, Theorem \ref{cor:generalization} in Section \ref{sec:recommendation_finite} shows that for any \emph{testing} distribution $\mathbb{Q}$ for which $\iW_{r,\rho,\sigma}(\mathbb{P}, \mathbb{Q}) \leq \epsilon $, we have  with probability at least $1-\alpha$,
\begin{equation} \label{eq:generalization_Q}
\E_{\Q} \big[|Y-\mathbf{X}^\top {\boldsymbol \beta}|^r\big]^{1/r} \le  \E_{{\P_n}} \big[|Y-\mathbf{X}^\top {\boldsymbol \beta}|^r\big]^{1/r} + (\delta_{n,r} + \epsilon ) (\sigma + \rho({\boldsymbol \beta}) ),~ \quad \forall {\boldsymbol \beta}.
\end{equation}

\emph{Ranking the out-of-sample performance of competing estimators.} Finally, we present a test statistic to rank the out-of-sample performance of two different linear estimators (Section \ref{sec:ranking_estimators}). 

\subsection{Related Literature} 
The distributionally robust optimization problem in \eqref{eq:min2} has been shown to be equivalent to various forms of penalized regression, variance-penalized estimation, and dropout training  \cite{mohajerin2018data, gao2023distributionally, lam2016robust, blanchet2020machine,  blanchet2019quantifying, duchi2021statistics, nguyen2021mean}, depending on the choice of uncertainty set. It is typical to define uncertainty sets using metrics or divergences: e.g., total variation, Hellinger, Gelbrich distance \cite{nguyen2021mean} or Kullback-Leibler divergence \cite{renyi1961measures,christensen2023counterfactual}. To the best of our knowledge, the use of the max-sliced Wasserstein metric to define uncertainty sets in DRO problems is novel. 

As we have discussed above, the equivalence between \eqref{eq:penalized} and \eqref{eq:min2} has been established in \cite[Theorem 1]{blanchet2019robust} using a ball in the Wasserstein metric, which is a common choice for the uncertainty set in the distributionally robust optimization literature \cite{kuhn2019wasserstein, blanchet2019data, sinha2017certifying, lee2018minimax, gao2023distributionally, mohajerin2018data, gao2022wasserstein, shafieezadeh2015distributionally}. Relative to previous results, we focus on convex penalty functions (and not only norms), and also we explicitly identify a worst-case measure $\P^*\in B_\delta^{r,\rho,\sigma}(\P_n)$ for \eqref{eq:min2}, instead of relying on duality arguments. In this sense, our proof can be seen as a natural extension of \cite[Theorem 1]{bertsimas2018characterization}. 

DRO representations similar to \eqref{eq:min2} are known to be useful in many situations, for example, those where the trained procedure will be evaluated on test data from a distribution $\widetilde{\mathbb{P}}$ that is close to that of the training data, $\mathbb{P}$, but may be different \cite{ben2010theory}, when there are covariate shifts \cite{shimodaira2000improving, quinonero2008dataset, sugiyama2007covariate, sugiyama2005input, agarwal2011linear, wen2014robust, reddi2015doubly, chen2016robust}, or when one experiences adversarial attacks \cite{kurakin2016adversarial, goodfellow2014explaining}. As discussed in the seminal work of \cite{bertsimas2018characterization}, DRO representations similar to \eqref{eq:min2} \emph{``offer a different perspective on regularization methods by identifying which adversarial perturbations the model is protected against''}. This \emph{a fortiori} means that in any helpful representation similar to \eqref{eq:min2}, the set of adversarial distributions for which a regularization method protects against must depend on the regularizer itself. Thus, it should not be surprising that the max-sliced Wasserstein metric introduced in this paper depends on $\rho$. And in fact, previous uses of the Wasserstein metric for DRO representations of the $\sqrt{\text{LASSO}}$ and related estimators also depend on $\rho$; c.f.\ Proposition 2 in \cite{blanchet2019robust}.

Starting from \cite{dereich2013constructive, fournier2015rate}, the question of establishing finite sample bounds on the Wasserstein metric and its variants has seen a spike in research activity over the last years: an incomplete list is \cite{boissard2014mean, singh2018minimax, weed2019sharp, niles2019estimation, lei2020convergence, chizat2020faster}; see also the references therein. When $d>2r$, tight rates for $\mathcal{W}_r(\P_n, \P)$ are of the order $n^{-1/(rd)}$ , i.e.\ they suffer from the curse of dimensionality. As our results show, this is not the case for the $\rho$-MSW distance. 
The faster rates of convergence for the max-sliced Wasserstein metric were first observed in \cite{niles2019estimation} for subgaussian probability measures and in \cite{lin2021projection} under a projective Poincar\'e/Bernstein inequality. More recently, \cite{bartl2022structure} have obtained sharp rates for $r=2$ and isotropic distributions. Our rates are of the same order, up to logarithmic factors, and simultaneously hold for all $r\ge 1$ and all distributions with finite higher-order moments. Lastly, let us mention that most of the papers cited above only give explicit \textit{rates}, while the \textit{constants} are often non-explicit and large, cf.\ \cite{fournier2022convergence}. A notable exception is the recent work of \cite{kengo12022statistical} and \cite{kengo22022statistical}. In particular, using log-concavity,  \cite{kengo12022statistical} derives sharp rates for the max-sliced Wasserstein metric that explicitly state the dependence on the dimension of the data. In Section \ref{sec:recommendation} we further discuss how these results can be used to provide a recommendation for $\delta$ based on our DRO representation.

A large part of the theoretical literature studying penalized regressions as in \eqref{eq:penalized} has explicit recommendations for the choice of the penalization parameters. 
For the case of the LASSO estimator, \cite{chetverikov2021cross} present conditions such as the widely used cross-validation method has nearly optimal rates of convergence in prediction norms, and \cite{chernozhukov2023high} suggest utilizing a bootstrap approximation to estimate the penalization parameter. For the case of the $\sqrt{\text{LASSO}}$, \cite{belloni2011square,belloni2014pivotal} proposed a pivotal penalization parameter with asymptotic guarantees. Our work complements these previous results by recommending a  penalization parameter that explicitly controls the out-of-sample prediction error (for a finite sample and/or asymptotically).

\subsection{Outline} 
The rest of the paper is organized as follows. In Section \ref{sec:dro}, we present a detailed discussion of the equivalence of \eqref{eq:penalized} and \eqref{eq:min2}. In Sections \ref{sec:rates} and \ref{sec:asymptotic} we present rates for the MSW distance $\iW_{r,\rho,\sigma}$ between the true and empirical measure, both for $\P$ with compact support and for $\P$ satisfying $\Gamma<\infty$. Section \ref{sec:rates} gives a finite sample analysis, while Section \ref{sec:asymptotic} provides asymptotics.
In Section~\ref{sec:recommendation}, we present a recommendation for the selection of regularization parameter, $\delta_{n,r}$, that guarantees good out-of-sample prediction error. We also present a test statistic to rank the out-of-sample performance of two different linear estimators.  In Section \ref{sec:simulation}, we present a small-scale simulation to illustrate the performance of predictions based on the $\sqrt{\text{LASSO}}$ but using our recommended parameter $\delta$. All the proofs are collected in the Supplementary Material \cite{MRVW2024}.

\subsection{Notation} \label{sec:notation}

\quad \emph{Random Variables.} We use capital, bold letters---such as $\mathbf{Z}$ and $\widetilde{\mathbf{Z}}$---to denote Borel measurable random vectors in $\mathbb{R}^{d}$, and use $Z_j$ to denote the $j$-th coordinate of $\mathbf{Z}$. We denote the set of all Borel probability measures in $\mathbb{R}^{d}$ by  $\mathcal{P}(\mathbb{R}^{d})$ and let $\mathcal{P}_r(\mathbb{R}^{d}) \subset   \mathcal{P}(\mathbb{R}^{d})$ denote all Borel probability measures with finite $r$th moments. If the random vector $\mathbf{Z}$ has distribution or law $\mathbb{P} \in \mathcal{P}(\mathbb{R}^{d})$, we write $\mathbf{Z} \sim \mathbb{P}$. The expectation of $\mathbf{Z}$ is denoted as $\mathbb{E}_{\mathbb{P}}[\mathbf{Z}]$. 

\quad \emph{Covariates and outcome variables.} We reserve $\mathbf{X}$ for the random column vector collecting the $d$ covariates available for prediction, and $Y$ for the scalar outcome variable. The realizations of covariates and outcomes are denoted as $\mathbf{x}$ and $y$, respectively. In a slight abuse of notation, we sometimes write $(\mathbf{X},Y)$ to denote a random vector in $\mathbb{R}^{d+1}$ (instead of $(\mathbf{X}^{\top},Y)^{\top}$).  

\emph{Couplings.} For two probability measures $\mathbb{Q}$ and $\mathbb{P}$, we define a \emph{coupling} of $\mathbb{Q}$ and $\mathbb{P}$ as any element of $\mathcal{P}(\mathbb{R}^{d} \times \mathbb{R}^{d})$ that preserves the marginals over $\mathbb{R}^{d}$. We denote the collection of all such couplings as $\Pi(\mathbb{Q},\mathbb{P})$. By definition, if $(\widetilde{\mathbf{Z}}, \mathbf{Z})$ is an $\mathbb{R}^{d} \times \mathbb{R}^{d}$-valued random vector with distribution $\pi \in \Pi(\mathbb{Q},\mathbb{P})$, then $\widetilde{\mathbf{Z}} \sim \mathbb{Q}$ and $\mathbf{Z} \sim \mathbb{P}$.

\emph{Penalty functions.} 
For a function $\rho:\mathbb{R}^{d} \rightarrow \R$ we write 
\begin{equation*}
    \rho^* \left(\boldsymbol{\beta} \right) := \sup_{\mathbf{x} \in \mathbb{R}^d} \left\{\boldsymbol{\beta}^\top \mathbf{x} - \rho(\mathbf{x}) \right\},
\end{equation*}
for its conjugate (see \cite{rockafellar2015convex}).
If $\rho$ is convex, a vector $\boldsymbol{\beta}^*$ is said to be a subgradient of $\rho$ at a point $\boldsymbol{\beta}$ if:
\[ \rho(\mathbf{x}) \geq \rho(\boldsymbol{\beta}) + {\boldsymbol{\beta}^*}^{\top}\left(\mathbf{x}-\boldsymbol{\beta} \right), \: \, \qquad \forall \, \mathbf{x}\in \R^d.  \] 
The set of all subgradients of $\rho$ at $\boldsymbol{\beta}$ is called the subdifferential of $\rho$ at $\boldsymbol{\beta}$ and is denoted by  $\partial \rho (\boldsymbol{\beta})$, (\cite{rockafellar2015convex}; pp.\ 214-215). 

Lastly, let us mention two important facts that will be relevant in Section \ref{sec:examples}. If $\rho$ is differentiable, then its subdifferential $\partial \rho (\boldsymbol{\beta})$ is a singleton that contains the gradient of $\rho$ at $\boldsymbol{\beta}$;  see, for example, (\cite{rockafellar2015convex}; Theorem 25.1). If $\rho$ is a norm in $\mathbb{R}^{d}$, then $\rho^*$ is only equal to zero or infinity; see (\cite{Boyd:2004}; p.\ 93).

\section{Reformulation as a DRO problem}\label{sec:dro}

For any $r,\sigma \in [1,\infty)$, and $\rho\colon \mathbb{R}^{d} \rightarrow [0,+\infty)$ define the collection of distributions
\begin{equation}
\begin{split}
\label{eq:Q_ball}
B_\delta^{r,\rho,\sigma}(\P)&:= \left\{ \Q \in \mathcal{P}_r (\mathbb{R}^{d+1}):  \iW_{r,\rho,\sigma} (\Q,\P)\le \delta \right\}\\
&= \Big \{ \Q \in \mathcal{P}_r (\mathbb{R}^{d+1}) : \:  \forall \, \boldsymbol{\gamma} \in \mathbb{R}^{d},  \quad \exists \textrm{ a coupling } \pi(\boldsymbol{\gamma}) \in \Pi(\P, \Q)  \\  
&\qquad  \textrm{ for which } \: \mathbb{E}_{\pi(\boldsymbol \gamma)}\big[|(\tilde{Y}-Y)+ ( \mathbf{X}-\widetilde{\mathbf{X}})^{\top} \boldsymbol{\gamma} |^r \big]   \leq \delta^r(\sigma+\rho(\boldsymbol{\gamma}))^r, \\
& \qquad \textrm{ where } \big((\mathbf{X}, Y),\,(\widetilde{\mathbf{X}}, \widetilde{Y})\big) \sim \pi(\boldsymbol \gamma) \Big\}.
\end{split}
\end{equation}

As explained in the Introduction, a distribution $\mathbb{Q}$ belongs to the ball in  \eqref{eq:Q_ball} if and only if for any $\boldsymbol \gamma$ there exists a coupling of $\mathbb{P}$ and $\mathbb{Q}$ that makes the $r$-th norm of their prediction errors small, relative to $\rho(\boldsymbol \gamma)$. We remark that the infimum in the definition of $\iW_{r,\rho,\sigma}(\Q,\P)$ given in \eqref{eq:Wass} is attained for fixed $\boldsymbol{\gamma}$.\footnote{Indeed, note that the function $((\mathbf{x},y),(\tilde{ \mathbf{x}},\tilde y))\mapsto |(\tilde y-y)+(\tilde {\mathbf{x}}- \mathbf{x})^\top \boldsymbol{\gamma}|^r$ is continuous and non-negative. The result then follows from \cite[Theorem 4.1]{villani2008optimal}.} Furthermore, for any norm $\rho$, the supremum over $\boldsymbol{\gamma}$ in equation \eqref{eq:Wass} is also attained. For notational simplicity we will suppress the dependence of the ball $B_\delta^{r,\rho,\sigma}(\P)$ on $r,\rho,\sigma$ and write $B_\delta (\P)$ instead.

The main result of this section establishes a formal connection 
between the solutions to the problems in \eqref{eq:penalized} and \eqref{eq:min2}.

\begin{thm} 
\label{thm:penalized}
Fix $1 \leq r < \infty$ and $1 \leq \sigma < \infty$. Let $\rho\colon \mathbb{R}^{d} \rightarrow [0,+\infty)$ be a convex penalty function. Suppose that, for any $\boldsymbol{\beta} \in \mathbb{R}^{d}$, there exists a subgradient $\boldsymbol{\beta}^* \in \partial \rho (\boldsymbol{\beta})$ such that 
\begin{equation} \label{equation:condition}
    \left| \boldsymbol{\gamma}^{\top} \left( {\boldsymbol{\beta}^*} - \frac{\boldsymbol{\beta}}{\boldsymbol{\beta}^{\top} \boldsymbol{\beta}} \, \rho^*\big(\boldsymbol{\beta}^*\big) \right) \right | \leq \rho(\boldsymbol{\gamma}), \qquad \forall \: \boldsymbol{\gamma} \in \mathbb{R}^{d}.
\end{equation}
Then, for any  $\delta \geq 0$ and any $\boldsymbol{\beta} \in \mathbb{R}^{d}$ we have
\begin{align}\label{eq:duality}
\sup_{\widetilde{\mathbb{P}} \in B_{\delta}(\mathbb{P})} \mathbb{E}_{\widetilde{\mathbb{P}}} \left[ \left|Y-\mathbf{X}^{\top} \boldsymbol{\beta} \right|^r \right] = \left(\sqrt[r]{\mathbb{E}_{\mathbb{P}}\left[ \left|Y-\mathbf{X}^{\top} \boldsymbol{\beta}\right|^r \right]} + \, \delta \left(\sigma+\rho(\boldsymbol{\beta})\right) \right)^r.
\end{align}
\end{thm}

Theorem \ref{thm:penalized} shows that the worst-case, out-of-sample performance of any linear predictor over the collection of distributions $B_{\delta}(\mathbb{P})$ equals the $r$-th power of the objective function in \eqref{eq:penalized}. The result in \eqref{eq:duality} thus implies 
\begin{equation}\label{eq:regression}
\text{arg inf}_{\boldsymbol\beta\in \mathbb{R}^d} \Big[ \sup_{\widetilde{\mathbb{P}} \in B_{\delta}(\mathbb{P})} \mathbb{E}_{\widetilde{\mathbb{P}}} \big[ \big|Y-\mathbf{X}^{\top} \boldsymbol{\beta}
\big|^r \big] \Big] = \text{arg inf} _{\boldsymbol\beta\in \mathbb{R}^d} \sqrt[r]{\mathbb{E}_{\mathbb{P}} \left[ \left|Y-\mathbf{X}^{\top} \boldsymbol{\beta} \right|^r \right] }+ \delta\rho(\boldsymbol{\beta}).
\end{equation}

Our interpretation of equation \eqref{eq:regression} is that the $\sqrt{\text{LASSO}}$ and related estimators in \eqref{eq:penalized} have good out-of-sample performance for any \emph{testing} distribution, $\widetilde{\mathbb{P}}$, that is not far (in terms of the $\rho$-MSW metric) from the baseline \emph{training} distribution, $\mathbb{P}$. This result is independent of how the regularization parameter $\delta$ is selected and generalizes the connection between regularization and generalization performance first established in \cite{bertsimas2018characterization}.  

We briefly sketch the proof of Theorem~\ref{thm:penalized} here and refer to Section~\ref{sec:THMpenalizedproof} for details. It proceeds in two steps:
\begin{itemize}
\item[\textbf{Step 1.}] We use the triangle inequality and the definition of the $\rho$-MSW metric to show that
\begin{equation}
\label{eq:step1}
\mathbb{E}_{\widetilde{\mathbb{P}}} \left[ \left|Y-\mathbf{X}^{\top}\boldsymbol{\beta} \right|^r \right] \leq \left(  \sqrt[r]{\mathbb{E}_{\mathbb{P}} \left[ \left|Y-\mathbf{X}^{\top} \boldsymbol{\beta}\right|^r\right]}  + \delta  \left(\sigma+\rho(\boldsymbol{\beta})\right) \right)^r,
\end{equation}
holds for any $\boldsymbol{\beta} \in \mathbb{R}^d$ and any $\widetilde{\mathbb{P}} \in B_{\delta}(\mathbb{P})$.

\item[\textbf{Step 2.}] We show that for any $\boldsymbol{\beta}\in \text{dom}(\rho)$, the upper bound given in \textbf{Step 1} is tight. That is, we explicitly construct, for each $\boldsymbol \beta \in \text{dom}(\rho)$ a distribution $\mathbb{P}_{\boldsymbol{\beta}}^* \in B_{\delta}(\mathbb{P})$, for which the bound holds exactly. The worst-case distribution is presented in Corollary \ref{cor:1} below. 
 
\end{itemize}

\begin{corollary}\label{cor:1}
For each $\boldsymbol{\beta} \in \R^d$ the supremum in \eqref{eq:duality} is attained for the distribution $\mathbb{P}_{\boldsymbol{\beta}}^*$ corresponding to the random vector $(\widetilde{\mathbf{X}},\widetilde{Y})$ defined as
\begin{equation*} 
    \widetilde{\mathbf{X}} = \mathbf{X} - e \left( \boldsymbol{\beta}^* - \frac{\boldsymbol{\beta}}{\boldsymbol{\beta}^{\top} \boldsymbol{\beta}} \, \rho^*\left(\boldsymbol{\beta}^*\right) \right),  \qquad \widetilde{Y} = Y+\sigma e,
\end{equation*}
where  
\[ e := \frac{\delta \left(Y-\mathbf{X}^{\top}\boldsymbol{\beta} \right)}{ \sqrt[r]{\mathbb{E}_{\mathbb{P}} \left[ \left|Y-\mathbf{X}^{\top}\boldsymbol{\beta} \right|^r \right]} }, \qquad \left( \mathbf{X},Y\right) \sim \mathbb{P}.  \]
\end{corollary}

One aspect of Corollary \ref{cor:1} that is worth emphasizing is that the testing distribution that attains the worst out-of-sample performance is an additive perturbation of the baseline training distribution. The perturbation has a low-dimensional structure where a one-dimensional error, $e$, which is proportional to prediction error,  $Y-\mathbf{X}^{\top} \boldsymbol{\beta} $, is added to $\mathbf{X}$ using loadings that depend on the subgradient of $\rho$ at $\boldsymbol \beta$ and also on the conjugate of $\rho$.  

It is easy to see that the minimizer in \eqref{eq:regression} is attained. Denote this minimizer by $\boldsymbol{\beta}(\mathbb{P})$. Then $\boldsymbol{\beta}(\P)$ is also a minimizer of $\mathbb{E}_{\mathbb{P}_{\boldsymbol{\beta}}^*} \left[ \left|Y-\mathbf{X}^{\top}\boldsymbol{\beta} \right|^r \right]$.  Indeed, for any $\boldsymbol{\beta} \in \R^d$ we have
\begin{eqnarray*}
\mathbb{E}_{\mathbb{P}_{\boldsymbol{\beta}}^*} \left[ \left|Y-\mathbf{X}^{\top}\boldsymbol{\beta} \right|^r \right] & = &  \sup_{\widetilde{\mathbb{P}} \in B_{\delta}(\mathbb{P})} \mathbb{E}_{\widetilde{\mathbb{P}}} \left[ \left|Y-\mathbf{X}^{\top} \boldsymbol{\beta} \right|^r \right] \\
& \geq &  \inf_{\boldsymbol\beta\in \mathbb{R}^d} \sup_{\widetilde{\mathbb{P}} \in B_{\delta}(\mathbb{P})} \mathbb{E}_{\widetilde{\mathbb{P}}} \big[ \left|Y-\mathbf{X}^{\top} \boldsymbol{\beta}
\right|^r \big] \\
& = & \sup_{\widetilde{\mathbb{P}} \in B_{\delta}(\mathbb{P})} \mathbb{E}_{\widetilde{\mathbb{P}}} \big[ \left|Y-\mathbf{X}^{\top} \boldsymbol{\beta} (\mathbb{P})
\right|^r \big] \\
& \geq &  \mathbb{E}_{\mathbb{P}_{\boldsymbol{\beta}}^*} \left[ \left|Y-\mathbf{X}^{\top}\boldsymbol{\beta}(\mathbb{P}) \right|^r \right].
\end{eqnarray*}
In particular, for any linear predictor with slope $\boldsymbol{\beta}$, it is always possible to find a perturbation of $\mathbb{P}$ for which a predictor based on \eqref{eq:regression} performs better.

\begin{rem}[On condition \eqref{equation:condition}] If $\rho$ is a norm, then the condition in \eqref{equation:condition} is automatically satisfied; i.e., there exists a $\boldsymbol{\beta}^* \in \partial \rho(\boldsymbol{\beta})$ such that \eqref{equation:condition}  is true. Thus, the conclusion of Theorem~\ref{thm:penalized} holds for all $\rho(\cdot)=\|\cdot\|$  that are norms. Indeed,  recalling that the dual norm of $\|\cdot\|$ is given by
\begin{align*}
\|\mathbf{x}\|_*:= \sup_{\mathbf{y}: \|\mathbf{y}\|=1} \mathbf{x}^\top \mathbf{y},  
\end{align*} 
in that case \cite[Example 3.26]{Boyd:2004} states that $\rho^*(\mathbf{x})=\infty \mathds{1}_{\{\|\mathbf{x}\|_*>1\}}$. Recall further that $\boldsymbol{\beta}^* \in \partial \rho(\boldsymbol{\beta})$ if and only if
\begin{align}\label{eq:beta^*}
(\boldsymbol{\beta}^*)^{\top}\boldsymbol{\beta} - \rho^*(\boldsymbol{\beta}^*) = \rho(\boldsymbol{\beta}).    
\end{align}
Both facts together imply that $\rho^*(\boldsymbol{\beta}^*)=0$; thus, $\|\boldsymbol{\beta}^*\|_*\le 1$ for all $\boldsymbol{\beta}^* \in \partial \rho(\boldsymbol{\beta}).$ Hence in \eqref{equation:condition}, as claimed, we have 
\begin{equation} 
    \bigg| \boldsymbol{\gamma}^{\top} \left( {\boldsymbol{\beta}^*} - \frac{\boldsymbol{\beta}}{\boldsymbol{\beta}^{\top} \boldsymbol{\beta}} \, \rho^*(\boldsymbol{\beta}^*) \right) \bigg | = \big|\boldsymbol{\gamma}^{\top} \boldsymbol{\beta}^* \big| \le \| \boldsymbol{\gamma}\| \, \| \boldsymbol{\beta}^*\|_* \leq \|\boldsymbol{\gamma}\|, \quad \forall \: \boldsymbol{\gamma} \in \mathbb{R}^d.
\end{equation}
\label{rem:norms}
\end{rem}

On the other hand, the following example shows that condition \eqref{equation:condition} is not only satisfied by norms:

\begin{ex}[Condition \eqref{equation:condition} for a function $\rho$ that is not a norm]
Fix any compact set $K\subseteq \R^d$ such that $-K=K$ and consider
\begin{align*}
\rho(\boldsymbol{\beta})= \sup_{\mathbf{y}\in K} \boldsymbol{\beta}^\top \mathbf{y}.
\end{align*}
Then $\rho$ is convex (as a supremum of linear functions), finite (as $K$ is compact), non-negative (as $K=-K$), symmetric $\rho(\boldsymbol{\beta})=\rho(-\boldsymbol{\beta})$ (as $K=-K$), and homogeneous $\rho(\lambda\boldsymbol{\beta})=\lambda \rho(\boldsymbol{\beta})$. Thus, 
\begin{align*}
\rho^*(\boldsymbol{\beta}^*)= \sup_{\boldsymbol{\gamma} \in \R^d} \left({\boldsymbol{\beta}^*}^\top \boldsymbol{\gamma} -\rho(\boldsymbol{\gamma})\right)=\begin{cases}
\infty & \text{if } \exists \, \boldsymbol{\gamma}\in \R^d \text{ s.t. } {{\boldsymbol{\beta}^*}}^\top \boldsymbol{\gamma} -\rho(\boldsymbol{\gamma})>0,\\
0 & \text{if } {\boldsymbol{\beta}^*}^\top \boldsymbol{\gamma} -\rho(\boldsymbol{\gamma})\le 0 \text{ for all } \boldsymbol{\gamma} \in \R^d.
\end{cases}
\end{align*}
By \eqref{eq:beta^*} we 
conclude that $\rho^*(\boldsymbol{\beta}^*)=0$ for all $\boldsymbol{\beta}\in \R^d$; therefore,
${\boldsymbol{\beta}^*}^\top \boldsymbol{\gamma} \le \rho(\boldsymbol{\gamma})$ for all $\boldsymbol{\gamma} \in \R^d$. By symmetry of $\rho$, we also have that 
$|{\boldsymbol{\beta}^*}^\top \boldsymbol{\gamma}| \le \rho(\boldsymbol{\gamma}).$ It follows that
\begin{equation*} 
    \left| \boldsymbol{\gamma}^{\top} \left( {\boldsymbol{\beta}^*} - \frac{\boldsymbol{\beta}}{\boldsymbol{\beta}^{\top} \boldsymbol{\beta}} \, \rho^*(\boldsymbol{\beta}^*) \right) \right |  \leq \rho(\boldsymbol\gamma), \quad \forall \: \boldsymbol{\gamma} \in \mathbb{R}^d.
\end{equation*}
\end{ex}

\begin{rem} \label{rem:Wass_ball}
Take any norm $\|\cdot\|$ on $\R^{d+1}$ satisfying $\|(0,\dots,0,1)\|=1$ and recall that its dual norm is given by
\begin{align}\label{eq:dual}
\|\mathbf{x}\|_*:= \sup_{\mathbf{y}: \|\mathbf{y}\|=1} \mathbf{x}^\top \mathbf{y}.    
\end{align}
Assume that $\E_{\P}\left[\left\|(\mathbf{X},Y)\right\|_*^r\right]<\infty$ and consider a Wasserstein ball $\mathcal{B}^{\mathcal{W}}_{ \delta}(\mathbb{P})$ with cost $\|\cdot\|_*$, defined as
\begin{equation}
\mathcal{B}^{\mathcal{W}}_{\delta}(\mathbb{P}) = \left\{ \widetilde{\mathbb{P}} \in \mathcal{P}_r (\mathbb{R}^{d+1}): \, \: \mathcal{W}_{r}(\mathbb{P}, \widetilde{\mathbb{P}}) \leq \delta \right\},
\label{eq:Wass_ball}
\end{equation}
where
\[\mathcal{W}_{r}(\mathbb{P}, \widetilde{\mathbb{P}}) = \inf_{\substack{\pi \, \in \, \Pi(\mathbb{P}, \widetilde{\mathbb{P}}): \\ ((\mathbf{X}, Y), \, (\widetilde{\mathbf{X}}, \widetilde{Y})) \sim \pi}}  \sqrt[r]{\mathbb{E}_{\pi}\left[\norm{(\mathbf{X},Y) - (\widetilde{\mathbf{X}},\tilde{Y})}^r_* \right]}. \]

We show that for $\rho(\cdot) = \norm{\cdot}$, the ball defined in \eqref{eq:Q_ball} contains the ball in \eqref{eq:Wass_ball}, i.e.\ $ \mathcal{B}^{\mathcal{W}}_{\delta}(\mathbb{P}) \subseteq B_{\delta}(\mathbb{P})$. For this, we note that by \eqref{eq:dual} we have
\begin{align*}
\mathbb{E}_{\pi}\left[ \left|\left(\tilde{Y}-Y\right)+ \left(  \mathbf{X}-\widetilde{\mathbf{X}} \right)^{\top} \boldsymbol{\gamma} \right| ^r\right] &\leq \norm{\left(\boldsymbol{\gamma},-1 \right)}^r \,\, \mathbb{E}_{\pi}\left[  \norm{ (\mathbf{X},Y)- (\widetilde{\mathbf{X}},\tilde{Y})}^r_* \right]\\
&\le  \left(1+\norm{\boldsymbol{\gamma}}\right)^r \,\, \mathbb{E}_{\pi}\left[  \norm{ (\mathbf{X},Y)- (\widetilde{\mathbf{X}},\tilde{Y})}^r_* \right].
\end{align*}
As $\sigma \ge 1$, we conclude
\begin{align*}
&\sup_{\boldsymbol{\gamma} \in  \mathbb{R}^d} \inf_{\pi \, \in \, \Pi(\mathbb{P}, \widetilde{\mathbb{P}})} \frac{1}{\sigma+\norm{\boldsymbol{\gamma}}} \sqrt[r]{\mathbb{E}_{\pi}\left[ \left| (\tilde Y-Y )+ ( \mathbf{X}-\widetilde{\mathbf{X}} )^{\top} \boldsymbol{\gamma} \right|^r  \right]} \\
&\qquad \hspace{15mm} \leq  \inf_{\pi \in \Pi(\mathbb{P}, \widetilde{\mathbb{P}})}  \sqrt[r]{ \mathbb{E}_{\pi}\left[  \| (\mathbf{X},Y )- (\widetilde{\mathbf{X}},\tilde{Y})\|^r_* \right]}.
\end{align*}
The above can be applied in particular to $\|\cdot\|=\|\cdot\|_p$ and $\|\cdot\|_*=\|\cdot\|_q$, where $1/p+1/q=1.$
\end{rem}

We note that the conditions used for the derivations in Remark 2 are sufficient, but not necessary. To make this point, consider the case in which $r=2$ and $\rho(\boldsymbol{\beta}) = \|\boldsymbol{\beta} \|_1$. The Wasserstein distance, $\mathcal{W}_2$, between $\mathbb{P}$ and $\Tilde{\mathbb{P}}$ is defined by
$$ \mathcal{W}_2(\mathbb{P},\Tilde{\mathbb{P}}) =  \inf_{\pi \in \Pi(\mathbb{P},\Tilde{\mathbb{P}})} \mathbb{E}_{\pi} [ \| (\mathbf{X},Y) - (\Tilde{\mathbf{X}},\Tilde{Y}) \|_2^2  ]^{1/2}~,$$
where $\| \cdot \|_2$ is the Euclidean distance. For any coupling $\pi$, the Cauchy-Schwarz inequality implies
\begin{eqnarray*}
\mathbb{E}_{\pi }\left[ \left|(\tilde{Y}-Y)+ (  \mathbf{X}-\widetilde{\mathbf{X}} )^{\top} \boldsymbol{\beta} \right| ^2\right]^{1/2}  &\leq& \|\left(\boldsymbol{\beta},-1 \right)\|_2 \,\, \mathbb{E}_{\pi}\left[  \| (\mathbf{X},Y)- (\widetilde{\mathbf{X}},\tilde{Y})\|_2^2 \right]^{1/2},
\end{eqnarray*}
where the right-hand side of the previous inequality is less than 
\[(1 + \|\boldsymbol{\beta}\|_2)\mathbb{E}_{\pi}\left[  \| (\mathbf{X},Y)- (\widetilde{\mathbf{X}},\tilde{Y})\|_2^2 \right]^{1/2},\]
due to the triangle inequality. Since we have assumed that $\sigma \ge 1$, we have $1 + \|\boldsymbol{\beta}\|_2 \le \sigma + \|\boldsymbol{\beta}\|_1$. Consequently:
$$ \widehat{\mathcal{W}}_{2, \rho, \sigma}(\mathbb{P},\Tilde{\mathbb{P}}) \le \mathcal{W}_2(\mathbb{P},\Tilde{\mathbb{P}}).$$
Thus, balls based on the $\rho$-MSW metric can be larger than balls based on the $d$-dimensional Wasserstein metric, even when the latter does not use a cost function based on the dual norm of $\rho$. Our previous derivations also hold, \emph{mutatis mutandi}, for penalty functions $\rho(\boldsymbol{\beta}) = \|\boldsymbol{\beta}\|_p$ whenever $p \in [1,2]$. Remark 2 focuses on the case in which i) $\rho$ is a norm, and ii) the cost function used in the $d$-dimensional Wasserstein metric is associated to the dual norm of $\rho$ to make our results directly comparable to those in Proposition 2 in \cite{blanchet2019robust}.

\subsection{Examples of distributions in the $\rho$-MSW ball}\label{sec:examples}

In this subsection we analyze the types of testing distributions that are contained in the ball defined in \eqref{eq:Q_ball}. We do this by considering different estimators that take the form \eqref{eq:penalized}.  

\subsubsection{$\sqrt{\text{LASSO}}$}
Let us take $r=2$ and
\[ \rho(\boldsymbol{\beta}) =  \|\boldsymbol{\beta}\|_1 =  \sum_{j=1}^{d} |\beta_j|.  \]
Under this choice of penalty function, the regression problem \eqref{eq:regression} is the objective function of the $\sqrt{\mathrm{LASSO}}$ of \cite{belloni2011square}, also studied in \cite{belloni2014pivotal}. These papers have shown that the $\sqrt{\mathrm{LASSO}}$ estimator achieves the near-oracle rates of convergence in sparse, high-dimensional regression models over data distributions that extend significantly beyond normality. 

Clearly $\rho$ is a norm; in particular, it is nonnegative and convex. Thus, Condition \eqref{equation:condition} of Theorem~\ref{thm:penalized} is satisfied, cf.\ Remark~\ref{rem:norms}.

One set of distributions that belongs to a neighborhood of size $\delta$ based on the $\rho$-MSW metric is:  
\begin{equation}
\begin{split}
\label{eq:BLASSO}
B^{\sqrt{\mathrm{LASSO}}}_\delta(\mathbb{P}) := &\left\{ \mathbb{Q} \in \mathcal{P}_2 (\mathbb{R}^{d+1}) \: | \:  \exists \textrm{ a coupling } \pi \in \Pi(\mathbb{Q}, \mathbb{P}) \textrm{ for which:} \right. \\
& \mathbb{E}_{\pi}\left[\left| \tilde{X}_j - X_j \right|^2 \right ] \leq \delta^2, \: \forall \: j=1,\ldots,d, \text{ and } \E_\pi\left[ \left|\widetilde{Y}- Y \right|^2\right] \le (\delta \sigma)^2,\\
& \textrm{where } ((\mathbf{X}, Y), (\widetilde{\mathbf{X}}, \widetilde{Y})) \sim \pi  \}.
\end{split}
\end{equation}
This set of distributions contains perturbations of covariates and outcomes that are small in 2-norm. We verify that $B^{\sqrt{\mathrm{LASSO}}}_\delta(\mathbb{P}) \subseteq B_{\delta}(\mathbb{P})$, where $B_{\delta}(\mathbb{P})$ is the set of balls used in Theorem~\ref{thm:penalized} and defined in \eqref{eq:Q_ball}.  

To see this, notice $\mathbb{E}_{\pi}[| \tilde{X}_j - X_j |^2 ] \leq \delta^2$ for all $j=1,\ldots,d$ implies condition \eqref{eq:Q_ball}, i.e.\
$$\mathbb{E}_{\pi}\left[ \left |(\tilde{Y}-Y)+ (\mathbf{X}-\widetilde{\mathbf{X}})^{\top} \boldsymbol{\gamma} \right |^2 \right]  \leq \delta^2 \left(\sigma+\rho(\boldsymbol{\gamma})\right)^2.$$
Indeed, the triangle inequality implies that for any  $\boldsymbol{\gamma} \in  \mathbb{R}^d$ and any coupling $\pi \in \Pi (\mathbb{P},\mathbb{Q})$ consistent with \eqref{eq:BLASSO}, we have
\begin{align*}
\sqrt{\mathbb{E}_{\pi}\left[ \left|(\tilde{Y}-Y)+ ( \mathbf{X}-\widetilde{\mathbf{X}})^{\top} \boldsymbol{\gamma} \right |^2 \right]} &\leq \sqrt{\E_\pi\left[ \left|\tilde{Y}-Y \right|^2 \right]} + \sum_{j=1}^{d} \left|\gamma_j \right| \sqrt{\mathbb{E}_{\pi}\left[\left(\tilde{X}_j-X_j \right)^2 \right]}  \\
&\leq \delta\sigma+ \delta \sum_{j=1}^{d} |\gamma_j|  =\delta (\sigma+\rho(\boldsymbol{\gamma})).
\end{align*}
Consequently, $B^{\sqrt{\mathrm{LASSO}}}_\delta(\mathbb{P}) \subseteq B_{\delta}(\mathbb{P})$. We note that the other direction, namely, $B_{\delta}(\mathbb{P}) \subseteq B^{\sqrt{\mathrm{LASSO}}}_\delta(\mathbb{P})$ does not hold in general.

It is worth mentioning that the set $B^{\sqrt{\mathrm{LASSO}}}_\delta(\mathbb{P})$ contains different versions of $(\mathbf{X},Y)$ measured with error. For example, any additive measurement error model of the form
$\tilde{X}_j = X_j + u_j$ and $\tilde{Y} = Y + v,$
where $\mathbb{E}[u_j^2] \leq \delta^2$ and $\mathbb{E}[v^2] \leq (\delta\sigma)^2$. Also,  $B^{\sqrt{\mathrm{LASSO}}}_\delta(\mathbb{P})$ contains multiplicative errors-in-variables models where
$\tilde{X}_j = X_j u_j,$ and $\tilde{Y} = Yv,$
with $u$'s independent of $(\mathbf{X},Y)$, having mean equal to one, $\mathbb{E}_{\mathbb{P}}[X_j^2] \, \mathbb{E}[(u_j-1)^2 \leq \delta^2$, and independent of $v$ having mean equal to one and $\mathbb{E}_{\mathbb{P}}[Y^2] \, \mathbb{E}[\left(v-1\right)^2] \leq (\delta\sigma)^2$.

It is well known that the conjugate of $\rho$ is
\[ \rho^*(\boldsymbol{\beta}) = 
\begin{cases}
  0   & \max \left\{ |\beta_1|, \ldots, |\beta_d| \right\} \leq 1,   \\
  \infty   & \textrm{otherwise}.
\end{cases}
\]
The argument is analogous to Remark~\ref{rem:norms}. Moreover, algebra shows
$\boldsymbol{\beta}^* =  \left(\textrm{sign}(\beta_1), \ldots, \textrm{sign}(\beta_d) \right)^{\top},$
is a subgradient of $\rho$ at $\boldsymbol{\beta}$. Using these facts, we can determine the worst-case distribution for each particular $\boldsymbol{\beta}$. Indeed, Corollary \ref{cor:1} states that:
$\tilde{\mathbf{X}} = \mathbf{X} -  e \left(\textrm{sign}(\beta_1), \ldots, \textrm{sign}(\beta_d) \right)^{\top},$ and $\tilde{Y} = Y+ \sigma e,$
where 
\[ e := \frac{\delta\left(Y-\mathbf{X}^{\top}\boldsymbol{\beta}\right)}{ \sqrt{\mathbb{E}_{\mathbb{P}} \left[ \left(Y-\mathbf{X}^{\top}\boldsymbol{\beta} \right)^2 \right]} }, \qquad ( \mathbf{X},Y) \sim \mathbb{P}. \]
The worst-case mean-squared error of $\sqrt{\mathrm{LASSO}}$ is attained at distributions where there is a (possibly correlated) measurement error that has a factor structure. Note that the worst-case distribution is an element of \eqref{eq:BLASSO}.

\subsubsection{Square-root SLOPE}

Now suppose again that $r=2$, but let
$$ \rho(\boldsymbol{\beta}) = \sum_{j=1}^d \lambda_j |\boldsymbol{\beta}|_{(j)},$$
where $\lambda_1\ge \dots \ge \lambda_d \ge 0$ and $|\boldsymbol{\beta}|_{(j)}$ are the decreasing order statistics of the absolute values of the coordinates of $\boldsymbol{\beta}$. 
Under this penalty function---which is nonnegative---the penalized regression problem in \eqref{eq:regression} is the objective function of the square-root SLOPE of \cite{stucky2017sharp}.

An equivalent definition for this penalty function is 
\begin{equation} \label{equation:aux-slope}
\rho(\boldsymbol{\beta}) = \max_{pm} \sum_{j=1}^d \lambda_{pm(j)} |\beta_j |,
\end{equation}
where we maximize over all permutations, $pm$, of the coordinates $\{1,\dots,d\}$. It follows that $\rho$ is a norm, so Condition \eqref{equation:condition} of Theorem~\ref{thm:penalized} is satisfied (see Remark~\ref{rem:norms}).

For a given $\boldsymbol{\beta} \in \mathbb{R}^d$, let $pm^*$ be a permutation that solves \eqref{equation:aux-slope}. Define $\boldsymbol{\beta}^*$ by $\beta^*_{j} =  \lambda_{pm^*(j)} \: \text{sign}(\beta_{j})$. Algebra shows that $\rho(\boldsymbol{\beta}) = {\boldsymbol{\beta}^*}^{\top}\boldsymbol{\beta}$
and $ {\boldsymbol{\beta}^*}^{\top} \boldsymbol{\gamma}\le \rho(\boldsymbol{\gamma}) ,$
for any $\boldsymbol{\gamma} \in \mathbb{R}^d$. It follows that $\rho(\boldsymbol{\gamma}) \ge \rho(\boldsymbol{\beta}) + {\boldsymbol{\beta}^*}^{\top} \boldsymbol{\gamma} - {\boldsymbol{\beta}^*}^{\top} \boldsymbol{\beta}$, which implies that $\boldsymbol{\beta}^*$ is a subgradient of $\rho$ at $\boldsymbol{\beta}$. Recall that $\rho^*(\boldsymbol{\beta}^*) = 0$; thus, \eqref{equation:condition} holds.

In this case, distributions belonging to balls of size $\delta$ based on the $\rho$-MSW metric are
\begin{align*}
B^{\textrm{SLOPE}}_\delta(\mathbb{P}):= &\{ \mathbb{Q} \in \mathcal{P}_2 (\mathbb{R}^{d}) \:: \: \exists \textrm{ a coupling } \pi \in \Pi(\mathbb{Q},\mathbb{P}) \textrm{ for which:} \  \\
& \quad \mathbb{E}_{\pi}\left[\left| \tilde{X}_{(j)} - X_{(j)} \right|^2 \right ] \leq (\delta\lambda_j)^2, \: \forall \: j=1,\ldots,d,\\
& \quad \text{and } \E_{\pi} \left[ \left|\widetilde{Y} -Y \right|^2 \right| \le (\delta\sigma)^2, \textrm{ where } ((\mathbf{X}, Y), (\widetilde{\mathbf{X}}, \widetilde{Y})) \sim \pi \},
\end{align*}
where the decreasing order statistic is induced by the vector $\Big(\mathbb{E}_{\pi}\Big[ \Big|\tilde{X}_j - X_j \Big|^2 \Big]\Big)_{j=1,\dots,d}$. As for the $\sqrt{\mathrm{LASSO}}$, we check that $B^{\textrm{SLOPE}}_\delta(\mathbb{P})\subseteq B_\delta(\P)$. The triangle inequality implies that for any coupling $\pi \in \Pi(\mathbb{P},\mathbb{Q})$:
\begin{align*}
     \sqrt{\mathbb{E}_{\pi}\left[\left|(\tilde{Y}-Y)+(\mathbf{X}-\tilde{\mathbf{X}})^{\top} \gamma\right|^2 \right]} &\leq \sqrt{\E_{\pi}\left[\left|\tilde{Y}-Y \right|^2\right]}+\sum_{j=1}^{d} \left|\gamma_j \right| \sqrt{\mathbb{E}_{\pi} \left[\left|\tilde{X}_j-X_j \right|^2 \right]} \\
     &=  \sqrt{\E_{\pi}\left[\left|\tilde{Y}-Y \right|^2\right]}+\sum_{j=1}^{d} \left|\gamma_{(j)} \right| \sqrt{\mathbb{E}_{\pi}\left[\left|\tilde{X}_{(j)}-X_{(j)} \right|^2\right]} \\
     &\le \delta \left(\sigma+\rho(\gamma)\right),
\end{align*}
 where the last equality follows by the definition of $B_\delta^{\textrm{SLOPE}}(\mathbb{P})$ and \eqref{equation:aux-slope}.

Finally, we report the worst-case distribution for each particular $\boldsymbol{\beta}$. Corollary \ref{cor:1} shows that
$\tilde{\mathbf{X}} = \mathbf{X} - e {\boldsymbol{\beta}^*}$ and $ \tilde{Y}=Y+\sigma e,$
where the $j$-coordinate of $\boldsymbol{\beta}^*$ is $\lambda_{pm^*(j)} \: \text{sign}(\beta_{j})$ and
\[ e := \frac{\delta \left(Y-\mathbf{X}^{\top}\boldsymbol{\beta}\right)}{ \sqrt{\mathbb{E}_{\mathbb{P}} \left[ \left|Y-\mathbf{X}^{\top}\boldsymbol{\beta} \right|^2 \right]} }, \quad (\mathbf{X},Y) \sim \mathbb{P}. \]
Note that the worst-case distribution is an element of $B^{\textrm{SLOPE}}_\delta(\mathbb{P})$.

\section{Finite sample guarantees for the $\rho$-MSW-distance}\label{sec:rates}

Throughout this section, we assume that the data $\{( \mathbf{X}_i,Y_i)\}_{i=1}^n$ consists of i.i.d.\ draws from a true distribution that we denote by $\P$. We denote the empirical distribution based on the available data by $\P_n$. 

This section provides explicit upper bounds on the radius $\delta$ of the ball $B_{\delta}(\mathbb{P}_n)$ defined in \eqref{eq:Q_ball}, to guarantee that the true (and unknown) distribution, $\P$, belongs to the ball $B_\delta(\P_n)$ with a pre-specified probability. Our derivations are valid for any finite sample, which means that they hold regardless of the dimension of the covariates, $d$, the sample size, $n$, and the true distribution $\mathbb{P}$. 

Recall from \eqref{eq:Wass} that
\begin{equation}
\label{eq:proj_wass}
 \iW_{r,\rho,\sigma}(\P, \tilde\P)=\sup_{\boldsymbol{\gamma} \in  \mathbb{R}^d} \inf_{\substack{\pi \, \in \, \Pi(\mathbb{P}, \widetilde{\mathbb{P}}): \\ ((\mathbf{X}, Y), \, (\widetilde{\mathbf{X}}, \widetilde{Y})) \sim \pi}} \frac{1}{\sigma+\rho(\boldsymbol{\gamma})} \sqrt[r]{\mathbb{E}_{\pi}\left[ \left| (\tilde Y-Y)+ ( \mathbf{X}-\widetilde{\mathbf{X}} )^{\top} \boldsymbol{\gamma} \right|^r \right]}.
\end{equation}
An important comment about the hyperparameter $\sigma$ is in turn. Suppose that $\rho$ is a norm and we tried to set $\sigma = 0$ in the definition of $\iW_{r,\rho,\sigma}$ above. It then follows from the definition, that any distribution $\Tilde{\mathbb{P}}$ for which $\iW_{r,\rho,0}(\mathbb{P},\Tilde{\mathbb{P}}) \leq \delta$ must have the same marginal distributions over the outcome variable; i.e.,  $\Tilde{Y} \sim Y$. As a result, whenever $Y$ is a random variable of the absolutely continuous type, any ball around $\mathbb{P}_n$ will \emph{fail to} contain the true distribution $\mathbb{P}$. Therefore, introducing a hyperparameter $\sigma > 0$ is important for our analysis.

Notice that we can rewrite the equation above in terms of the one-dimensional Wasserstein metric:
\begin{align}
\label{eq:proj_wass_rep}
 \iW_{r,\rho,\sigma}(\P, \tilde \P)=\sup_{\boldsymbol{\gamma} \in  \mathbb{R}^d}
\frac{1}{\sigma+\rho(\boldsymbol{\gamma})} \,
\mathcal{W}_r\left(\left[(\mathbf{X},Y)^{\top}\boldsymbol{\bar{\gamma}}\right]_*\mathbb{P},\,  \left[(\widetilde{\mathbf{X}},\tilde Y)^{\top}\boldsymbol{\bar{\gamma}}\right]_*\widetilde{\mathbb{P}}\right),
\end{align}
where $\boldsymbol{\bar{\gamma}}^\top = (\boldsymbol{\gamma}^\top,-1)$ and $f_* \mathbb{P}$ denotes the pushforward measure of $\mathbb{P}$ with respect to a map $f:\mathbb{R}^{d+1} \to \R$ and $\mathcal{W}_r$. The one-dimensional Wasserstein metric is simply defined as
\begin{align}
\label{eq:wass_def}
\mathcal{W}_r(\mathbb{Q}, \widetilde{\mathbb{Q}})=  \inf_{\substack{\pi \, \in  \, \Pi(\mathbb{Q}, \widetilde{\mathbb{Q}}): \\ (X, \tilde X) \sim \pi}}  \sqrt[r]{\mathbb{E}_{\pi}\left[\left|X-\tilde{X} \right|^r \right]}.
\end{align}

We focus on the case where  $\rho$ satisfies \eqref{eq:new}, i.e.\
$c_d \|(\boldsymbol{\gamma}, -1)\|\le \rho(\boldsymbol{\gamma})+1,$ for all $\boldsymbol{\gamma} \in \R^d,$
for some constant $c_d>0$ and an arbitrary norm $\|\cdot\|$ on $\mathbb{R}^{d+1}$. E.g.\ for $\gamma(\cdot)=\|\cdot\|_1$, \eqref{eq:new} is satisfied with $c_d=1$.

Using the definition in \eqref{eq:proj_wass_rep} we derive the following upper bound for $\iW_{r,\rho,\sigma}(\P, \tilde \P)$:
\begin{align}
\iW_{r,\rho,\sigma}(\P, \tilde \P)&=\sup_{\boldsymbol{\gamma} \in  \mathbb{R}^d}
\textcolor{black}{\frac{\|\boldsymbol{\bar{\gamma}}\|}{\sigma+\rho(\boldsymbol{\gamma})}\frac{1}{\|\boldsymbol{\bar{\gamma}}\|}}
\mathcal{W}_r\left(\left[(\mathbf{X},Y)^{\top}\boldsymbol{\bar{\gamma}}\right]_*\mathbb{P},\,  \left[(\widetilde{\mathbf{X}},\tilde Y)^{\top}\boldsymbol{\bar{\gamma}}\right]_*\widetilde{\mathbb{P}}\right) \notag \\
&\le  c_{\rho,d} \:  \left(  \sup_{ \widetilde{\boldsymbol{\gamma}}:  \|\widetilde{\boldsymbol{\gamma}}\| = 1} 
\mathcal{W}_r\left(\left[(\mathbf{X},Y)^\top\widetilde{\boldsymbol{\gamma}}\right]_*\mathbb{P},\,  \left[(\widetilde{\mathbf{X}},\tilde Y)^\top \widetilde{\boldsymbol{\gamma}}\right]_*\widetilde{\mathbb{P}}\right) \right) =: c_{\rho,d} \:  \W_r(\P, \tilde\P),\label{eq:bar_W}
\end{align}
where $\boldsymbol{\bar{\gamma}}^\top = (\boldsymbol{\gamma}^\top,-1)$ and $c_{\rho,d} : =  \max \left\{ 1/c_d,1 \right\}.$

The quantity $\W_r$ defined in \eqref{eq:bar_W} is known as the max-sliced Wasserstein (MSW) distance on $(\R^{d+1},\|\cdot\|)$. Moreover, it is a special case of the Projection Robust Wasserstein (PRW) distance, also called the Wasserstein Projection Pursuit (WPP), see \cite[Definition 1]{Paty2019SubspaceRW}. The work in \cite[Proposition 1]{Paty2019SubspaceRW} shows that $\W_r(\mathbb{P}, \widetilde{\mathbb{P}})$ is a metric (the proof is stated for the case $r=2$, but carries over line by line to arbitrary $r\ge 1$). Note that the connection between our metric $\iW_{r,\rho,\sigma}(\P, \tilde\P)$ and the more typical MSW metric used the fact that $\sigma \geq 1$. 

As stated in the Introduction, it is well known that, in the worst case, $\mathcal{W}_{r}(\P_n, \P)^r \sim n^{-1/(d+1)}.$
In what follows, we show that the MSW distance $\W_r$ does not have this limitation.
To show this, we first make a few notational simplifications.  
We write $\P_{\boldsymbol{\gamma}}$ and $F_{\boldsymbol{\gamma}}$, respectively, for the distribution and cdf of the scalar $(\mathbf{X} ,Y)^{\top} \boldsymbol{\gamma}$ under $\P$. Similarly, we write $\P_{\boldsymbol{\gamma}, n}$ and $F_{\boldsymbol{\gamma}, n}$, respectively, for the probability measure and cdf of $(\mathbf{X},Y)^{\top} \boldsymbol{\gamma}$ under $\P_n$. Note that, in particular, by \eqref{eq:bar_W} we have
$\W_r(\P, \tilde{\P}) =\sup_{ \|\boldsymbol{\gamma}\|=1 }\mathcal{W}_r (\P_{\boldsymbol{\gamma}}, \tilde{\P}_{\boldsymbol{\gamma}}).$

We now provide explicit upper bounds for $\W_r(\P, \P_n)$. By equation \eqref{eq:bar_W}, for any $\delta$ we have
\begin{equation} \label{eqn:MSW_to_rhoMSW}
\mathbb{P} \left( \iW_{r,\rho,\sigma}(\P,  \P_n) \leq c_{\rho,d} \cdot \delta \right) \geq \mathbb{P} \left( \W_r(\P, \P_n) \leq \delta \right).
\end{equation}
This means that probabilistic statements about $\W_r(\P, \P_n)$ translate immediately to the $\rho$-MSW metric. For simplicity in the exposition, we first cover compactly supported measures $\P$ in Section \ref{sec:compact} and then the general case in Section \ref{sec:general}.

\subsection{The compactly supported case}\label{sec:compact}

\begin{thm}\label{thm:1}
Let $\P$ have compact support. With probability at least $1-\alpha$,
\begin{align*}
 \W_r(\P_n, \P)^r &\le \frac{C}{\sqrt{n}},
\end{align*}
where 
\begin{equation} \label{eq:constant_thm2}
  C:=\Big( 180\sqrt{d+2}+\sqrt{2\log\Big(\frac{1}{\alpha}\Big)} \Big) \mathrm{diam}\left( \mathrm{supp}(\P)\right)^r ,  
\end{equation}
and $\mathrm{diam}\left( \mathrm{supp}(\P)\right) = \sup\{\|\mathbf{x}-\tilde{\mathbf{x}}\|_*: \,\mathbf{x},\tilde{\mathbf{x}}\in \mathrm{supp}(\P)\},$ is the diameter of the support of $\mathbb{P}$ measured with respect \textcolor{black}{to the dual norm}.
\end{thm}

\subsection{The general case} \label{sec:general}

We now consider a more general set-up where $\mathbb{P}$ is an arbitrary random variable that satisfies a mild moment condition; namely,
\begin{align}
\label{eq:gamma}
    \Gamma:=\E_{\P}\left[  \left\|(\mathbf{X},Y) \right\|_*^{s}\right]<\infty, \qquad \text{for some }s>2r.
\end{align}
Our result generalizes the work of \cite{niles2019estimation} and \cite{lin2021projection}, who provide rates for $\W_r(\P_n, \P)$ assuming certain transport or Poincar\'e inequalities:  we give similar rate statements with fully explicit constants under assumption \eqref{eq:gamma}, that is easy to verify in practice. 

Our main result in this section is the following:
\begin{thm} \label{thm:rates}
Assume $s>2r$ and $\Gamma <\infty.$ Then, with probability greater than $1-3\alpha$,
\begin{equation}
\label{eq:W_bound_q_2}
\W_r(\P_n,\P)^r \le \frac{C\,\log\left(2n+1 \right)^{r/s}}{\sqrt{n}} ,
\end{equation}
where 
\begin{align}\label{eq:constant_thm3}
C:= 
2^rr\Big(180\sqrt{d+2}+\sqrt{2\log\left(\frac{1}{\alpha}\right)}+\sqrt{\frac{\Gamma}{\alpha}} \frac{8}{s/2-r} \sqrt{\log\left(\frac{8}{\alpha}\right) + (d+2)}\Big).
\end{align}
\end{thm}

\section{Asymptotics for $\rho$-MSW-distance}\label{sec:asymptotic}

We now provide asymptotic upper bounds for the $\rho$-MSW distance between the true and empirical measure. For this it is sufficient to prove the corresponding bounds for $\W_r(\P, \P_n)$ as explained in 
\eqref{eq:bar_W} and \eqref{eqn:MSW_to_rhoMSW}. 
The following theorem provides a Donsker type result, i.e.\ asymptotic $\sqrt{n}$-rates without logarithmic factors, as well as an inequality for the expectation without an explicit constant. One can then obtain concentration results similarly to \cite[ Theorem 3.7, 3.8]{lin2021projection} if a Bernstein tail condition or Poincare inequality is satisfied. As before, we relegate the proofs of these results to the Supplementary Material \cite{MRVW2024}. We first consider probability measures $\P$ with compact support. 

\begin{thm}\label{thm:ass_bdd}
If $\P$ is compactly supported, then 
\begin{align*}
\limsup_{n\to \infty} \P\left( \sqrt{n} ~ \W_r\left(\P_{ n},\P\right)^r \ge x \right) \le  \P\Big( \sup_{t\in [0,1]} \left|B(t) \right|\ge \frac{x}{c }\Big),    
\end{align*}
where $c=\mathrm{diam}(\mathrm{supp}(\P))^r$ and $(B(t))_{t\in [0,1]}$ is a standard Brownian bridge.
\end{thm}

We now state the general result:

\begin{thm}\label{thm:ass}
 Assume $\Gamma=\E_{\P}\left[  \left\|(\mathbf{X},Y) \right\|_*^{s}\right]<\infty,$ for some $s>2r,$
and define 
\begin{align*}
\mathcal{H}^+ &:= \left\{ \left|t\right|^{s} \mathds{1}_{\{t \le \boldsymbol {x}^\top \boldsymbol{\gamma} \}}:\ (\boldsymbol\gamma,t)\in \R^{d+1}\times [0,\infty), \ \|\boldsymbol{\gamma}\|=1 \right\}, \\
\mathcal{H}^0 &:= \left\{ \mathds{1}_{\{\boldsymbol{x}^\top \boldsymbol\gamma \le t \}}:\ (\boldsymbol\gamma,t)\in \R^{d+1}\times \R, \ \|\boldsymbol{\gamma}\|=1 \right \},\\
\mathcal{H}^- &:= \left\{ \left|t\right|^{s} \mathds{1}_{\{t> \boldsymbol{x}^\top \boldsymbol{\gamma} \}} :\ (\boldsymbol\gamma,t)\in \R^{d+1}\times (-\infty,0), \ \|\boldsymbol{\gamma}\|=1 \right\}.
\end{align*}
Then there exists a constant $C :=C(r,s,d)$, such that for all $t\ge 0$,
\begin{align*}
\limsup_{n\to \infty} \P\left( \sqrt{n} ~ \W_r(\P_{ n},\P)^r \ge t \right) \le  \P\Big( \sup_{f \in \mathcal{H}^+\cup \mathcal{H}^0\cup \mathcal{H}^-} |G_f|\ge \frac{t}{C\sqrt{\Gamma} }\Big),
\end{align*}
where $(G_f)_{f\in \mathcal{H}^+\cup \mathcal{H}^0\cup \mathcal{H}^-}$ is a zero-mean Gaussian process with covariance 
\begin{align}\label{eq:covariance}
\E\left[G_{f_1} G_{f_2} \right] = \E_{\P}\left[f_1 f_2 \right] -\E_{\P} \left[f_1 \right]\E_{\P}\left[f_2 \right] \qquad \forall f_1, f_2 \in \mathcal{H}^+\cup \mathcal{H}^0\cup \mathcal{H}^-.
\end{align}
Furthermore, for all $n\in \mathbb{N}$, we have
$
\E_{\P} \left[ \sqrt{n} ~ \W_r(\P_{n},\P)^r \right] \le C \sqrt{\Gamma}.
$
\end{thm}

\section{Recommendation to select the regularization parameter $\delta_{n,r}$ } \label{sec:recommendation}

\subsection{Recommendation based on finite sample bounds}\label{sec:recommendation_finite}

Our statistical analysis in Section \ref{sec:rates} provides a concrete \emph{oracle} recommendation to select the regularization parameter $\delta_{n,r}$ in \eqref{eq:penalized}. Our choice is based on Theorem \ref{thm:rates} and guarantees that the true data generating process is contained in the ball $B_{\delta_{n,r}}(\P_n)$ with high probability:
\begin{equation}\label{eq:delta_n}
    \delta_{n,r} =   \max \left\{\frac{1}{c_d}, 1 \right\} \left[  \frac{C\,\log\left(2n+1\right)^{r/s}}{\sqrt{n}} \right]^{1/r} ~,
\end{equation}
where $c_d$ is the constant such that $c_d \|(\boldsymbol{\gamma}, -1)\|\le \rho(\boldsymbol{\gamma})+1$ for all $\boldsymbol \gamma$ and some norm $\|\cdot\|$ in $\mathbb{R}^{d+1}$ and $C$ is the constant defined in \eqref{eq:constant_thm3}. 

In the case where the support of $\P$ is compact, we can specialize our recommendation to select the regularization parameter $\delta_{n,r}$ with guidance from  Theorem \ref{thm:1}. This recommendation is given in the following:
\begin{equation}\label{eq:delta_n2}
    \delta_{n,r} =   \max \left\{\frac{1}{c_d}, 1 \right\}  \left[  \frac{C}{\sqrt{n}} \right]^{1/r} ~,
\end{equation}
where $C$ is now the constant defined in \eqref{eq:constant_thm2}. In the case of compact support, our recommended regularization parameter only depends on $\mathbb{P}$ through the diameter of its support.\footnote{Let us remark that estimating the support of a distribution is an intricate statistical question, going back at least to \cite{fisher1943relation}. We refer to \cite{wu2019chebyshev, cuevas1997, BIAU2008} for some recent results in support estimation. We also remark that in some applications (e.g.\ for discrete distributions arising in surveys) it is plausible that $\text{supp}(\P)$ is known and thus it need not be estimated.}   

Theorem~\ref{cor:generalization} below shows that the objective function of the penalized regression in \eqref{eq:penalized} constitutes---up to some adjustment terms---an upper bound for the expected prediction error at $\Q$ (provided it is close to $\P$). 
\begin{thm}
\label{cor:generalization}
Suppose the conditions of Theorem \ref{sec:rates} (or \ref{thm:1}) holds. Consider $\delta_{n,r}$ defined in \eqref{eq:delta_n} (or \eqref{eq:delta_n2}). Then, for any $\epsilon\ge 0$ and $\Q$ such that $\iW_r(\P, \Q) \le \epsilon$, with probability greater than $1-3\alpha$, we have
\begin{equation*}
    \E_{\Q} \left[\left|Y-\mathbf{X}^\top {\boldsymbol \beta}\right|^r\right]^{1/r} \le  \E_{{\P_n}} \left[\left|Y-\mathbf{X}^\top {\boldsymbol \beta}\right|^r\right]^{1/r} +  \left(\delta_{n,r} + \epsilon \right) \left(\sigma+\rho({\boldsymbol \beta}) \right),~ \forall {\boldsymbol \beta}~.
\end{equation*}
\end{thm}
The result implies that linear predictors that solve \eqref{eq:penalized} and use our recommended parameters $\delta_{n,r}$ have good out-of-sample performance at the true, unknown distribution of the data $\P$, and, also, at \emph{testing} distributions $\Q$ that are close to $\P$ in the $\rho$-MSW metric.

The oracle recommendation for the regularization parameter $\delta_{n,r}$ is typically not feasible as it depends on the unknown parameter $\Gamma$. The next section presents a normalization strategy on the covariates such that Theorem~\ref{cor:generalization} holds with $\Gamma = 2^s$ and $\sigma = \max \{ \E_{\P} [ |Y|^s ]^{1/s},~1  \}$.

An interesting avenue for future work is to use the results in \cite{kengo12022statistical} (which assume log-concavity of the joint distribution of covariates and outcomes) to recommend a regularization parameter roughly of order: 
$|| \Sigma ||_{op}^{1/2} \sqrt{d \log(n)} / n^{1/r} ,$
where $r \geq 2$ and $|| \cdot ||_{op}$ is the operator norm of the covariance matrix of $(X,Y)$. To do this, it would be necessary to recover the implicit constants that appear in Theorem 1 of \cite{kengo12022statistical} (which only depend on $r$), and additionally provide  some results for the consistent estimation of the operator norm of $\Sigma$.  Because the rates in \cite{kengo12022statistical} are faster than ours (when $d$ is fixed, their rates are of order $n^{1/r}$ whereas ours are of order $n^{1/2r}$), the regularization parameters based on the results of \cite{kengo12022statistical} will typically be smaller (making it less likely that the robust predictors ignore the available covariates). However, we remark that even for the Wasserstein metric on the real line, the rates of order $n^{1/2r}$ cannot be improved upon, unless one imposes further structure on the true data generating process. See Theorem 7.11 and Corollary 7.12 in \cite{bobkov2019one}, and the discussion therein.

We note that our approach for choosing the regularization parameter, $\delta$, is explicitly designed to guarantee the bound on out-of-sample prediction error presented in  Theorem 6. As we have explained before, a sufficient condition to obtain such a bound is to ensure that the true distribution, $\mathbb{P}$, belongs to the ball $B_{\delta_{n,r}}^{r,\rho,\sigma}(\mathbb{P}_n)$ with probability at least $1-\alpha$. Thus, Theorem 6 is possible thanks to the statistical analysis of the $\rho$-MSW metric. \cite{blanchet2019robust} acknowledges that a similar strategy for selecting $\delta$ using concentration inequalities for the standard Wasserstein metric would yield a recommendation of order $O(n^{-1/d})$; see their discussion after Theorem 4, p. 848. However, it is important to mention that there are other possibilities for choosing $\delta$ that do not necessarily target generalization error. For instance, if we followed the objective described in Section 1.1.2 of \cite{blanchet2019robust} (which the authors describe as covering the true parameter of a linear regression model with probability at least $1-\alpha$), it would be possible to recommend values for the regularization parameter of order $O(n^{-1/2})$. In particular, for the $\sqrt{\text{LASSO}}$ the authors recommend a tuning parameter equal to \[ \lambda =  \frac{\pi}{\pi-2} \frac{\Phi^{-1}(1-\alpha/2d)}{\sqrt{n}}, \]   which, up to a constant, coincides with the recommendation in \cite{belloni2011square}. In Section \ref{subsection:blanchet_delta} of the Supplementary Material \cite{MRVW2024}, we show that if we adopt the objective of \cite{blanchet2019robust} (and their assumptions),  but use our DRO representation based on the $\rho$-MSW metric, we could recommend the same or even a smaller regularization parameter.

\subsection{Covariate Normalization} \label{sec:normalization}

In this section, we assume that the covariates in the data have been \emph{normalized} to satisfy $\E_{\P_n}\left[  \left\|(\mathbf{X} ,0)\right\|_*^{s}\right] =1$. This means that under minimal regularity conditions we can assume that the true data generating process satisfies  $\E_{\P}\left[  \left\|(\mathbf{X},0) \right\|_*^{s}\right] =1 $. It is common practice to impose some covariate normalization to estimate the parameters of the best linear predictor using the $\sqrt{\text{LASSO}}$ and related estimators; see \cite[Equation 4 p.2]{belloni2011square} for an example of a coordinate-wise, unit variance normalization.   

The next theorem proposes a simple formula to select the regularization parameter $\delta_{n,r}$ under our suggested normalization.

\begin{thm} \label{thm:pivotal}
Suppose \textcolor{black}{$\E_{\P}\left[  \left\|(\mathbf{X},0) \right\|^{s}_*\right] =1$ and $\E_{\P} [ \|(0, \dots, 0,Y)\|^s_*]^{1/s} < +\infty$ for some $s>2r$. In addition, assume that \eqref{eq:new}  holds for an arbitrary norm in $\R^d$ and $c_d>0$,} and $\sigma = \max \{ \E_{\P} [ \|(0, \dots, 0,Y)\|^s_*]^{1/s},~1  \}$. Consider
\begin{equation}\label{eq:delta_n3}
 \delta_{n,r} :=  \max \left\{\frac{1}{c_d}, 1 \right\} \left[  \frac{C\,\log(2n+1)^{r/s})}{\sqrt{n}} \right]^{1/r},
\end{equation}
where 
\begin{align*}
C:= 
2^r r \Big(180\sqrt{d+2}+\sqrt{2\log\left(\frac{1}{\alpha}\right)}+\sqrt{\frac{2^{s} }{\alpha}} \frac{8}{s/2-r} \sqrt{\log\left(\frac{8}{\alpha}\right) + (d+2)}\Big).
\end{align*}
 Then, for any $\epsilon \ge 0$ and $\Q$ such that $\iW_r(\P,\Q) \le \epsilon$, with probability greater than $1-3\alpha$,
 \begin{equation*}
    \E_{\Q} \Big[\big|Y-\mathbf{X}^\top {\boldsymbol \beta}\big|^r\Big]^{1/r} \le  \E_{{\P_n}} \Big[\big|Y-\mathbf{X}^\top {\boldsymbol \beta}\big|^r\Big]^{1/r} +  \left(\delta_{n,r} + \epsilon \right) \left(\sigma+\rho({\boldsymbol \beta})\right),~ \quad \forall {\boldsymbol \beta}~.
\end{equation*}
 
\end{thm}

\subsection{Asymptotic recommendation}\label{sec:recommendation_ass}

For compactly supported measures, Theorem \ref{thm:ass_bdd} yields the asymptotic \emph{oracle} recommendation 
\begin{equation} \label{eqn:oracle_simulations}
 \delta_{n,r} = c_{\rho,d}~ \left[  n^{-1/2} \cdot   q_{1-\alpha} \right]^{1/r} \cdot \mathrm{diam}\left( \mathrm{supp}(\P)\right), 
\end{equation}
where $c_{\rho,d}$ is as in \eqref{eq:bar_W},  $q_{1-\alpha}$ is the $(1-\alpha)$-quantile of the Kolmogorov distribution. In the general case, Theorem \ref{thm:ass} yields
$\delta_{n,r} = [ n^{-1/2}\cdot \Gamma^{1/2} \cdot   C]^{1/r},$
for some constant $C=C(r,s,d,\alpha).$
However, the constant $C$ is hard to determine explicitly. In particular it depends on $\alpha$ through the quantile of the zero-mean Gaussian process $(G_f)_{f\in \mathcal{H}^+\cup \mathcal{H}^0 \cup \mathcal{H}^-},$ whose covariance structure depends on $\P$ and is given in \eqref{eq:covariance} and is hard to bound explicitly. We leave this issue for future research.

\subsection{Application: ranking of estimators}\label{sec:ranking_estimators}

Consider two estimators ${\boldsymbol \beta}_1={\boldsymbol \beta}_1(\mathbb{P}_n)$ and ${\boldsymbol \beta}_2={\boldsymbol \beta}_2(\mathbb{P}_n)$, where $\mathbb{P}_n$ denotes the empirical distribution of i.i.d.\ draws from a true distribution $\mathbb{P}$.  In this section, we investigate whether ${\boldsymbol \beta}_1$ has a better out-of-sample performance than ${\boldsymbol \beta}_2$ over an uncertainty set $B$. That is,
\begin{equation}\label{eq:H0}
  \sup_{\mathbb{Q} \in B} ~ \mathbb{E}_{\mathbb{Q}}  \left[\left|Y-\mathbf{X}^\top {\boldsymbol \beta}_1\right|^r\right]^{1/r} ~ \le ~ \sup_{\mathbb{Q} \in B} ~ \mathbb{E}_{\mathbb{Q}}  \left[\left|Y-\mathbf{X}^\top {\boldsymbol \beta}_2\right|^r\right]^{1/r}.
\end{equation}
We restrict our attention to uncertainty sets $B$ that verify two conditions: 
\begin{enumerate}%
    \item[(i)] $B \subseteq B_{\delta}(\mathbb{P})=B_\delta^{r,\rho,\sigma}(\P) $ for some $\sigma$, $\delta$, and $\rho$.
    
    \item[(ii)] The supremum on the left side of \eqref{eq:H0} is achieved for $\mathbb{P}^*_{{\boldsymbol \beta}_1}$, and the supremum on the right side of \eqref{eq:H0} is achieved for $\mathbb{P}^*_{{\boldsymbol \beta}_2}$, where $ \mathbb{P}^*_{{\boldsymbol \beta}_j}$ are defined according to Corollary \ref{cor:1}, $j=1,2$.
\end{enumerate}
Examples of such sets $B$ are given in Section \ref{sec:examples}.
 Note that we cannot evaluate \eqref{eq:H0} directly, as $\P$ is not observed. Instead, we propose the test statistic
$$ T_n = n^{1/(2r)} \left( \frac{\E_{{\P_n}} \left[\left|Y-\mathbf{X}^\top {\boldsymbol \beta}_1\right|^r\right]^{1/r} - \E_{{\P_n}} \left[\left|Y-\mathbf{X}^\top {\boldsymbol \beta_2}\right|^r\right]^{1/r} + \delta \rho({\boldsymbol \beta_1}) -  \delta \rho({\boldsymbol \beta_2})}{2\sigma + \rho({\boldsymbol \beta_1}) + \rho({\boldsymbol \beta_2})} \right)~.$$
The next corollary states that $T_n$ gives rise to a size--$\alpha$ test. For notational simplicity we focus on compactly supported probability measures $\P$, and simply remark that the same reasoning can be used to derived tests for general $\P$ satisfying the assumptions of Theorems \ref{thm:rates} and \ref{thm:ass}.

\begin{corollary}\label{cor:rank_estimators}
    In the setting of Theorems \ref{thm:penalized} and  \ref{thm:1}, consider $C$ and $c_{\rho,d}$ defined in \eqref{eq:constant_thm2} and \eqref{eq:bar_W}. Then, for any ${\boldsymbol \beta}_1$ and ${\boldsymbol \beta}_2$ satisfying \eqref{eq:H0}, we have  
$ P ( T_n >  c_{\rho,d} C^{1/r}) \le \alpha.$
\end{corollary}

\section{Simulations}\label{sec:simulation}
Suppose that the training data consists of $n$ i.i.d.\ draws from a linear regression model, meaning $Y_i = \mathbf{X}_i^{\top} \boldsymbol{\beta} + \sigma \varepsilon_i.$
We take $\varepsilon_i$ to be uniformly distributed over the interval $[-1,1]$. The vector of covariates, $\mathbf{X}_i \in \mathbb{R}^{d}$, is generated as 
$\mathbf{X}_i = \sigma \lambda \tilde{\mathbf{X}}_i,$
where $\tilde{\mathbf{X}}_i$ is a $d$-dimensional vector of independent uniform random variables over the $[0,1]$ interval, independently of $\varepsilon_i$. The parameters controlling the simulation design are $(\boldsymbol{\beta}, \sigma, \lambda, d)$. 

We first focus on linear prediction using coefficients estimated via the $\sqrt{\text{LASSO}}$ ($r=2$). Recall from \eqref{eqn:oracle_simulations} that our oracle recommendation for the tuning parameter is  
$n^{-1/4} \cdot  \left( q_{1-\alpha} \right)^{1/2} \cdot \mathrm{diam}\left( \mathrm{supp}(\P)\right),$
where $q_{1-\alpha}$ is the $1-\alpha$ quantile of the Kolmogorov distribution. Algebra shows (see Section \ref{subsubsection:diameter_P} of the Supplementary Material \cite{MRVW2024}) that 
\[\mathrm{diam}\left( \mathrm{supp}(\P)\right) = \sigma \lambda \left( d + \left( \| \boldsymbol{\beta} \|_{1}   + (2/\lambda) \right)^2   \right)^{1/2}. \]

For comparison, the typical oracle recommendation for the $\sqrt{\text{LASSO}}$ based on \cite{belloni2011square}, can be shown to equal
\begin{equation} \label{eqn:oracle_lasso}
n^{-1/2} \cdot 3^{-1/2} \sigma \lambda \cdot \Phi^{-1} \left( \frac{1}{2} + \frac{(1-\alpha)^{(1/d)}}{2} \right). 
\end{equation}

\begin{figure}%
    \centering
    \includegraphics[scale=0.6]{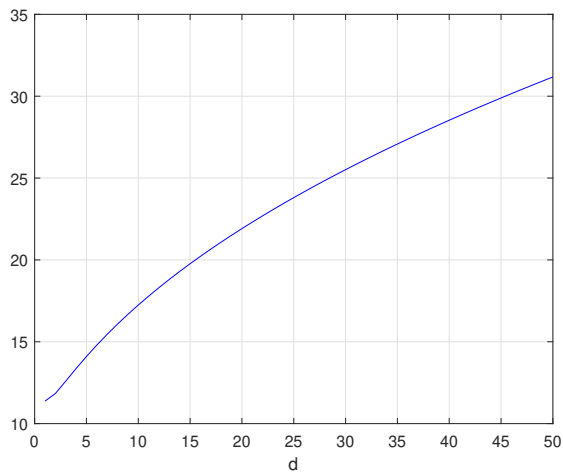}
    \caption{(Blue, Solid Line) Ratio of the oracle recommendations in \eqref{eqn:oracle_simulations} and \eqref{eqn:oracle_lasso}: $\sigma=1, \lambda=10, \alpha=0.05, n=2,500$ and $\boldsymbol{\beta} =[1,0,\ldots 0]^{\top}$. } 
    \label{fig:comparison}
\end{figure}

Figure \ref{fig:comparison} compares the ratio of \eqref{eqn:oracle_simulations} relative to \eqref{eqn:oracle_lasso}. The figure shows that our recommendation can be more than ten times larger than the typical recommendations in the literature. Thus, one first concern is that the distributional robustness guaranteed by our choice of $\delta_n$ could be achieved by setting all the coefficients to be zero (an adversarial nature cannot increase much the generalization error of such a predictor, as it does not rely at all on covariates). We also note that the recent work of \cite{caner2024should} has shown that larger tuning parameters could lead to \emph{incentive compatibility} in certain human-machine interactive environments.   

We now argue that in our simulation design it is possible to figure out the smallest sample size that would be required to avoid a ``trivial'' prediction. It is known, see \cite{stucky2017sharp}, that $\boldsymbol{\beta}=0_{d \times 1}$ is a solution to the $\sqrt{\text{LASSO}}$ problem if and only if
\begin{equation} \label{eqn:FOC_sqrtlasso}
\frac{ \| \frac{1}{n} \sum_{i=1}^{n} \mathbf{X}_i y_i \|_\infty  }{\sqrt{\frac{1}{n} \sum_{i=1}^n y_i^2 }} \leq \delta_{n,2}.
\end{equation}
Using a Central Limit Theorem and a Law of Large of numbers, algebra shows (see Section \ref{subsubsection:bound_appendix} in the Supplementary Material \cite{MRVW2024}) that \eqref{eqn:FOC_sqrtlasso} holds with high probability whenever  

\begin{equation} \label{eqn:conjecture}
n \leq 9 \cdot \left  \| \frac{\boldsymbol{\beta}}{ \sqrt{\boldsymbol{\beta}^{\top} \boldsymbol{\beta} } } \right  \|^{-4}_\infty  \cdot \left( q_{1-\alpha} \right)^{2} \cdot \left( d +   \left(  \| \boldsymbol{\beta} \|_{1}   + (2/\lambda) \right)^2   \right)^{2}.
\end{equation}
For $d=10$, $\boldsymbol{\beta}=(1,0,0 \ldots 0)^{\top}$, $\alpha=.05$ (or equivalently $q_{1-\alpha}=1.358$) the corresponding conservative bound is of about 2,200. This means that it will take a relatively large sample size in order for our regularization parameter to select at least some covariates for prediction. 

We verify this conjecture numerically. We simulate data using the $\sigma =1, \lambda=10, d=10, \boldsymbol{\beta} =[1,0,\ldots,0]^{\top}$ and consider sample sizes $n \in \{2125, 2250, 2375, 2500\}$. Our design corresponds to a low-dimensional problem (10 covariates and at least 2,000 observations). Figure \ref{fig:non_zero_coeffs} below reports the histogram associated to the number of nonzero coefficients selected by the $\sqrt{\text{LASSO}}$ using the regularization parameters in \eqref{eqn:oracle_simulations} and \eqref{eqn:oracle_lasso}. The numerical results reported below are in line with the bound derived in \eqref{eqn:conjecture}.

\begin{figure}%
\centering
     \begin{subfigure}[c]{0.45\textwidth}
        \centering
        \includegraphics[width=\textwidth]{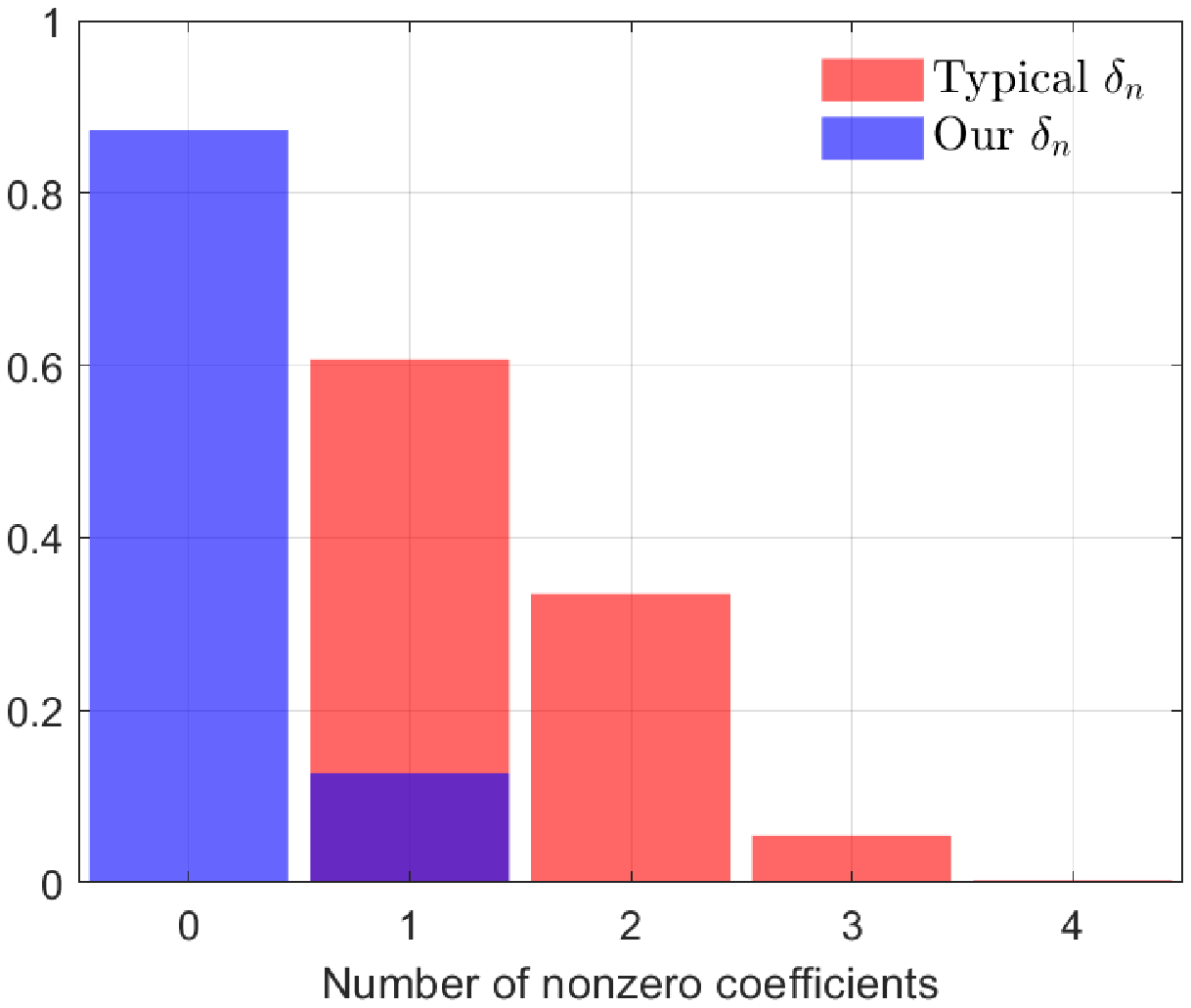}
    \caption{ $n=2125$ }
    \end{subfigure}
    \hfill
    \begin{subfigure}[c]{0.45\textwidth}
        \centering
        \includegraphics[width=\textwidth]{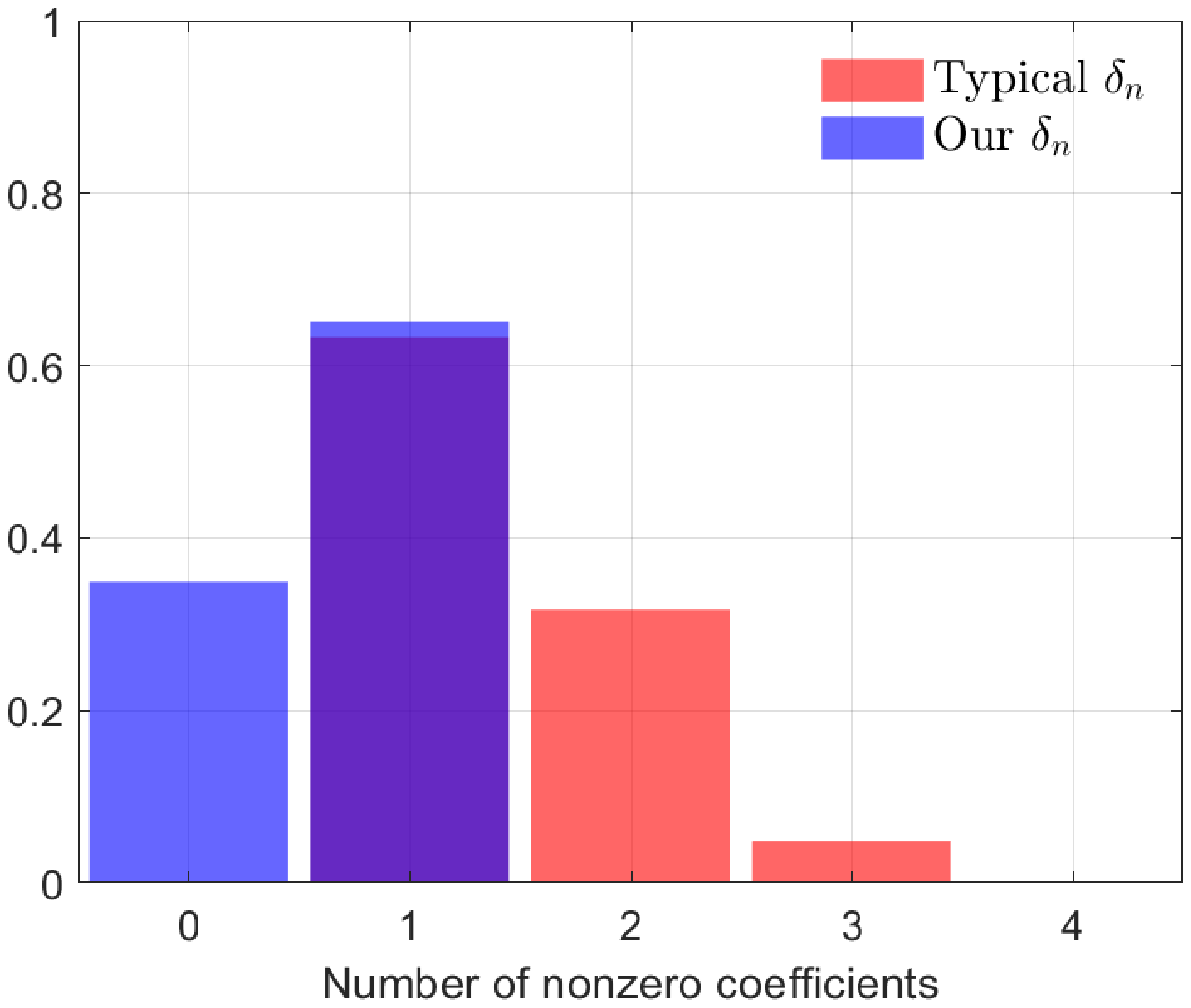}
    \caption{ $n=2250$ } 
    \end{subfigure}
    \vskip\baselineskip
    \begin{subfigure}[c]{0.45\textwidth}
        \centering
        \includegraphics[width=\textwidth]{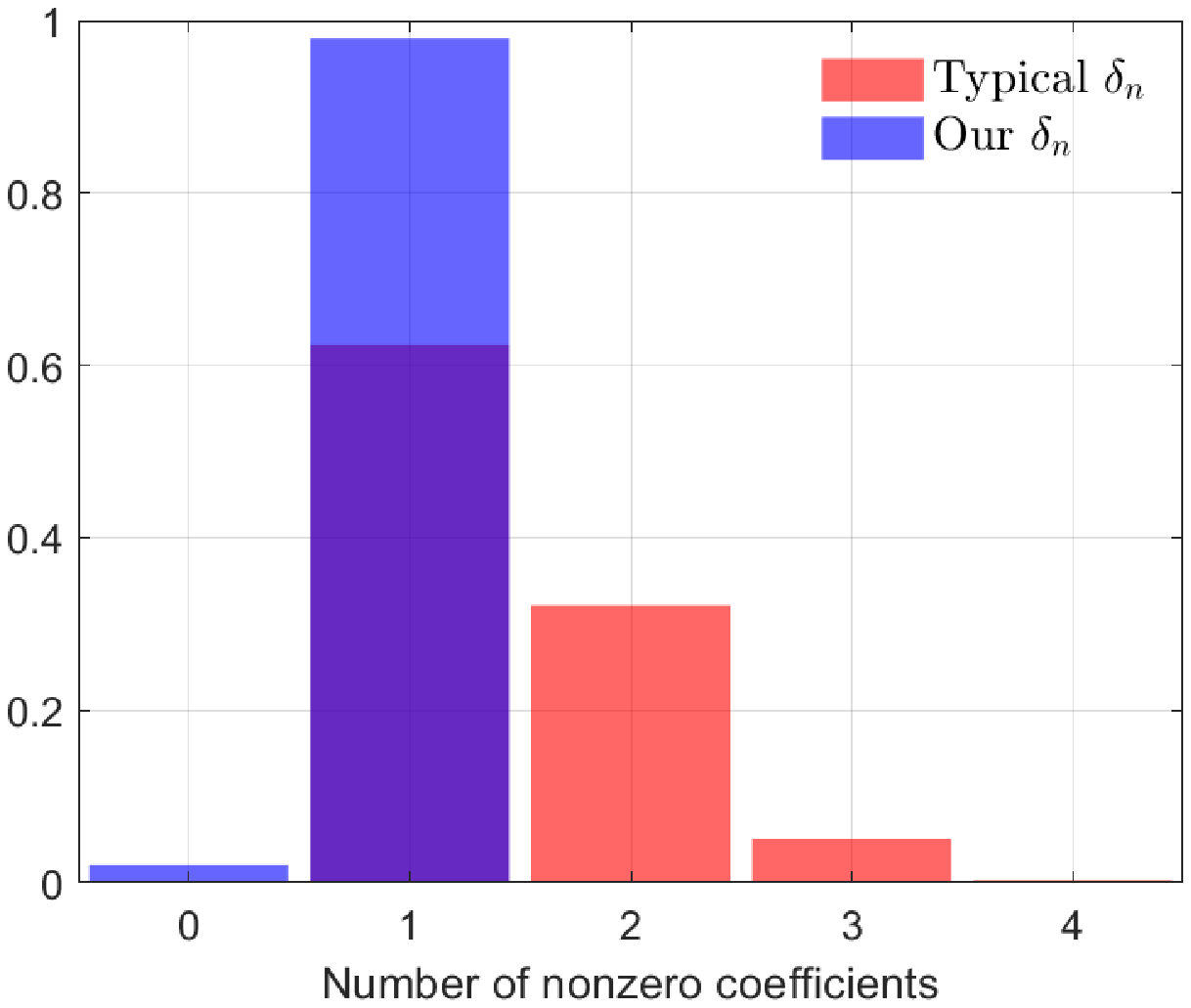}
    \caption{ $n=2375$ }
    \end{subfigure}
    \hfill
    \begin{subfigure}[c]{0.45\textwidth}
        \centering
        \includegraphics[width=\textwidth]{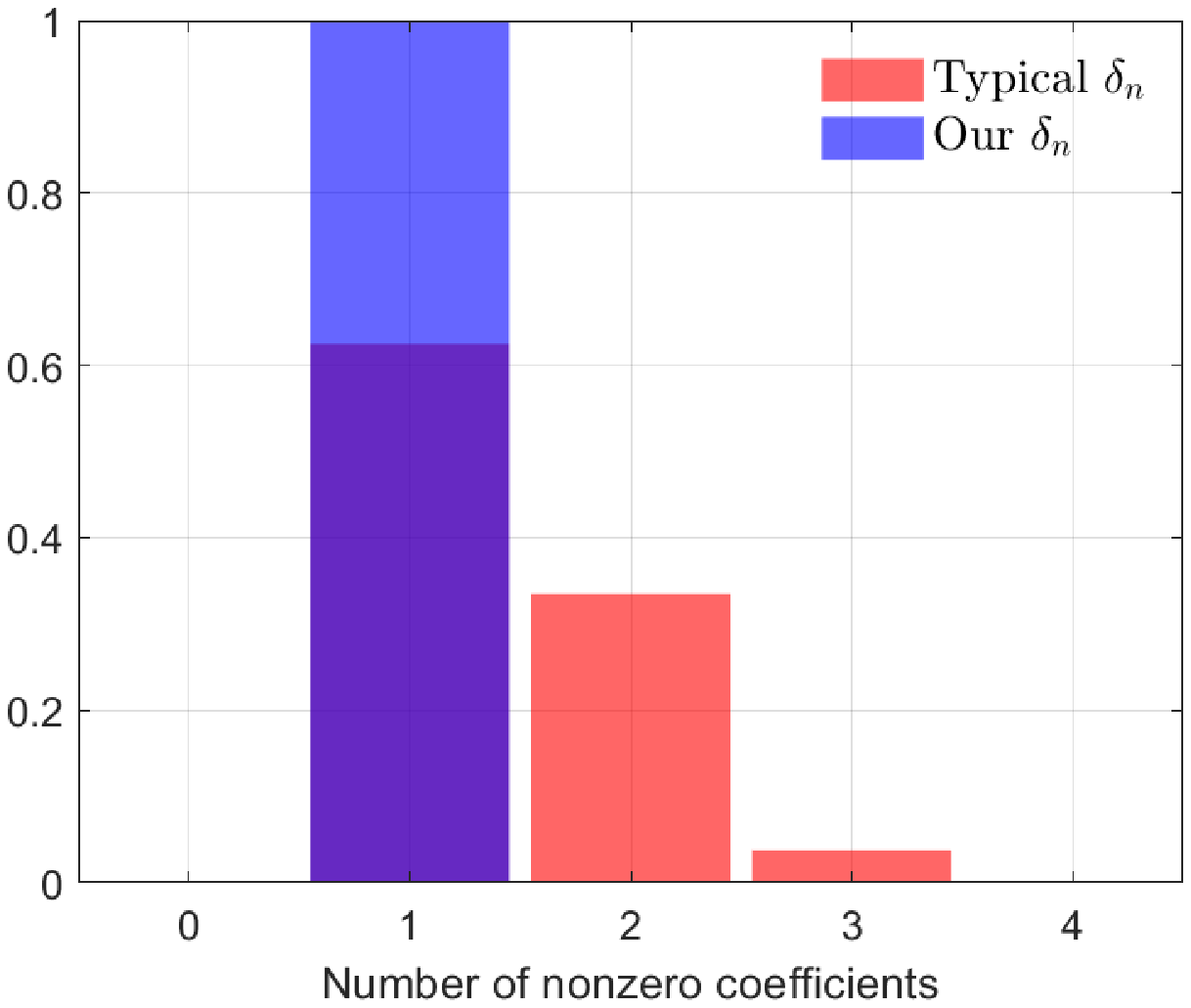}
    \caption{ $n=2500$ } 
    \end{subfigure}    
    \caption{\small Fraction of simulation draws in which the $\sqrt{\text{LASSO}}$ selects 0,1,\ldots, 4 nonzero coefficients: (Blue) oracle $\delta_n$ in  \eqref{eqn:oracle_lasso}; (Red) our oracle recommendation of $\delta_n$ in \eqref{eqn:oracle_simulations}.}
    \label{fig:non_zero_coeffs}
\end{figure}

\emph{Training/Testing error.} Figure \ref{fig:training_testing} below reports the training/testing root-mean squared prediction error (RMSPE) associated to the three estimators considered in our simulations: the OLS estimator, the $\sqrt{\text{LASSO}}$ with the $\delta_n$ in \eqref{eqn:oracle_lasso}, and the $\sqrt{\text{LASSO}}$ with the $\delta_n$ in \eqref{eqn:oracle_simulations}. The training data is generated according to the design described above for a sample size of $n=2500$. For testing, we perturb the true data generating process according to the worst-case distribution derived in  Corollary 1 with $\delta_n$ in \eqref{eqn:oracle_simulations} replacing $\delta$. The plots report the histogram---across simulations---of the ``relative'' root mean-squared prediction error in the training (or testing) data. Each figure compares the estimators indicated in the legend below them. For example, Panel a) of Figure \ref{fig:training_testing} reports the root MSPE of the $\sqrt{\text{LASSO}}$ (SQL), divided by the root MSPE of OLS, in both the training and testing data.  

\begin{figure}[t!]
\centering
     \begin{subfigure}[b]{0.45\textwidth}
        \centering
        \includegraphics[width=\textwidth]{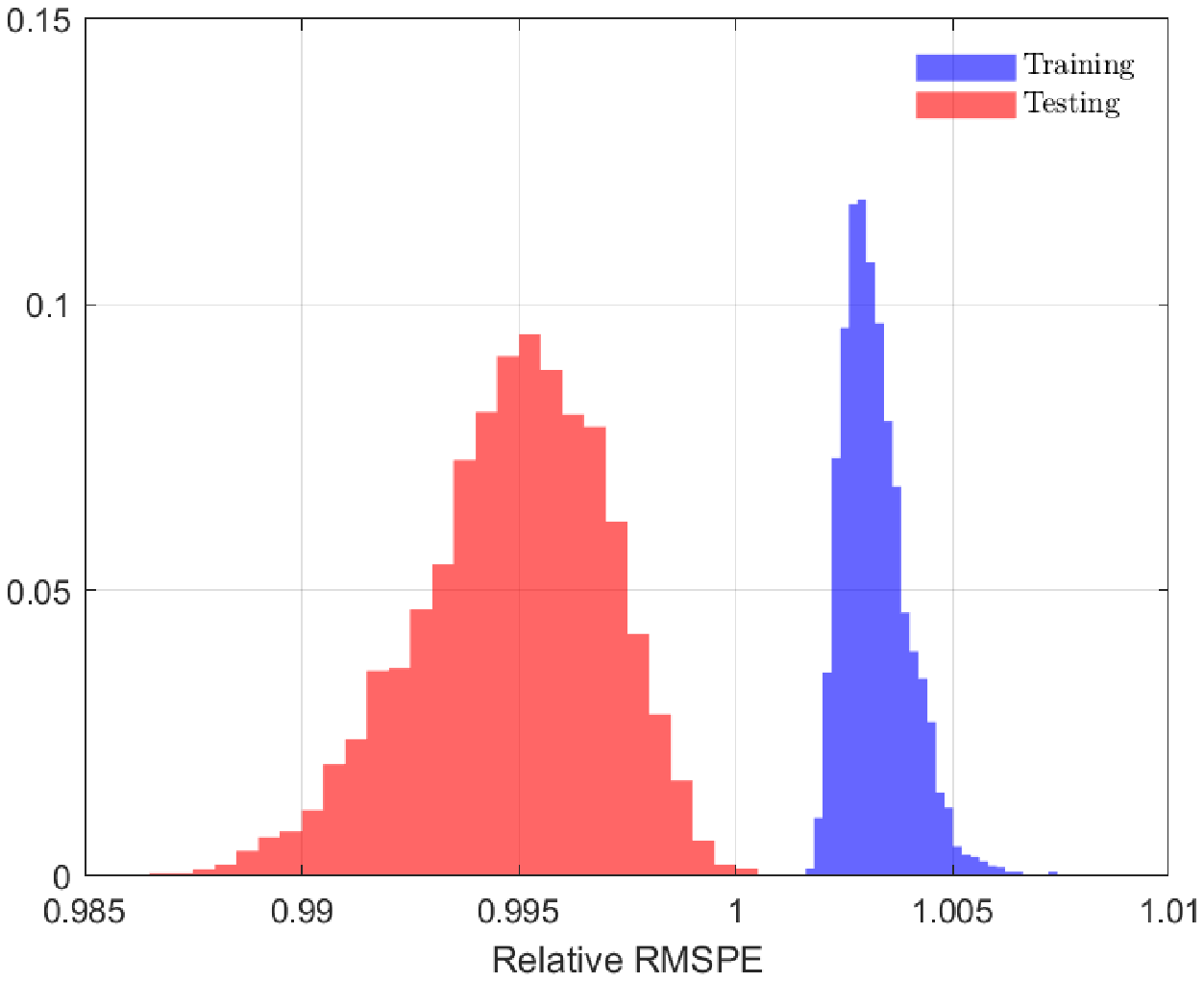}
    \caption{ SQL/OLS }
    \end{subfigure}
    \hfill
    \begin{subfigure}[b]{0.45\textwidth}
        \centering
        \includegraphics[width=\textwidth]{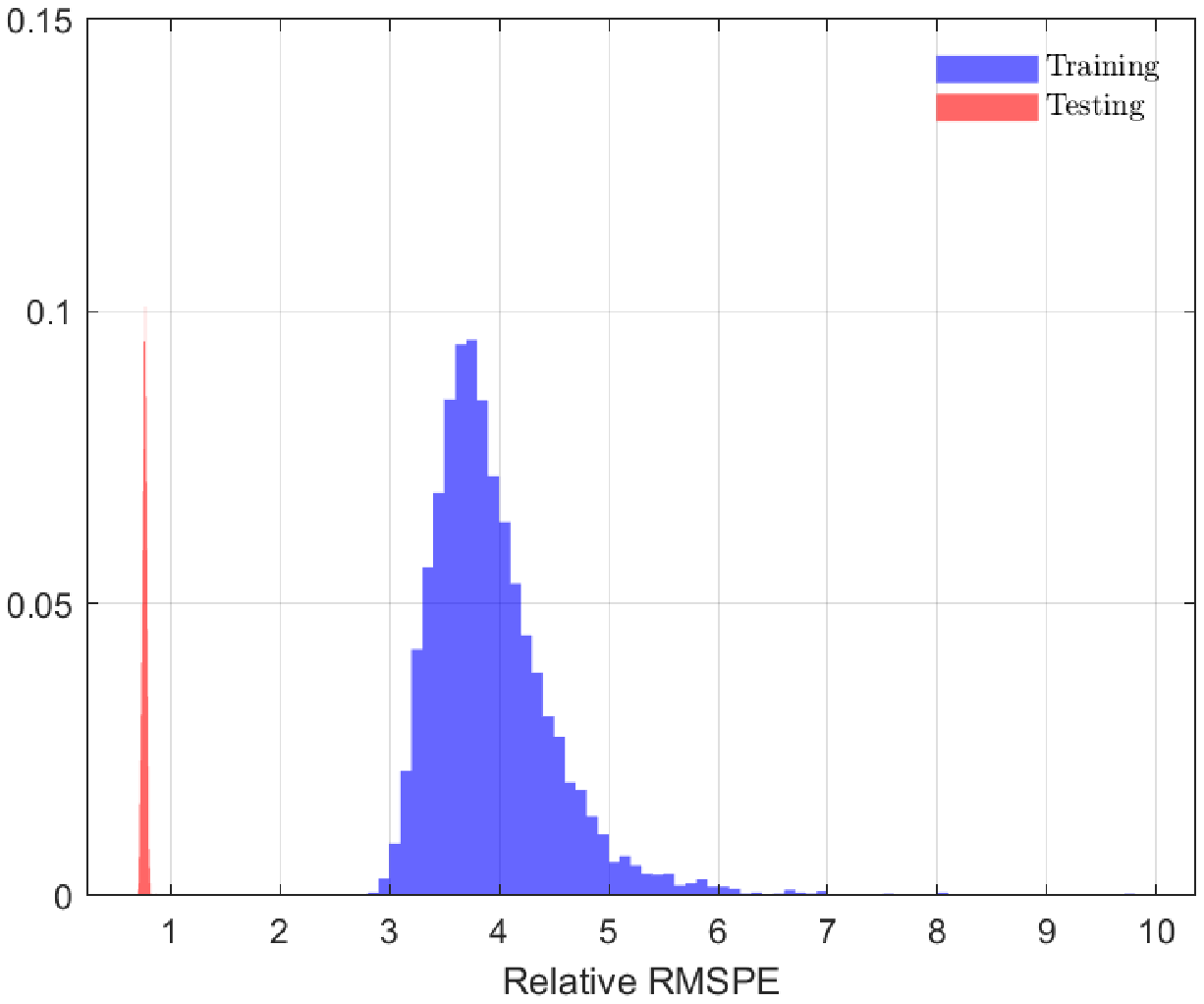}
    \caption{ SQL (new $\delta_n$) / SQL} 
    \end{subfigure}
    \caption{\small Relative Training/Testing Root Mean-Squared Prediction Error.}
    \label{fig:training_testing}
\end{figure}

The simulation results are in line with the theoretical predictions. First, since we are considering a simulation design where $n$ is large relative to $d$, the oracle $\delta_n$ in \eqref{eqn:oracle_lasso} is close to zero. This means that the predictions associated to the $\sqrt{\text{LASSO}}$ in the training sample are not very different to those obtained via OLS. Panel a) of Figure \ref{fig:training_testing} indeed shows that the relative training error between the $\sqrt{\text{LASSO}}$ (with the typical $\delta_n$) and OLS remains very close to one across simulations. Panel b) shows that that the difference between the regularization parameters in \eqref{eqn:oracle_lasso} and \eqref{eqn:oracle_simulations} generates a sizeable difference in training error. However, that the larger value of $\delta_n$ does translate to better out-of-sample performance.

Finally, we verify the bound in Theorem~\ref{cor:generalization}. The corollary implies that with probability at least $95\%$ the root-MSPE of the $\sqrt{\text{LASSO}}$ in the \emph{testing} set (for any distribution in the ball that is $\epsilon$ away from the true data generating process) must be bounded by the sum of a) the root-MSPE of the $\sqrt{\text{LASSO}}$ in the \emph{training} set and ii) $(\delta_n + \epsilon) (\sigma + \rho(\boldsymbol{\beta}))$. Figure \ref{fig:bounds}  shows that the bound holds for the recommended $\delta_n$, but not for the usual one. Additional simulations are reported in Section \ref{sec:additional_sims} of the Supplementary Material \cite{MRVW2024}.

\begin{figure}[t!]
    \begin{center}
     \begin{subfigure}[c]{0.45\textwidth}
        \centering
        \includegraphics[scale=.5]{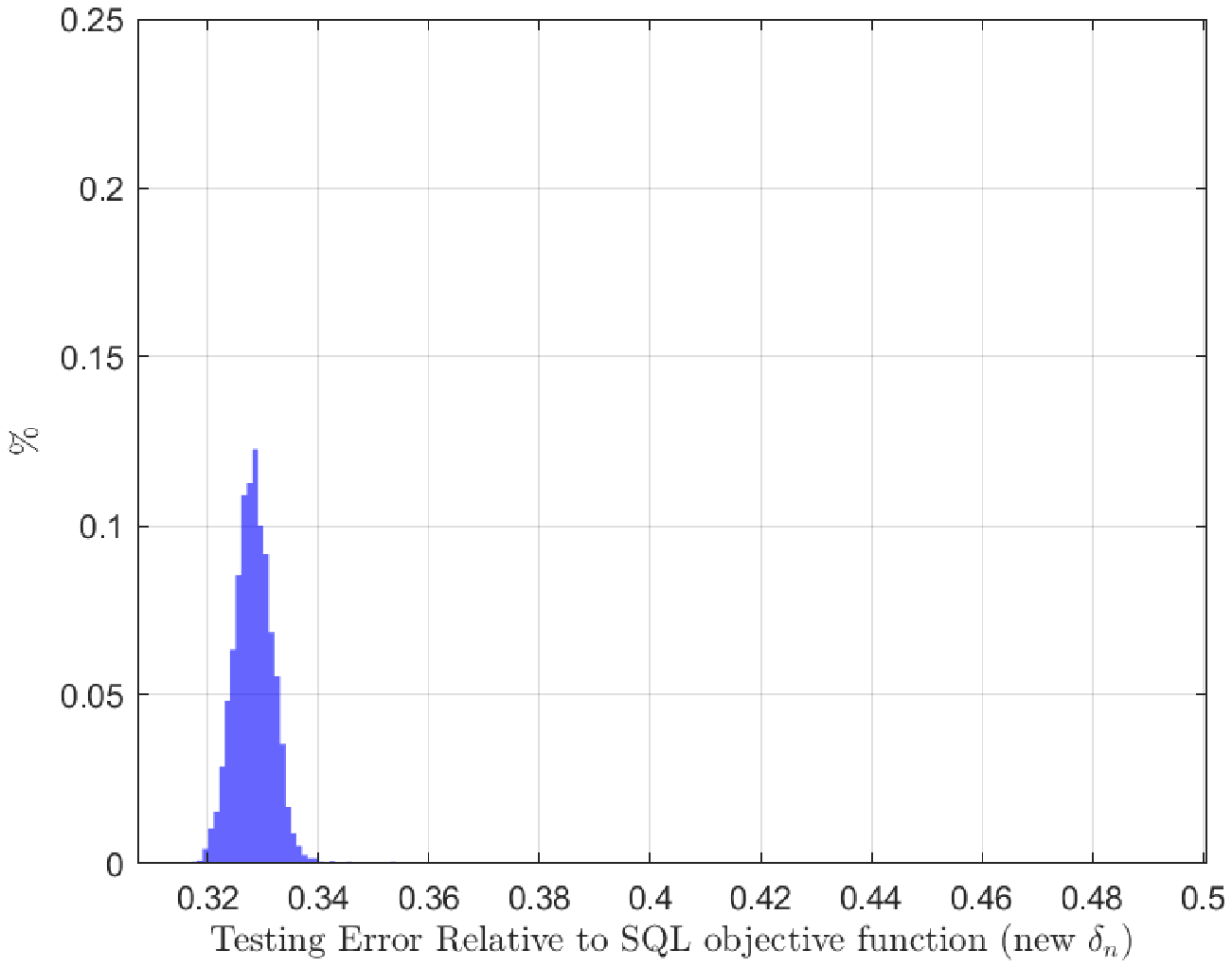}
    \caption{ SQL (new $\delta_n$) }
    \end{subfigure}
    \hfill
    \begin{subfigure}[c]{0.45\textwidth}
        \centering
        \includegraphics[scale=.5]{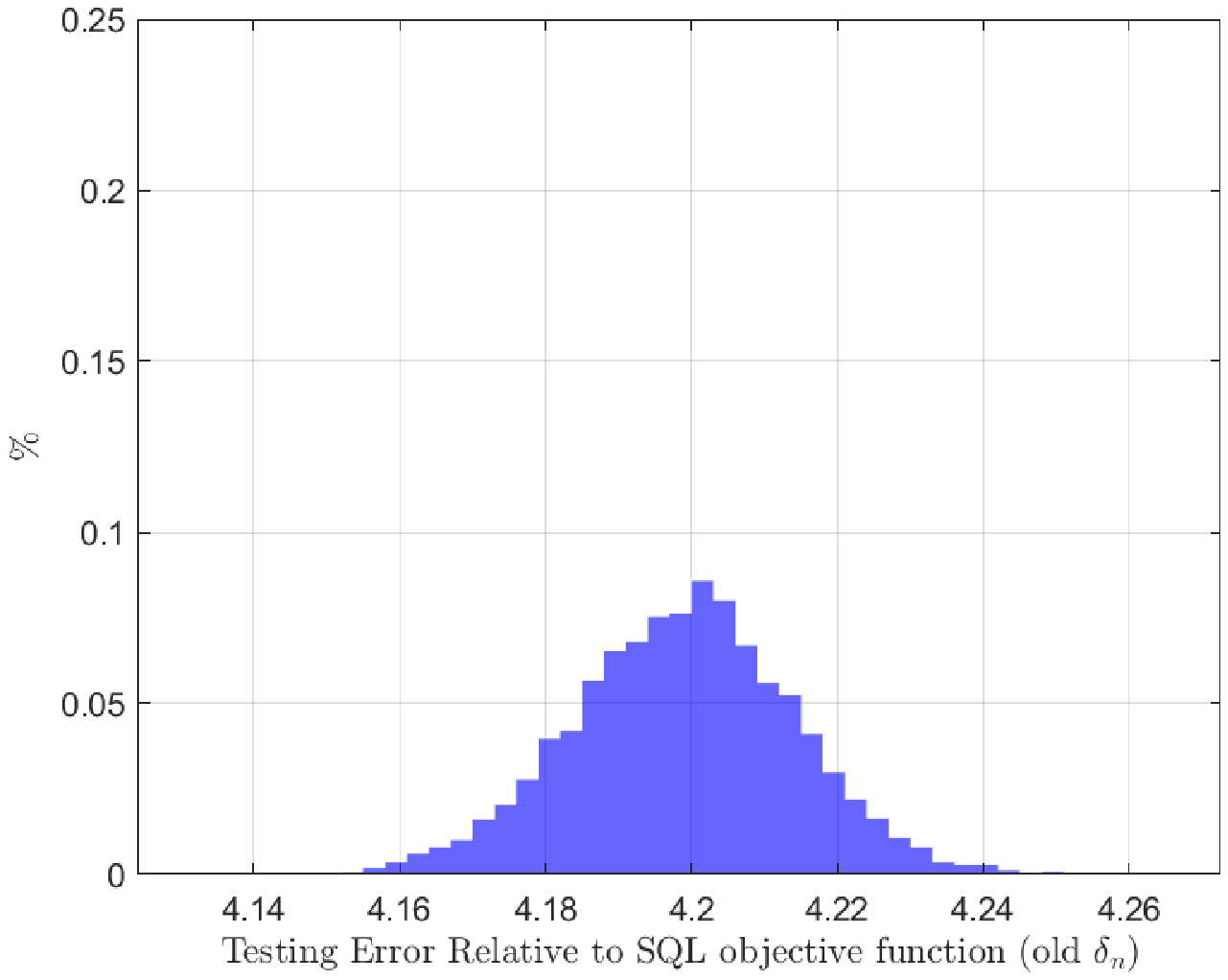}
    \caption{ SQL (old $\delta_n$) } 
    \end{subfigure}
    \end{center}
    \caption{\small Testing error and Objective function}
    \label{fig:bounds}
\end{figure}

\bibliographystyle{ecta}
\bibliography{references}

\newpage
\appendix

\section{Proof of Results in the Main Text}

\subsection{Proof of Theorem~\ref{thm:penalized}} \label{sec:THMpenalizedproof}

In Section~\ref{sec:dro} we provide a proof sketch after the statement of Theorem~\ref{thm:penalized}. Here we elaborate on the details of the proof. The statements of two steps mentioned in Section~\ref{sec:dro} are repeated below for the reader's convenience.

\begin{itemize}
\item[\textbf{Step 1.}] We show that
\begin{equation}
\label{eq:step1_secproofs}
\left( \mathbb{E}_{\widetilde{\mathbb{P}}} \left[ \left|Y-\mathbf{X}^{\top}\boldsymbol{\beta} \right|^r \right] \right)^{1/r} \leq \sqrt[r]{\mathbb{E}_{\mathbb{P}} \left[ \left|Y-\mathbf{X}^{\top} \boldsymbol{\beta}\right|^r\right]}  + \delta \left(\sigma+\rho(\boldsymbol{\beta})\right),
\end{equation}
holds for any $\boldsymbol{\beta} \in \mathbb{R}^d$ and any $\widetilde{\mathbb{P}} \in B_{\delta}(\mathbb{P})$.

\textbf{Proving Step 1.} Take an arbitrary $\widetilde{\mathbb{P}} \in B_{\delta}(\mathbb{P})$ and let $\pi(\boldsymbol{\beta})$ be an optimal coupling for $\iW_{r,\rho, \sigma}$. By writing $\pi(\boldsymbol{\beta})$ we emphasize that the coupling will depend on $\boldsymbol{\beta}$; though, this matters little for the proof. Namely, $((\mathbf{X}, Y), (\widetilde{\mathbf{X}},\widetilde{Y})) \sim \pi(\boldsymbol{\beta})$ with $(\mathbf{X}, Y) \sim \mathbb{P}$ and $(\widetilde{\mathbf{X}},\widetilde{Y}) \sim \widetilde{\mathbb{P}}$. 
Consequently we conclude that
\[  \mathbb{E}_{\widetilde{\mathbb{P}}} \left[ \left|Y-\mathbf{X}^{\top}\boldsymbol{\beta}\right|^r \right]  =  \mathbb{E}_{\pi(\boldsymbol{\beta})} \left[ \left|\widetilde{Y}-\widetilde{\mathbf{X}}^{\top}\boldsymbol{\beta}\right|^r \right] . \]
By the triangle inequality we obtain
\begin{equation*}
\begin{split}
&\sqrt[r]{\mathbb{E}_{\pi(\boldsymbol{\beta})} \big[ \big |\widetilde{Y}-\widetilde{\mathbf{X}}^{\top}\boldsymbol{\beta}\big |^r \big]}  = \sqrt[r]{\mathbb{E}_{\pi(\boldsymbol{\beta})} \big[ \big | (\widetilde{Y} - Y) + (Y - \mathbf{X}^{\top}\boldsymbol{\beta}) + ( \mathbf{X} - \widetilde{\mathbf{X}})^{\top}\boldsymbol{\beta})\big|^r \big]} \\
& \hspace{22mm}  \leq  \sqrt[r]{\mathbb{E}_{\pi(\boldsymbol{\beta})} \big[|Y - \mathbf{X}^{\top}\boldsymbol{\beta}|^r \big]}+  \sqrt[r]{\mathbb{E}_{\pi(\boldsymbol{\beta})}\big[\big| (\widetilde{Y} - Y)+(\mathbf{X} - \widetilde{\mathbf{X}})^{\top}\boldsymbol{\beta})\big|^r \big]}.
\end{split}
\end{equation*}
Recalling the choice of $\pi(\beta)$ we conclude that
\begin{equation}
\begin{split}
\sqrt[r]{\mathbb{E}_{\pi(\boldsymbol{\beta})} \left[ \left \lvert\widetilde{Y}-\widetilde{\mathbf{X}}^{\top}\boldsymbol{\beta}\right \lvert^r \right]}  \leq  \sqrt[r]{\mathbb{E}_{\mathbb{P}} \left[\left|Y - \mathbf{X}^{\top}\boldsymbol{\beta}\right|^r \right] }+ \delta  \left(\sigma+\rho(\boldsymbol{\beta})\right).
\end{split}
\end{equation}

\item[\textbf{Step 2.}] We show that for any $\boldsymbol{\beta}\in \text{dom}(\rho)$, the upper bound given in \textbf{Step 1} is tight; i.e.\ we construct $\mathbb{P}^* \in B_{\delta}(\mathbb{P})$, for which the bound holds exactly.
 
\textbf{Proof Step 2.} 
Let $\boldsymbol{\beta}^*$ be an element of $\partial \rho (\boldsymbol{\beta})$ satisfying Equation \eqref{equation:condition}.

Consider the distribution $\mathbb{P}^*$ corresponding to the random vector $( \widetilde{\mathbf{X}},\widetilde{Y})$ defined by
\begin{equation} \label{equation:worst-case-dist-rho}
     \widetilde{\mathbf{X}} = \mathbf{X} - e \left( \boldsymbol{\beta}^* - \frac{\boldsymbol{\beta}}{\boldsymbol{\beta}^{\top} \boldsymbol{\beta}} \rho^*(\boldsymbol{\beta}^*) \right),  \qquad \widetilde{Y} = Y+\sigma e,
\end{equation}
where  
\[ e := \frac{\delta(Y-\mathbf{X}^{\top}\boldsymbol{\beta})}{ \sqrt[r]{\mathbb{E}_{\mathbb{P}} \left[ \left|Y-\mathbf{X}^{\top}\boldsymbol{\beta}\right|^r \right]} }, \qquad (Y, \mathbf{X}) \sim \mathbb{P}. \]
The distributions $\mathbb{P}^*$ and $\mathbb{P}$ are already coupled, since $( \widetilde{\mathbf{X}}, \widetilde{Y})$ are measurable functions of $(\mathbf{X},Y) \sim \mathbb{P}$. Let $\pi^*(\boldsymbol{\beta})$ denote the coupling of $(\mathbb{P}^*,\mathbb{P})$. 

Next we show that the distribution $\mathbb{P}^*$ of $( \widetilde{\mathbf{X}}, \widetilde{Y})$ is an element of $B_{\delta}(\mathbb{P})$: by construction we have
\begin{align*}
\mathbb{E}_{\pi^*(\boldsymbol{\beta})} \Big[ \Big |(\tilde Y-Y)+(\mathbf{X}-\widetilde{\mathbf{X}})^{\top} \boldsymbol{\gamma}\Big |^r\Big] &= \mathbb{E}_{\pi^*(\boldsymbol{\beta})} \left[ |e|^r \left|\sigma + \left( \boldsymbol{\beta}^* - \frac{\boldsymbol{\beta}}{\boldsymbol{\beta}^{\top} \boldsymbol{\beta}} \rho^*(\boldsymbol{\beta}^*) \right)^{\top} \boldsymbol{\gamma} \right|^r \right] \\
 &=  \left | \sigma+(\boldsymbol{\beta}^*)^{\top} \boldsymbol{\gamma} - \frac{\boldsymbol{\beta}^{\top} \boldsymbol{\gamma}}{\boldsymbol{\beta}^{\top} \boldsymbol{\beta}} \rho^*(\boldsymbol{\beta}^*)  \right|^r \,\mathbb{E}_{\pi^*(\boldsymbol{\beta})} \left[ |e|^r \right]\\
 &\leq  \left[\delta \left(\sigma+ \rho(\boldsymbol{\gamma})\right)\right]^r,
\end{align*}
where the last inequality follows because $$\big| \big( \boldsymbol{\beta}^* - \frac{\boldsymbol{\beta}}{\boldsymbol{\beta}^{\top} \boldsymbol{\beta}} \rho^*(\boldsymbol{\beta}^*) \big)^{\top} \boldsymbol{\gamma}\big|\le \rho(\boldsymbol{\gamma}),$$ for any $\boldsymbol{\gamma} \in \R^d$ by the assumption in \eqref{equation:condition} and since  $\mathbb{E}_{\mathbb{P}}[ |e|^r ] = \delta^r$.

Thus, we only need to compute $\mathbb{E}_{\mathbb{P}^*} [| Y - \mathbf{X}^{\top}\boldsymbol{\beta}|^r ]=\mathbb{E}_{\pi^*(\boldsymbol\beta)} [| \widetilde{Y} - \widetilde{\mathbf{X}}^{\top}\boldsymbol{\beta}|^r ]$.
Adding and subtracting $\mathbf{X}^{\top} \beta$ and $Y$ to $\widetilde{Y} - \widetilde{\mathbf{X}}^{\top}\boldsymbol{\beta}$ we have from \eqref{equation:worst-case-dist-rho}
\begin{equation}
\begin{split}
\label{eq:expand}
 \widetilde{Y} - \widetilde{\mathbf{X}}^{\top}\boldsymbol{\beta} &= \widetilde{Y}-Y + Y - \mathbf{X}^{\top}\boldsymbol{\beta}  +  (\mathbf{X} - \widetilde{\mathbf{X}})^{\top}\boldsymbol{\beta}=  \left(Y - \mathbf{X}^{\top}\boldsymbol{\beta}\right) + e \left(\sigma +\rho(\boldsymbol{\beta})\right),
\end{split}
\end{equation}
where the last term applies \cite[Theorem 23.5, p.\ 218]{rockafellar2015convex}, which shows that for any proper, convex function $\boldsymbol{\beta}^* \in \partial \rho(\boldsymbol{\beta})$ if and only if
\[ (\boldsymbol{\beta}^*)^{\top}\boldsymbol{\beta} - \rho^*(\boldsymbol{\beta}^*) = \rho(\boldsymbol{\beta});\]
hence, 
\[\left(\mathbf{X} - \widetilde{\mathbf{X}}\right)^{\top}\boldsymbol{\beta}=  e \left( \boldsymbol{\beta}^* - \frac{\boldsymbol{\beta}}{\boldsymbol{\beta}^{\top} \boldsymbol{\beta}} \rho^*(\boldsymbol{\beta}^*) \right)^{\top}\boldsymbol{\beta} = e\rho\left(\boldsymbol{\beta}\right).\]
Therefore, using \eqref{eq:expand} and
writing $(Y- \mathbf{X}^{\top}\boldsymbol{\beta})$ as $e \, \sqrt[r]{\mathbb{E}_{\mathbb{P}}[ |Y- \mathbf{X}^{\top}\boldsymbol{\beta}|^r ]}/\delta$, we have that
\begin{align*}
    \mathbb{E}_{\mathbb{P}^*} \left[\left| Y - \mathbf{X}^{\top}\boldsymbol{\beta} \right|^r \right] &=  \mathbb{E}_{\pi^*(\boldsymbol\beta)} \left[\left| \big(Y - \mathbf{X}^{\top}\boldsymbol{\beta}) + e (\sigma+\rho(\boldsymbol{\beta}))\right|^r\right] \\\
    &=  \bigg| \frac{1}{\delta}\sqrt[r]{\mathbb{E}_{\mathbb{P}}[ |Y- \mathbf{X}^{\top}\boldsymbol{\beta}|^r ]}+ (\sigma+\rho(\boldsymbol{\beta}))\bigg|^r\, \mathbb{E}_{\mathbb{P}} \left[|e|^r\right] \\
    &= \left| \sqrt[r]{\mathbb{E}_{\mathbb{P}}\left[ |Y- \mathbf{X}^{\top}\boldsymbol{\beta}|^r \right]} + \delta (\sigma+\rho(\boldsymbol{\beta}))\right|^r. 
\end{align*}
In the final step above, we again used that $\mathbb{E}_{\mathbb{P}}[ |e|^r ] = \delta^r$.
\end{itemize} 

\subsection{Proof of Theorem 2}\label{proof:theorem2}

We first recall the representations for the one-dimensional Wasserstein distance
\begin{align}
    \mathcal{W}_r(\P_{\boldsymbol{\gamma},n}, \P_{\boldsymbol{\gamma}})^r
&=\int_0^1 \left|F^{-1}_{\boldsymbol{\gamma},n}(p)-F^{-1}_{\boldsymbol{\gamma}}(p)\right|^r\,dp,
\label{eq:Wass1_secproof}
\end{align}
for $r\ge 1$ and
\begin{align}
   \mathcal{W}_1(\P_{\boldsymbol{\gamma},n}, \P_{\boldsymbol{\gamma}})
&=\int_{\R} \left|F_{\boldsymbol{\gamma},n}(t)-F_{\boldsymbol{\gamma}}(t)\right|\,dt,
\label{eq:Wass2_secproof}
\end{align}
see e.g.\ \cite[Theorem 2.9, Theorem 2.10]{bobkov2019one}. We also note that $\mathcal{W}_r$ is translation invariant, which implies in particular that 
\begin{align*}
&\mathcal{W}_r\left(\P_{\boldsymbol{\gamma},n}, \, \P_{\boldsymbol{\gamma}}\right)= \\\
& \qquad \mathcal{W}_r\left(\left[\left((\mathbf{X},Y)-(\mathbf{x_0},y_0)\right) ^{\top} \boldsymbol{\gamma} \right]_* \,\P_{\boldsymbol{\gamma},n}, \,\left[\left((\mathbf{X},Y)-(\mathbf{x_0}, y_0)\right)^{\top} \boldsymbol{\gamma} \right]_* \,\P_{\boldsymbol{\gamma}}\right),
\end{align*}
for any $\mathbf{x_0}\in \R^d$ and $y_0\in \R$. Defining $c:=\mathrm{diam}\left( \text{supp}(\P)\right)$, there is thus no loss of generality if we assume that 
\begin{align}\label{eq:support}
\text{supp}(\P_{\boldsymbol{\gamma}})\subseteq \left[0, \, c\right].
\end{align}

Noting that $|F^{-1}_{\boldsymbol{\gamma}, n}(p)-F^{-1}_{\boldsymbol{\gamma}}(p)|\le c$ for all $p\in (0,1)$, we estimate 
\begin{align*}
\overline{\mathcal{W}}_r(\P_{\boldsymbol{\gamma},n}, \P_{\boldsymbol{\gamma}})^r
&=\sup_{ \|\boldsymbol{\gamma}\|=1 } \left(\int_0^1 \left|F^{-1}_{\boldsymbol{\gamma},n}(p)-F^{-1}_{\boldsymbol{\gamma}}(p)\right|^r\,dp\right)\\
&\le c^{r-1}\sup_{ \|\boldsymbol{\gamma}\|=1 } \left(\int_0^1 \left|F^{-1}_{\boldsymbol{\gamma},n}(p)-F^{-1}_{\boldsymbol{\gamma}}(p) \right|\,dp\right)\\
&=c^{r-1}\sup_{ \|\boldsymbol{\gamma}\|=1 } \left(\int_{\R}  \left|F_{\boldsymbol{\gamma},n}(t)-F_{\boldsymbol{\gamma}}(t) \right|\,dt\right),
\end{align*}
where the final inequality follows from \eqref{eq:Wass1_secproof} and \eqref{eq:Wass2_secproof}.
Next, recalling \eqref{eq:support},
\begin{align*}
\sup_{ \|\boldsymbol{\gamma}\|=1 } \int_{\R}  \left|F_{\boldsymbol{\gamma},n}(t)-F_{\boldsymbol{\gamma}}(t) \right|\,dt
&\le \sup_{ \|\boldsymbol{\gamma}\|=1 } \int_{0}^{c} \sup_t   \left|F_{\boldsymbol{\gamma}, n}(t)-F_{\boldsymbol{\gamma}}(t) \right|\, ds\le c\sup_{f\in \mathcal{H}}   \left|\E_{\P_n}[f]-\E_{\P}[f] \right|,
\end{align*}
where
\[ \mathcal{H}:= \left\{\mathds{1}_{\left\{\mathbf{x}^\top \boldsymbol{\gamma} \, \le \, t \right\}}:  \boldsymbol{\gamma}\in \mathbb{R}^{d+1}, \, t\in \R \right\}. \]
The claim now follows from Lemma \ref{lem:3} in Section \ref{sec:aux_results_proofs}.

\subsection{Proof of Theorem 3}\label{proof:theorem3}

By Lemma \ref{lem:key} in Section \ref{sec:aux_results_proofs} with $k=\log\left(2n+1\right)^{1/s}$ we have 
\begin{equation}
 \W_r(\P_n,\P)^r \le 2^rr\log \left(2n+1 \right)^{r/s} \left ( I_1 + \frac{\sqrt{\Gamma\vee\Gamma_n}}{s/2-r}\log\left(2n+1\right)^{-1/2} I_2 \right)~,
\end{equation}
where
\begin{align*}
    I_1 &= \sup_{(\boldsymbol{\gamma},  \, t) \, \in \, \R^{d+1}\times \R} \left|F_{\boldsymbol{\gamma}}(t)-F_{\boldsymbol{\gamma}, n}(t))\right|~,\\
    I_2 &=  \sup_{(\boldsymbol{\gamma}, \, t)  \, \in \, \R^{d+1}\times \R} \frac{(F_{\boldsymbol{\gamma}}(t)-F_{\boldsymbol{\gamma}, n}(t))^+}{\sqrt{F_{\boldsymbol{\gamma}}(t)(1-F_{\boldsymbol{\gamma},n}(t))}} + \sup_{(\boldsymbol{\gamma}, \, t)  \, \in \,  \R^{d+1}\times \R} \frac{(F_{\boldsymbol{\gamma}, n}(t)-F_{\boldsymbol{\gamma}}(t))^+}{\sqrt{F_{\boldsymbol{\gamma}, n}(t)(1-F_{\boldsymbol{\gamma}}(t))}}~,\\
    \Gamma_n &= \sup_{ \|\boldsymbol{\gamma}\|=1 }\mathbb{E}_{\P_n}\big[|(\mathbf{X},Y)^\top\boldsymbol{\gamma}|^s\big] = \sup_{ \|\boldsymbol{\gamma}\|=1 }\frac{1}{n}\sum_{i=1}^n  \big|(\mathbf{X}_i,Y_i)^\top \boldsymbol\gamma\big|^s~,\\
    \Gamma &=\E_{\P}\left[  \big\|(\mathbf{X},Y ) \|_2^{s}\right]=\E_{\P}\Big[ \sup_{ \|\boldsymbol{\gamma}\|=1 } |(\mathbf{X},Y)^\top \boldsymbol{\gamma} |^{s}\Big]~.
\end{align*}

Next, by Markov's inequality and the triangle inequality
\begin{align*}
\P\left( \Gamma_n \ge C \right)&\le \frac{\E_\P[\Gamma_n]}{C} =
\frac{1}{C}\E_{\P} \left[\sup_{ \|\boldsymbol{\gamma}\|=1 }\frac{1}{n}\sum_{i=1}^n  \left|(\mathbf{X}_i,Y_i)^\top \boldsymbol\gamma\right|^s \right]
\le \frac{\Gamma}{C}.
\end{align*}

Setting the last expression equal to $\alpha$ yields $\Gamma_n\le \Gamma/\alpha$ on a set of probability at least $1-\alpha.$ Combining this with Lemma \ref{lem:3} (to control $I_1$) and Lemma \ref{lem:suboptimal} (to control $I_2$) yields that $\W_r(\P_n,\P)^r$ is less than or equal to the following with probability greater than $1-3\alpha$:
\begin{align*}
\begin{split}
&2^rr\log\left(2n+1\right)^{\frac{r}{s}} \Bigg[\frac{1}{\sqrt{n}}\Bigg(
180\sqrt{d+2}+\sqrt{2\log\left(\frac{1}{\alpha}\right)}\Bigg)+\sqrt{\frac{\Gamma}{\alpha}} \frac{1}{s/2-r}  \frac{1}{\sqrt{\log\left(2n+1\right)}}I_2\Bigg]\\
&\le \frac{2^rr\log\left(2n+1\right)^{\frac{r}{s}}}{\sqrt{n}} \Bigg[
180\sqrt{d+2}+\sqrt{2\log\left(\frac{1}{\alpha}\right)}+\sqrt{\frac{\Gamma}{\alpha}} \frac{8}{s/2-r} \sqrt{\log\left(\frac{8}{\alpha}\right) + (d+2)}\Bigg],
\end{split}
\end{align*}
which is the claim.

\subsection{Proof of Theorem \ref{thm:ass_bdd}}

This claim follows from the estimate 
\begin{align*}
\mathcal{W}_r(\P_{\boldsymbol{\gamma}, n},\P_{\boldsymbol{\gamma}})^r\le \mathrm{diam}(\mathrm{supp}(\P))^r \sup_{f\in \mathcal{H}^0} |\E_{\P_n}[f]-\E_{\P}[f]|,
\end{align*}
stated in the proof of Theorem \ref{thm:1} together with
\begin{align*}
 \sqrt{n}\sup_{f \in \mathcal{H}^0} |\E_{\P_n}[f]-\E_{\P}[f]|  \Rightarrow \sup_{f \in  \mathcal{H}^0} |G_f|,
\end{align*}
as in the proof of Theorem \ref{thm:ass}. 
As $\sup_{f \in  \mathcal{H}^0} |G_f|$ dominates $\sup_{t\in [0,1]} |B(t)|$ in stochastic order, this concludes the proof.

\subsection{Proof of Theorem \ref{thm:ass}}\label{proof:theorem5}

Note that again by \cite[Proposition 7.14]{bobkov2019one} we have
\begin{align*}
\mathcal{W}_r(\P_{\boldsymbol{\gamma}, n},\P_{\boldsymbol{\gamma}})^r &\le r2^{r-1}\int |t|^{r-1} |F_{\boldsymbol{\gamma}, n}(t)-F_{\boldsymbol{\gamma}}(t)|\,dt\\
&= r2^{r-1}\Big( \int_0^\infty |t|^{r-1} |(1-F_{\boldsymbol{\gamma}, n}(t))-(1-F_{\boldsymbol{\gamma}}(t))|\,dt
+\int_{-\infty}^0 |t|^{r-1} |F_{\boldsymbol{\gamma}, n}(t)-F_{\boldsymbol{\gamma}}(t)|\,dt\Big)\\
&\le r2^{r-1} \sup_{f\in \mathcal{H}^+\cup \mathcal{H}^0\cup \mathcal{H}^-}
|\E_{\P_n}[f]-\E_{\P}[f]| \int (1\wedge |t|^{r-s-1}) \,dt\\
&\le c \sup_{f\in \mathcal{H}^+\cup \mathcal{H}^0\cup \mathcal{H}^-} |\E_{\P_n}[f]-\E_{\P}[f]|,
\end{align*}
where $c:=r2^{r-1} \int (1\wedge |t|^{r-s-1}) \,dt. $
We next find an envelope $F$ for $\mathcal{H}^+\cup \mathcal{H}^0\cup \mathcal{H}^-$: it is easy to see that 
\begin{align*}
\sup_{f\in \mathcal{H}^+} |f(\boldsymbol x)|\le \sup_{\|\boldsymbol{\gamma}\|=1}
|\boldsymbol x^\top \boldsymbol \gamma|^{s}\le \|\boldsymbol x\|_*^{s}.
\end{align*}
A similar argument for $\mathcal{H}^-$ yields
\begin{align*}
 F(\boldsymbol x):=\sup_{f\in \mathcal{H}^+\cup \mathcal{H}^0\cup \mathcal{H}^-}   |f(\boldsymbol x)|\le \|\boldsymbol x\|_*^{s} \vee 1.
\end{align*}

As $\mathcal{H}^0$ is VC-subgraph by Lemma \ref{lem:3}, [Van der Vaart, Wellner, Lemma 2.6.22] implies that $\mathcal{H}^+$ and $\mathcal{H}^-$ are also VC-subgraph: indeed note that
\begin{align*}
    \{(\boldsymbol x,u):\  u\le |t|^{s} \mathds{1}_{\{t\le \boldsymbol x^\top \boldsymbol\gamma \}} \}
    &=  \{(\boldsymbol x,u):\ t\le  \boldsymbol x^\top \boldsymbol\gamma,  u\le |t|^{s} \} \cup \{(\boldsymbol x,u):\ t> \boldsymbol x^\top \boldsymbol\gamma \} \\
    &= \{(\boldsymbol x,u):\ t\le \boldsymbol x^\top \boldsymbol\gamma \} \cap \{(\boldsymbol x,u):\ u\le |t|^{s} \}  \cup \{(\boldsymbol x,u):\ t> \boldsymbol x^\top \boldsymbol\gamma \},
\end{align*}
so the claim follows from the fact that $\mathcal{H}$ is VC, finite dimensional vector spaces of functions are VC subgraph  [Van der Vaart, Wellner, Lemma 2.6.15], and [Van der Vaart, Wellner, Lemma 2.6.17 (ii), (iii)].
Then, [Van der Vaart, Wellner, Theorem 2.6.7] states that for all $\epsilon\in (0,1),$
\begin{align*}
N(\epsilon \|F\|_{\Q,2},    \mathcal{H}^+\cup \mathcal{H}^0\cup \mathcal{H}^-, L_2(\Q)) \le C_1\Big(\frac{1}{\epsilon}\Big)^{2C_2-1},
\end{align*}
for universal constants $C_1, C_2>1$ and any probability measure $\Q$, for which $\|F\|_{\Q,2}>0.$ Thus,
\begin{align*}
\int_0^\infty \sup_{\Q} \sqrt{\log N(\epsilon \|F\|_{\Q,2},    \mathcal{H}^+\cup \mathcal{H}^0\cup \mathcal{H}^-, L_2(\Q))} \,d\epsilon<\infty,
\end{align*}
and together with $\Gamma<\infty$, [Van der Vaart, Wellner, Theorem 2.5.2] implies that $\mathcal{H}^+\cup \mathcal{H}^0\cup \mathcal{H}^-$ is Donsker. Thus, the convergence in distribution
\begin{align*}
\sqrt{n}\sup_{f \in \mathcal{H}^+\cup \mathcal{H}^0\cup \mathcal{H}^-} |\E_{\P_n}[f]-\E_{\P}[f]|  \Rightarrow \sup_{f \in \mathcal{H}^+\cup \mathcal{H}^0\cup \mathcal{H}^-} |G_f|,
\end{align*}
holds, where $(G_f)$ is a zero-mean Gaussian process satisfying 
\begin{align*}
\E[G_{f_1} G_{f_2}] = \E_{\P}[f_1 f_2] -\E_{\P}[f_1]\E_{\P}[f_2],
\end{align*}
for any $f_1, f_2 \in \mathcal{H}^+\cup \mathcal{H}^0\cup \mathcal{H}^-.$ Next, from the proof of [Van der Vaart, Wellner, Theorem 2.5.2] we obtain the inequality 
\begin{align*}
&\E_{\P} \left[ \sqrt{n}\sup_{f \in \mathcal{H}^+\cup \mathcal{H}^0\cup \mathcal{H}^-} \left|\E_{\P_n}[f]-\E_{\P}[f]\right| \right]\\
&\qquad \le C \sqrt{\Gamma} \int_0^\infty \sup_{\Q} \sqrt{\log N\left(\epsilon \|F\|_{\Q,2},    \mathcal{H}^+\cup \mathcal{H}^0\cup \mathcal{H}^-, L_2(\Q)\right)} \,d\epsilon.
\end{align*}
This shows the second claim.

\subsection{Proof of Theorem \ref{cor:generalization}}\label{proof:theorem6}

Note that $\iW_r(\P_n, \Q) \le \iW_r(\P_n, \P) + \iW_r(\P, \Q)$ by the triangle inequality because $\iW_r$ is a metric for any norm $\rho$. Then,
\begin{equation*}
    E_n := \left\{\iW_r(\P_n, \Q) > \delta_{n,r} + \epsilon \right \}  \subset \left \{\iW_r(\P_n, \P) > \delta_{n,r} \right\}  \subset \left\{\W_r(\P_n, \P) > \delta_{n,r}/c_{\rho, d} \right\}~,
\end{equation*}
which implies that the probability of $ E_n^c $ is greater than $1-3\alpha$ due to Theorem \ref{thm:rates} (or \ref{thm:1}). In the equation above, $c_{\rho, d}$ is defined via \eqref{eq:bar_W}. Finally, on the event $E_n^c$, we have
\begin{align*}
    \E_{\Q} \left[\left|Y-\mathbf{X}^\top {\boldsymbol \beta}\right|^r\right]^{1/r} &\le \sup_{\tilde{\P}\in B_{\delta_{n,r}+\epsilon}(\P_n)} \E_{\tilde{\P}} \left[\left|Y-\mathbf{X}^\top {\boldsymbol \beta}\right|^r\right]^{1/r} \\
    &= \E_{{\P_n}} \left[\left|Y-\mathbf{X}^\top {\boldsymbol \beta}\right|^r\right]^{1/r} + \left(\delta_{n,r} + \epsilon \right) \left(\sigma+\rho({\boldsymbol \beta}) \right),~ \quad \forall {\boldsymbol \beta}~,
\end{align*}
where the last equality follows from Theorem \ref{thm:penalized}.

\subsection{Proof of Theorem \ref{thm:pivotal}}\label{proof:theorem7}
 
The proof has three steps. The first two steps adapt what we learn in Section \ref{sec:rates} to our particular setup. The last step concludes based on observations about Theorems \ref{thm:penalized} and \ref{thm:rates}.

\textbf{Step 1:} Let us compare the $\rho$-MSW metric to the MSW metric using reasoning that is similar to our derivations in \eqref{eq:proj_wass_rep} -- \eqref{eq:bar_W}. Defining $\boldsymbol{\bar{\gamma}}_\sigma^\top = (\boldsymbol{\gamma}^\top,-\sigma)$, we obtain
\begin{align*}
\iW_r(\P, \tilde \P)&=\sup_{\boldsymbol{\gamma} \in  \mathbb{R}^d}
\frac{\|\boldsymbol{\bar{\gamma}}_\sigma\|}{\sigma+\rho(\boldsymbol{\gamma})}\frac{1}{\|\boldsymbol{\bar{\gamma}}_\sigma\|}
\mathcal{W}_r\left(\left[(\mathbf{X},Y/\sigma)^\top\boldsymbol{\bar{\gamma}}_\sigma\right]_*\mathbb{P},\,  [(\widetilde{\mathbf{X}},\tilde Y/\sigma)^\top\boldsymbol{\bar{\gamma}}_\sigma]_*\widetilde{\mathbb{P}}\right)\\
&\le \left( \max \left\{1/c_d, 1 \right\} \right) \sup_{ \|\boldsymbol{\gamma}\| = 1}
\mathcal{W}_r\left( \left[(\mathbf{X},Y/\sigma)^\top\boldsymbol{\gamma} \right]_*\mathbb{P},\,  [(\widetilde{\mathbf{X}},\tilde Y/\sigma)^\top \boldsymbol{\gamma}]_*\widetilde{\mathbb{P}}\right)\\
&= \left( \max \left\{1/c_d, 1 \right\} \right) \W_r \left(\P^\sigma, \, \tilde\P^\sigma \right),
\end{align*}
where $\P^\sigma:= (\mathbf{X},Y/\sigma)_*\P$ and $\tilde{\P}^\sigma:= (\mathbf{X},Y/\sigma)_*\tilde{\P}$.

\textbf{Step 2:} Let us apply Theorem \ref{thm:rates} to compute the rates for $\W_r(\P^\sigma, \tilde\P^\sigma)$, which depend on $ \E_{\P^\sigma}\left[  \left\|(\mathbf{X},Y) \right\|_*^{s}\right]$. Consider the following derivation 
\begin{equation*}
\E_{\P^\sigma}\left[  \left\|(\mathbf{X},Y) \right\|_*^{s}\right]  \le 2^{s-1} \left( \E_{\P^\sigma}\left[  \left\|(\mathbf{X},0)\right\|_*^{s}\right] + \E_{\P^\sigma}\left[  \left \|(0, \dots,0,Y) \right\|_*^{s}\right] \right)~.
\end{equation*}
Note that $\E_{\P^\sigma}\left[  \left\|(\mathbf{X},0)\right\|_*^{s}\right] = \E_{\P}\left[  \left\|(\mathbf{X},0) \right\|_*^{s}\right] =1$ and $$\E_{\P^\sigma}\left[  \left \|(0, \dots,0,Y) \right\|_*^{s}\right] = \E_{\P}\left[  \left \|(0, \dots, 0,Y/\sigma) \right\|_*^{s}\right] \le 1,$$ due to our assumptions. This implies that
\begin{equation*}
    \E_{\P^\sigma}\left[  \left\|(\mathbf{X},Y) \right\|_*^{s}\right] \le 2^{s}~.
\end{equation*}

\textbf{Step 3:} We note that Theorem \ref{thm:rates} still holds for any $\Gamma$ larger than $ \E_{\P}\left[  \left\|(\mathbf{X},Y) \right\|_2^{s}\right]$. In particular, we can consider $\Gamma = 2^{s}$ due to Step 2. In addition, we note that the conclusion of Theorem \ref{thm:penalized} is unaffected by the choice of $\sigma\ge 1$. These observations and the same argument presented in the proof of Theorem~\ref{cor:generalization} conclude our proof.  

\subsection{Proof of Corollary 2}

By Theorem \ref{eq:penalized} and conditions (i),\,(ii) above, \eqref{eq:H0} is equivalent to
    $$  \E_{\P} \left[\left|Y-\mathbf{X}^\top {\boldsymbol \beta}_1\right|^r\right]^{1/r} +  \delta \rho({\boldsymbol \beta_1}) \le \E_{\P} \left[\left|Y-\mathbf{X}^\top {\boldsymbol \beta}_2\right|^r\right]^{1/r} +  \delta \rho({\boldsymbol \beta_2})~. $$
The previous expression is equivalent to
$$ n^{-1/(2r)}\left(2\sigma + \rho({\boldsymbol \beta_1}) + \rho({\boldsymbol \beta_2})\right)T_n \le \Delta_n( {\boldsymbol \beta}_2 ) - \Delta_n( {\boldsymbol \beta}_1 )~,$$
where $\Delta_n( {\boldsymbol \beta}) = \E_{\P} \left[\left|Y-\mathbf{X}^\top {\boldsymbol \beta}\right|^r\right]^{1/r} - \E_{\P_n} \left[\left|Y-\mathbf{X}^\top {\boldsymbol \beta}\right|^r\right]^{1/r}$. By definition of $\iW_r$ and Theorem \ref{thm:penalized}, it follows that
$$ \frac{\Delta_n( {\boldsymbol \beta}_2)}{\sigma + \rho({\boldsymbol \beta_2})}  \le \iW_r(\P,\P_n),  \quad \text{and} \quad \frac{-\Delta_n( {\boldsymbol \beta}_1)}{\sigma + \rho({\boldsymbol \beta_1})}  \le \iW_r(\P_n,\P).$$
Therefore we have
$$ n^{-1/(2r)}\left(2\sigma + \rho({\boldsymbol \beta_1}) + \rho({\boldsymbol \beta_2})\right)T_n \le \iW_r(\P,\P_n) \left(2\sigma + \rho({\boldsymbol \beta_1})+ \rho({\boldsymbol \beta_2}) \right)~. $$
Using $\iW_r(\P,\P_n) \le c_{\rho,d} \W_r(\P_n, \P)$ from \eqref{eq:bar_W}, we derive that
$$  \P \left( T_n > c_{\rho,d} C^{1/r} \right) \le \P \left( c_{\rho,d} \W_r(\P,\P_n)  > c_{\rho,d} \, C^{1/r} n^{-1/(2r)} \right)~.  $$
Finally, Theorem  \ref{thm:1} implies that the above probability is bounded by $\alpha$.

\section{Auxiliary Results} \label{sec:aux_results_proofs}

We start with a preliminary discussion of $\iW_r$.

\begin{lemma}\label{lem:metric_new}
    Suppose that $\rho$ is a norm on $\mathbb{R}^d$. Let $\sigma>0$ and define the norm $\|\cdot\|_{\rho,\sigma}$ via   $$|\Tilde{\boldsymbol{\gamma}}|_{\rho,\sigma} := \sigma |\gamma_{d+1}| + \rho(\boldsymbol{\gamma}),$$ for $\tilde{\boldsymbol{\gamma}}=(\boldsymbol{\gamma}, \gamma_{d+1})$, where $\boldsymbol{\gamma}\in \mathbb R^d$ and $\gamma\in \R$. Then 
    $$ \widehat{W}_{r,\rho,\sigma}(\mathbb{P},\tilde{\mathbb{P}}) = \sup_{ \tilde{\boldsymbol{\gamma}} \in \mathbb{R}^{d+1} : \|\tilde{\boldsymbol{\gamma}}\|_{\rho,\sigma} = 1 }   \mathcal{W}_r( \tilde{\boldsymbol{\gamma}}_{*} \mathbb{P} ~, \tilde{\boldsymbol{\gamma}}_{*} \tilde{\mathbb{P}} )~,$$
    and there exist positive constants $c_1$ and $c_2$ (may depend on the dimension $d$) such that
    $$ c_1 \widehat{W}_{r,\rho,\sigma}(\mathbb{P},\tilde{\mathbb{P}}) \le \overline{\mathcal{W}}_{r}(\mathbb{P},\tilde{\mathbb{P}})  \le c_2 \widehat{W}_{r,\rho,\sigma}(\mathbb{P},\tilde{\mathbb{P}}).$$    
\end{lemma}
\begin{proof}
    Denote $\mathbf{Z} = (\mathbf{X}^\top,Y)$ and $ \Tilde{\mathbf{Z}} = (\Tilde{\mathbf{X}}^\top, \Tilde{Y})$.  By definition 
    \begin{align*}
        \widehat{W}_{r,\rho,\sigma}(\mathbb{P},\tilde{\mathbb{P}})  &= \sup_{\boldsymbol{\gamma} \in \mathbb{R}^d } \frac{1}{\sigma + \rho(\boldsymbol{\gamma})} \inf_{\pi\in \Pi(\P,\Q) } \E_{\pi} \left[ | \mathbf{Z}^\top (\boldsymbol{\gamma},-1) -  \Tilde{\mathbf{Z}}^\top (\boldsymbol{\gamma},-1)|^r \right]^{1/r} \\
        &= \sup_{\boldsymbol{\gamma}\in \mathbb{R}^d } \frac{1}{| (\boldsymbol{\gamma},-1)|_{\rho,\sigma}} \inf_{\pi \in \Pi(\P,\Q)} \E_{\pi} \left[ | \mathbf{Z}^\top (\boldsymbol{\gamma},-1) -  \Tilde{\mathbf{Z}}^\top (\boldsymbol{\gamma},-1) |^r \right]^{1/r} \\
        &= \sup_{\boldsymbol{\gamma} \in \mathbb{R}^d }  \inf_{\pi\in \Pi(\P,\Q) } \E_{\pi} \left[ \left| (\mathbf{Z} -\Tilde{\mathbf{Z}})^\top \left(\frac{\boldsymbol{\gamma}}{\| (\boldsymbol{\gamma},-1)\|_{\rho,\sigma}},\frac{-1}{\| (\boldsymbol{\gamma},-1)\|_{\rho,\sigma}} \right ) \right|^r  \right]^{1/r} \\
        &\le \sup_{ \tilde{\boldsymbol{\gamma}}=(\boldsymbol{\gamma}, \gamma_{d+1}) \in \mathbb{R}^{d+1} : \|\tilde{\boldsymbol{\gamma}}\|_{\rho,\sigma} = 1 } \inf_{\pi \in \Pi(\P,\Q)} \E_{\pi} \left[ | \mathbf{Z}^\top  \Tilde{\boldsymbol{\gamma}} - \Tilde{\mathbf{Z}}^\top \Tilde{\boldsymbol{\gamma}} |^{r}  \right]^{1/r}  \\
        &= \sup_{ \tilde{\boldsymbol{\gamma}}=(\boldsymbol{\gamma}, \gamma_{d+1}) \in \mathbb{R}^{d+1} : \|\tilde{\boldsymbol{\gamma}}\|_{\rho,\sigma} = 1 }   \mathcal{W}_r( \tilde{\boldsymbol{\gamma}}_{*} \mathbb{P} ~, \tilde{\boldsymbol{\gamma}}_{*} \tilde{\mathbb{P}} ) \\
        &= \sup_{ \tilde{\boldsymbol{\gamma}}=(\boldsymbol{\gamma}, \gamma_{d+1}) \in \mathbb{R}^{d+1} } \frac{1}{ |\gamma_{d+1}| (\sigma + \rho(\boldsymbol{\gamma}/|\gamma_{d+1}|))} \inf_{\pi } \E_{\pi} \left[ | \mathbf{Z}^\top  \Tilde{\boldsymbol{\gamma}} - \Tilde{\mathbf{Z}}^\top \Tilde{\boldsymbol{\gamma}} |^{r}  \right]^{1/r} \\
        &= \sup_{ \tilde{\boldsymbol{\gamma}}=(\boldsymbol{\gamma}, \gamma_{d+1}) \in \mathbb{R}^{d+1} } \frac{1}{ (\sigma + \rho(\boldsymbol{\gamma}/|\gamma_{d+1}|))} \inf_{\pi\in \Pi(\P,\Q) } \E_{\pi} \left[ \left| (\mathbf{Z} - \Tilde{\mathbf{Z}}) ^\top \left(\frac{\boldsymbol{\gamma}}{|\gamma_{d+1}|},\frac{\gamma_{d+1}}{| \gamma_{d+1}| } \right )   \right|^{r}  \right]^{1/r} \\
        & \le \widehat{W}_{r,\rho,\sigma}(\mathbb{P},\tilde{\mathbb{P}})~.
    \end{align*}
    This proves our first result. Now, to prove the second result we rely on the following representations.
\begin{align*}
    \overline{\mathcal{W}}_{r}(\mathbb{P},\tilde{\mathbb{P}}) &= \sup_{ \tilde{\boldsymbol{\gamma}} \in \mathbb{R}^{d+1} : \|\tilde{\boldsymbol{\gamma}}\|_{2} = 1 }   \mathcal{W}_r( \tilde{\boldsymbol{\gamma}}_{*} \mathbb{P} ~, \tilde{\boldsymbol{\gamma}}_{*} \tilde{\mathbb{P}} ) = \sup_{ \tilde{\boldsymbol{\gamma}} \in \mathbb{R}^{d+1}} \frac{1}{\|\tilde{\boldsymbol{\gamma}}\|_2}  \mathcal{W}_r( \tilde{\boldsymbol{\gamma}}_{*} \mathbb{P} ~, \tilde{\boldsymbol{\gamma}}_{*} \tilde{\mathbb{P}} ), 
\end{align*}
and
\begin{align*}
    \widehat{W}_{r,\rho,\sigma}(\mathbb{P},\tilde{\mathbb{P}}) = \sup_{ \tilde{\boldsymbol{\gamma}} \in \mathbb{R}^{d+1}} \frac{1}{\|\tilde{\boldsymbol{\gamma}}\|_{\rho,\sigma}}  \mathcal{W}_r( \tilde{\boldsymbol{\gamma}}_{*} \mathbb{P} ~, \tilde{\boldsymbol{\gamma}}_{*} \tilde{\mathbb{P}} )~.
\end{align*}
Because $\|\cdot\|_2$ and $\|\cdot\|_{\rho,\sigma}$ are norms on $\mathbb{R}^{d+1}$, it follows that there exist positive constants $c_1$ and $c_2$ such that
$$ \frac{c_1}{\|\tilde{\boldsymbol{\gamma}}\|_{\rho,\sigma}} \le \frac{1}{\|\tilde{\boldsymbol{\gamma}}\|_{2}} \le \frac{c_2}{\|\tilde{\boldsymbol{\gamma}}\|_{\rho,\sigma}}~,\quad \forall \Tilde{\boldsymbol{\gamma}} \in \mathbb{R}^{d+1}~.$$
We conclude the second result by using the previous inequality and the representations for $\overline{\mathcal{W}}_{r}(\mathbb{P},\tilde{\mathbb{P}})$ and $\widehat{W}_{r,\rho,\sigma}(\mathbb{P},\tilde{\mathbb{P}})$ presented above.
\end{proof}

More generally, $\iW_r$ is always a metric, as the following lemma shows:

\begin{lemma}\label{lem:metric} 
The $\rho$-max-sliced Wasserstein $\iW_r$ distance is a metric.
\end{lemma}

\begin{proof}

Recall from \eqref{eq:proj_wass_rep} that
\begin{align*}
    \iW_{r}(\P, \tilde \P)=\sup_{\boldsymbol{\gamma} \in  \mathbb{R}^d} \, 
\frac{1}{\sigma+\rho(\boldsymbol{\gamma})} \,
\mathcal{W}_r\left(\left[(\mathbf{X},Y)^{\top}\boldsymbol{\bar{\gamma}}\right]_*\mathbb{P},\,  \left[(\widetilde{\mathbf{X}},\tilde Y)^{\top}\boldsymbol{\bar{\gamma}}\right]_*\widetilde{\mathbb{P}}\right),
\end{align*}
where $\boldsymbol{\bar{\gamma}}^\top = (\boldsymbol{\gamma}^\top,-1)$. Because the one-dimensional Wasserstein metric, $\mathcal{W}_r$, is non-negative, symmetric and satisfies the triangle inequality,  the same is true for $\iW_r$. It remains to show that $\iW_r(\P, \tilde\P)=0$ implies $\P=\tilde \P.$ For this, we first realize that because $\sigma+\rho(\boldsymbol{\gamma})>0$, it follows that $\iW_r(\P, \tilde\P)=0$ implies 
\begin{align}
\mathcal{W}_r\left(\left[(\mathbf{X},Y)^{\top}\boldsymbol{\bar{\gamma}}\right]_*\mathbb{P},\,  \left[(\widetilde{\mathbf{X}},\tilde Y)^{\top}\boldsymbol{\bar{\gamma}}\right]_*\widetilde{\mathbb{P}}\right)=0, \qquad \forall \, \boldsymbol{\gamma}\in \R^d.
\end{align}
Now, for any $\tilde{\boldsymbol{\gamma}}\in \R^{d+1}$ satisfying $\|\tilde {\boldsymbol{\gamma}} \|_2=1$ and $\tilde{\gamma}_{d+1}\le 0$, there exists a sequence $\{\boldsymbol{\gamma}_n\}_{n\in \mathbb N}$ in $\R^d$ such that 
\begin{align*}
\lim_{n\to \infty} \frac{\bar{\boldsymbol{\gamma}}_n}{\|\bar{\boldsymbol{\gamma}}_n\|_2} =\tilde{\boldsymbol{\gamma}},
\end{align*}
where again $\boldsymbol{\bar{\gamma}}_n^\top := (\boldsymbol{\gamma}_n^\top,-1)$. 
By continuity, this implies
\begin{align*}
\mathcal{W}_r\left(\left[(\mathbf{X},Y)^{\top}\boldsymbol{\tilde{\gamma}}\right]_*\mathbb{P},\,  \left[(\widetilde{\mathbf{X}},\tilde Y)^{\top}\boldsymbol{\tilde{\gamma}}\right]_*\widetilde{\mathbb{P}}\right)=0, \qquad \forall \, \tilde{\boldsymbol{\gamma}}\in \R^{d+1} \text{ with } \tilde{\gamma}_{d+1}\le 0;
\end{align*}
and because $\left[(\mathbf{X},Y)^{\top}\boldsymbol{\tilde{\gamma}}\right]_*\mathbb{P} = \left[(\mathbf{X},Y)^{\top}\boldsymbol{\tilde{\gamma}}\right]_*\widetilde{\mathbb{P}}$ implies $\left[-(\mathbf{X},Y)^{\top}\boldsymbol{\tilde{\gamma}}\right]_*\mathbb{P} = \left[-(\mathbf{X},Y)^{\top}\boldsymbol{\tilde{\gamma}}\right]_*\widetilde{\mathbb{P}}$, we have
\begin{align*}
 \mathcal{W}_r\left(\left[(\mathbf{X},Y)^{\top}\boldsymbol{\tilde{\gamma}}\right]_*\mathbb{P},\,  \left[(\widetilde{\mathbf{X}},\tilde Y)^{\top}\boldsymbol{\tilde{\gamma}}\right]_*\widetilde{\mathbb{P}}\right)=0, \qquad \forall \, \tilde{\boldsymbol{\gamma}}\in \R^{d+1}.   
\end{align*}
Positivity of $\iW_r$ now follows from the fact that $\mathcal{W}_r$ is positive. This concludes the proof.
\end{proof}

\begin{lemma}[{cf.\ \cite[Theorem 4.10]{wainwright2019high}, \cite[Chapter 4]{devroye2001combinatorial}}]
\label{lem:3}
Let
\begin{align}
\label{eq:H_def}
\mathcal{H}:= \left\{\mathds{1}_{\left\{\mathbf{x}^\top \boldsymbol{\gamma} \, \le \, t \right\}}:  \boldsymbol{\gamma}\in \mathbb{R}^{d+1}, \, t\in \R \right\}
\end{align}
be the set of indicator functions of half spaces.
Then, with probably at least $1-\alpha$,
\begin{align*}
 \sup_{(\boldsymbol{\gamma},t)\in \R^{d+1}\times\R}   \left|F_{\boldsymbol{\gamma}, n}(t)-F_{\boldsymbol{\gamma}}(t) \right| = \sup_{f\in \mathcal{H}} \left|\E_{\P_n}[f]-\E_{\P}[f]\right| \le  180\sqrt{\frac{d+2}{n}}+\sqrt{\frac{2}{n}\log\left(\frac{1}{\alpha}\right)}.
\end{align*}
\end{lemma}

\begin{proof}
By \cite[Theorem 4.10]{wainwright2019high}, we have
\begin{align*}
\P\left(\sup_{f\in \mathcal{H}} \left|\E_{\P_n}\left[f\right]-\E_{\P}\left[f\right] \right|> 2\mathcal{R}_n(\mathcal{H})+\epsilon\right) \le e^{-n\epsilon^2/2},
\end{align*}
where 
\begin{align*}
\mathcal{R}_n(\mathcal{H}):=\mathbb{E}_{\P, \varepsilon}\left[\sup _{f \in \mathcal{H}}\left|\frac{1}{n} \sum_{i=1}^n \varepsilon_i f\left(\mathbf{X}_i\right)\right|\right],
\end{align*}
is the Rademacher complexity of $\mathcal{H}$. Next, following \cite[statement and proof of Theorem 3.2]{devroye2001combinatorial}, we obtain
\begin{align}
\label{eq:Rn_bound}
\mathcal{R}_n(\mathcal{H})\le \frac{12}{\sqrt{n}} \max_{\mathbf{x}_1, \dots, \mathbf{x}_n\in \R^{d+1}} \int_0^1 \sqrt{2\log N \left(\textsf{r}, \mathcal{H}(\mathbf{x}_1^n) \right)}\,d\textsf{r},
\end{align}
where $\mathbf{x}_1^n:= \{\mathbf{x}_1, \dots, \mathbf{x}_n\}$ and 
\begin{align*}
\mathcal{H}(\mathbf{x}_1^n):=\left\{\left(f(\mathbf{x}_1), \dots, f(\mathbf{x}_n)\right): f\in \mathcal{H}\right\},
\end{align*}
and $N(\textsf{r},B)$ is defined as the cardinality of the smallest cover for any set $B\subseteq\{0,1\}^n$ of radius $\textsf{r}$ with respect to the distance 
\begin{align*}
\rho(\mathbf{b},\mathbf{d}):=\sqrt{\frac{1}{n} \sum_{i=1}^n \mathds{1}_{\{b_i\neq d_i\}}},
\end{align*}
where in the above, vectors $\mathbf{b},\mathbf{d}\in B$.
\cite[Theorem 4.3]{devroye2001combinatorial} states that
\begin{align}
\label{eq:bound1}
N(r, \mathcal{H}(\mathbf{x}_1^n)) \le \left( \frac{4e}{r^2} \right)^{V/(1-1/e)} =  \left( \frac{4e}{r^2} \right)^{Ve/(e-1)},
\end{align}
where $V$ is the VC-dimension of $\mathcal{H}$. Furthermore, by \cite[Corollary 4.2]{devroye2001combinatorial}, the VC-dimension of $\mathcal{H}$ is bounded by $d+2$, i.e.\ $V \leq d+2$. In conclusion, using \eqref{eq:bound1},
\begin{align}
\label{eq:bound2}
\log N\left(\textsf{r}, \mathcal{H}(\mathbf{x}_1^n\right)) \le \frac{eV}{e-1} \log\left(\frac{4e}{\textsf{r}^2}\right)\le \frac{e(d+2)}{e-1} \log\left(\frac{4e}{\textsf{r}^2}\right).
\end{align}
Following \cite[proof of Theorem 3.3]{devroye2001combinatorial} we estimate
\begin{align}
\label{eq:bound3}
\int_0^1 \sqrt{\log\left(\frac{4e}{\textsf{r}^2}\right) }\,d\textsf{r} \le \sqrt{2\pi e},
\end{align}
so that from \eqref{eq:bound2} and \eqref{eq:bound3}, we have
\begin{align*}
\int_0^1 \sqrt{2\log N(\textsf{r}, \mathcal{H}(\mathbf{x}_1^n))}\,d\textsf{r} \le 2e\sqrt{\frac{(d+2)\pi}{e-1}}\le 7.5\sqrt{d+2}.
\end{align*}
Using \eqref{eq:Rn_bound}, this yields
\begin{align*}
\mathcal{R}_n(\mathcal{H})\le 90\sqrt{\frac{d+2}{n}}.
\end{align*}
\end{proof}

\begin{lemma}\label{lem:key}
Define 
\begin{align*}
\Gamma_n:= \sup_{ \|\boldsymbol{\gamma}\|=1 }\mathbb{E}_{\P_n}\left [\left|(\mathbf{X},Y)^\top\boldsymbol{\gamma}\right|^s\right] = \sup_{ \|\boldsymbol{\gamma}\|=1 }\frac{1}{n}\sum_{i=1}^n  \left|(\mathbf{X},Y)_i^\top \boldsymbol\gamma\right|^s.
\end{align*}
For any $k \in \R^+$, we have
\begin{equation*}
\begin{split}
&\W_r(\P_n,\P)^r \le 
r(2k)^r\sup_{(\boldsymbol{\gamma},t)\in \R^{d+1}\times \R} \left|F_{\boldsymbol{\gamma},n}(t)-F_{\boldsymbol{\gamma}}(t))\right|\\
&+\frac{2^rr\sqrt{\Gamma \vee \Gamma_n}}{s/2-r}k^{r-s/2} \Bigg[ \sup_{(\boldsymbol{\gamma},t)\in \R^{d+1}\times \R} \frac{(F_{\boldsymbol{\gamma}}(t)-F_{\boldsymbol{\gamma}, n}(t))^+}{\sqrt{F_{\boldsymbol{\gamma}}(t)(1-F_{\boldsymbol{\gamma},n}(t))}} + \sup_{(\boldsymbol{\gamma},t)\in \R^{d+1}\times\R} \frac{(F_{\boldsymbol{\gamma}, n}(t)-F_{\boldsymbol{\gamma}}(t))^+}{\sqrt{F_{\boldsymbol{\gamma}, n}(t)(1-F_{\boldsymbol{\gamma}}(t))}}\Bigg],
\end{split}
\end{equation*}
with the convention that $0/0=0$ and the notation $x^+ := \max\{0,x\}$.
\end{lemma}

\begin{proof}
We first note that \cite[Proposition 7.14]{bobkov2019one} yields, for any $k>0$,
\begin{equation}
\begin{split}
\label{eq:Bobkov_bound}
    \mathcal{W}_r(\P_{\boldsymbol{\gamma}, n},\P_{\boldsymbol{\gamma}})^r &\le r2^{r-1}\int |t|^{r-1} |F_{\boldsymbol{\gamma}, n}(t)-F_{\boldsymbol{\gamma}}(t)|\,dt\\
    &\le  r(2k)^r \sup_t |F_{\boldsymbol{\gamma},n}(t)-F_{\boldsymbol{\gamma}}(t))|\\
    &\qquad +r2^{r-1}\int_{\R\setminus[-k,k]} |t|^{r-1} \, \sqrt{F_{\boldsymbol{\gamma}}(t)(1-F_{\boldsymbol{\gamma},n}(t))} \,  \frac{(F_{\boldsymbol{\gamma}}(t)-F_{\boldsymbol{\gamma}, n}(t))^+}{\sqrt{F_{\boldsymbol{\gamma}}(t)(1-F_{\boldsymbol{\gamma},n}(t))}}\,dt\\
    &\qquad +r2^{r-1}\int_{\R\setminus[-k,k]} |t|^{r-1} \, \sqrt{F_{\boldsymbol{\gamma}, n}(t)(1-F_{\boldsymbol{\gamma}}(t))} \, \frac{(F_{\boldsymbol{\gamma}, n}(t)-F_{\boldsymbol{\gamma}}(t))^+}{\sqrt{F_{\boldsymbol{\gamma}, n}(t)(1-F_{\boldsymbol{\gamma}}(t))}}\,dt.
\end{split}
\end{equation}
By Markov's inequality, we have for any $s \geq 1$ and any $t\in \R\setminus\{0\}$,
\begin{align*}
\sqrt{F_{\boldsymbol{\gamma}}(t)(1-F_{\boldsymbol{\gamma},n}(t))} \vee \sqrt{F_{\boldsymbol{\gamma}, n}(t)(1-F_{\boldsymbol{\gamma}}(t))} \le \sqrt{\frac{\mathbb{E}_{\P}[|(\mathbf{X},Y)^\top\boldsymbol{\gamma}|^s]\vee \mathbb{E}_{\P_{n}}[|(\mathbf{X},Y)^\top\boldsymbol{\gamma}|^s]}{|t|^s}}.
\end{align*}
Plugging these bounds into \eqref{eq:Bobkov_bound}, we obtain 
\begin{equation}
\begin{split}
\label{eq:W_bound_q}
&\mathcal{W}_r(\P_{\boldsymbol{\gamma}, n},\P_{\boldsymbol{\gamma}})^r \\
&\le 
r(2k)^r \sup_t |F_{\boldsymbol{\gamma},n}(t)-F_{\boldsymbol{\gamma}}(t))| \\
& +r2^{r-1}\int_{\R\setminus[-k,k]} |t|^{r-1 -s/2} \sqrt{\mathbb{E}_{\P}[|(\mathbf{X},Y)^\top\boldsymbol{\gamma}|^s]\vee \mathbb{E}_{\P_{n}}[|(\mathbf{X},Y)^\top\boldsymbol{\gamma}|^s]} \frac{(F_{\boldsymbol{\gamma}}(t)-F_{\boldsymbol{\gamma}, n}(t))^+}{\sqrt{F_{\boldsymbol{\gamma}}(t)(1-F_{\boldsymbol{\gamma},n}(t))}}\,dt\\
&+r2^{r-1}\int_{\R\setminus[-k,k]} r|t|^{r-1 -s/2} \sqrt{\mathbb{E}_{\P}[|(\mathbf{X,Y)}^\top\boldsymbol{\gamma}|^s]\vee \mathbb{E}_{\P_{n}}[|(\mathbf{X},Y)^\top\boldsymbol{\gamma}|^s]} \frac{(F_{\boldsymbol{\gamma}, n}(t)-F_{\boldsymbol{\gamma}}(t))^+}{\sqrt{F_{\boldsymbol{\gamma}, n}(t)(1-F_{\boldsymbol{\gamma}}(t))}}\,dt.
\end{split}
\end{equation}
Recall that we have assumed $s/2 > r$ where $r \geq 1$. In particular, this means that $|t|^{r-1-s/2}$ is integrable on $\R\setminus [-k,k]$ and 
\begin{align*}
r2^{r-1}\int_{\R\setminus[-k,k]} |t|^{r-1 -s/2}\, dt =\frac{2^rr}{s/2-r}k^{r-s/2}.
\end{align*}
Taking the supremum over $\boldsymbol{\gamma}$ and $t$ in \eqref{eq:W_bound_q} thus yields the claim.
\end{proof}

\begin{lemma}\label{lem:suboptimal}
With probability greater than $1-\alpha$ we have
\begin{align*}
&\sup_{(\boldsymbol{\gamma},t)\in \R^{d+1}\times \R} \frac{(F_{\boldsymbol{\gamma}}(t)-F_{\boldsymbol{\gamma}, n}(t))^+}{\sqrt{F_{\boldsymbol{\gamma}}(t)(1-F_{\boldsymbol{\gamma},n}(t))}}\vee
\sup_{(\boldsymbol{\gamma},t)\in \R^{d+1}\times \R} \frac{(F_{\boldsymbol{\gamma},n}(t)-F_{\boldsymbol{\gamma}}(t))^+}{\sqrt{F_{\boldsymbol{\gamma},n}(t)(1-F_{\boldsymbol{\gamma}}(t))}} \\
&\qquad\qquad\le 4\sqrt{\frac{\log(8/\alpha)+(d+2)\log(2n+1)}{n}}.
\end{align*}
\end{lemma}

\begin{proof}[Proof of Lemma \ref{lem:suboptimal}]
We first define 
\begin{align*}
\mathcal{J}= \left\{ \mathds{1}_{\left\{\mathbf{x}^\top \boldsymbol{\gamma} \, \le \,  t \right\}}, \, \mathds{1}_{\left\{\mathbf{x}^\top \boldsymbol{\gamma} \, > \,  t \right\}}:  (\boldsymbol{\gamma},t)\in \mathbb{R}^{d+1}\times \R \right\}\supseteq \mathcal{H},
\end{align*}
where $\mathcal{H}$ was defined in Lemma \ref{lem:3}.
Considering the cases $F_{\boldsymbol{\gamma}, n}(t)<1/2$ and  $F_{\boldsymbol{\gamma}, n}(t)\ge 1/2$ separately---noting that e.g.\ $\E_{\P_n}[\mathds{1}_{\{\mathbf{x}^\top \boldsymbol{\gamma} > t\}}]=1-\E_{\P_n}[\mathds{1}_{\{\mathbf{x}^\top \boldsymbol{\gamma} \le  t\}}]=1-F_{\boldsymbol{\gamma},n}(t)$---one can check
\begin{equation}
\begin{split}
\label{eq:lem3_bound1}
\sup_{(\boldsymbol{\gamma},t)\in \R^{d+1}\times \R}  \frac{(F_{\boldsymbol{\gamma}}(t)-F_{\boldsymbol{\gamma}, n}(t))^+}{\sqrt{F_{\boldsymbol{\gamma}}(t)(1-F_{\boldsymbol{\gamma},n}(t))}} &\le 2\left( \sup_{f\in \mathcal{J}} \, \frac{(\mathbb{E}_{\P}[f]-\E_{\P_n}[f])^+}{\sqrt{\E_{\P}[f]}} \vee \sup_{f\in \mathcal{J}} \, \frac{(\mathbb{E}_{\P_n}[f]-\E_{\P}[f])^+}{\sqrt{\E_{\P_n}[f]}}\right).
\end{split}
\end{equation}
By symmetry,
\begin{equation}
\begin{split}
\label{eq:lem3_bound2}
\sup_{(\boldsymbol{\gamma},t)\in \R^{d+1}\times \R} \frac{(F_{\boldsymbol{\gamma},n}(t)-F_{\boldsymbol{\gamma}}(t))^+}{\sqrt{F_{\boldsymbol{\gamma},n}(t)(1-F_{\boldsymbol{\gamma}}(t))}} \le2\left( \sup_{f\in \mathcal{J}} \, \frac{(\mathbb{E}_{\P}[f]-\E_{\P_n}[f])^+}{\sqrt{\E_{\P}[f]}} \vee \sup_{f\in \mathcal{J}} \, \frac{(\mathbb{E}_{\P_n}[f]-\E_{\P}[f])^+}{\sqrt{\E_{\P_n}[f]}}\right).
\end{split}
\end{equation}
Concentration for the terms on the right hand side of equations \eqref{eq:lem3_bound1} and \eqref{eq:lem3_bound2} is well studied: indeed, e.g.\ by \cite[Exercises 3.3 \& 3.4]{devroye2001combinatorial} we have
\begin{align*}
\P\left(\sup_{f\in \mathcal{J}} \, \frac{\mathbb{E}_{\P}[f]-\E_{\P_n}[f]}{\sqrt{\E_{\P}[f]}}> \epsilon\right)&\le 4 S_\mathcal{J}(2n)e^{-n\epsilon^2/4},\\
\P\left(\sup_{f\in \mathcal{J}} \, \frac{\E_{\P_n}[f]-\E_{\P}[f]}{\sqrt{\E_{\P_n}[f]}}> \epsilon\right)&\le 4 S_\mathcal{J}(2n)e^{-n\epsilon^2/4}.
\end{align*}
for all $\epsilon>0$,
where $S_\mathcal{J}(2n)$ is the shattering coefficient of $\mathcal{J}$.
Note that by \cite[Theorem 4.1]{devroye2001combinatorial} we have
$S_\mathcal{J}(2n)\le 2 S_\mathcal{H}(2n)$. As the VC-dimension of $\mathcal{H}$ is bounded by $d+2$, Sauer's lemma \cite[Theorem Corollary 4.1]{devroye2001combinatorial} yields
\begin{align*}
\log ( S_\mathcal{J}(2n))&\le (d+2)\log(2n+1).
\end{align*} 
The claim follows by solving the above expression for $\epsilon$.
\end{proof}

\section{Additional Derivations} 

\subsection{Diameter of the support of $\mathbb{P}$ in the simulation} \label{subsubsection:diameter_P} 
Notice that $\mathbf{\tilde{X}}_i^{\top} \boldsymbol{\beta} \ge - \norm{ \boldsymbol{\beta}^{-}}_1$ where the equality holds, making equal to ones the entries of $\mathbf{\tilde{X}}_i$ that corresponds negative values of $\boldsymbol{\beta}$. Similarly, $\mathbf{\tilde{X}}_i^{\top} \boldsymbol{\beta} \le  \norm{ \boldsymbol{\beta}^{+}}_1$. Since $\mathbf{X}_i = \sigma \lambda \mathbf{\tilde{X}}_i$, it follows that $\inf_{\mathbf{X}_i} X_i^{\top} \boldsymbol{\beta} = -\sigma \lambda \norm{ \boldsymbol{\beta}^{-}}_1$  and $\sup_{\mathbf{X}_i} \mathbf{X}_i^{\top} \boldsymbol{\beta} = \sigma \lambda \norm{ \boldsymbol{\beta}^{+}}_1$. Therefore, $Y_i \in [- \sigma (\lambda \norm{ \boldsymbol{\beta}^{-}}_1 + 1), \sigma (\lambda \norm{ \boldsymbol{\beta}^{+}}_1+1)]$. Then, diameter of the support equals
$$\sqrt{ d(\sigma \lambda)^2 + \sigma^2 (\lambda \norm{ \boldsymbol{\beta}^{+}}_1 + \lambda \norm{ \boldsymbol{\beta}^{-}}_1 + 2)^2 } = \sigma \lambda\sqrt{d +   ( \norm{\boldsymbol{\beta}}_1 + 2/\lambda)^2}. $$

\subsection{Derivation of Equation \eqref{eqn:conjecture}}
\label{subsubsection:bound_appendix}
The tuning parameter $\delta_{n,2}$ used in the simulation is 
\[ \delta_{n,2} = n^{-1/4} \cdot  C_{\textrm{sim}}, \quad \textrm{where } C_{\textrm{sim}} \equiv \left( q_{1-\alpha} \right)^{1/2} \cdot \sigma \lambda \left( d + \left( \| \boldsymbol{\beta} \|_{1}   + (2/\lambda) \right)^2 \right)^{1/2}  .    \]

According to equation \eqref{eqn:FOC_sqrtlasso} it is known that $\boldsymbol{\beta}=0_{d \times 1}$ is a solution to the $\sqrt{\text{LASSO}}$ problem if and only if
\begin{equation} \label{eqn:FOC_appendix} 
\frac{ \| \frac{1}{n} \sum_{i=1}^{n} \mathbf{X}_i y_i \|_\infty  }{\sqrt{\frac{1}{n} \sum_{i=1}^n y_i^2 }} \leq \delta_{n,2} = n^{-1/4} \cdot  C_{\textrm{sim}}. 
\end{equation}
Because 
\[ Y_i = \mathbf{X}_i^{\top} \boldsymbol{\beta} + \sigma \varepsilon_i, \]
equation \eqref{eqn:FOC_appendix} holds if and only if
\begin{equation} \label{eqn:aux_foc_appendix}
\frac{ \| \left( \frac{1}{n} \sum_{i=1}^{n} \mathbf{X}_i \mathbf{X}^{\top}_i \right) \boldsymbol{\beta} + \frac{\sigma}{\sqrt{n}} \frac{1}{\sqrt{n}} \sum_{i=1}^{n} \mathbf{X}_i \varepsilon_i \|_\infty  }{\sqrt{\frac{1}{n} \sum_{i=1}^n y_i^2 }} \leq n^{-1/4} \cdot  C_{\textrm{sim}}. 
\end{equation}
Because $\mathbf{X}_i = \sigma \lambda \tilde{\mathbf{X}}_i$ 
where $\tilde{\mathbf{X}}_i$ is a $d$-dimensional vector of independent uniform random variables over the $[0,1]$, then 
\[  \frac{1}{n} \sum_{i=1}^{n} \mathbf{X}_i \mathbf{X}^{\top}_i \overset{p}{\rightarrow} \mathbb{E} [ \mathbf{X}_i \mathbf{X}^{\top}_i] =  \frac{1}{3} \sigma^2 \lambda^2  \mathbb{I}_{d},\]
where we have used that $\mathbb{E}[ \mathbf{\tilde{X}}_{i,j}^2 ] = 1/3$  because $\mathbf{\tilde{X}}_{i,j}$ is a uniform distribution on the $[0,1]$ interval. 
The Continuous Mapping Theorem and Central Limit Theorem then imply that 
\[ \Big \| \left( \frac{1}{n} \sum_{i=1}^{n} \mathbf{X}_i \mathbf{X}^{\top}_i \right) \boldsymbol{\beta} + \frac{\sigma}{\sqrt{n}} \frac{1}{\sqrt{n}} \sum_{i=1}^{n} \mathbf{X}_i \varepsilon_i \Big \|_\infty \overset{p}{\rightarrow} \frac{1}{3} \sigma^2 \lambda^2 \| \boldsymbol{\beta} \|_{\infty}.     \]
For the denominator, 
\begin{eqnarray*}
\frac{1}{n} \sum_{i=1}^n y_i^2  \overset{p}{\rightarrow} \mathbb{E}[Y^2_{i}] = \boldsymbol{\beta}^{\top} \mathbb{E} [ \mathbf{X}_i \mathbf{X}^{\top}_i ] \boldsymbol{\beta} + \sigma^2 \mathbf{V}(\varepsilon_{i}) = \frac{1}{3} \sigma^2 \lambda^2 \boldsymbol{\beta}^{\top} \boldsymbol{\beta}   + \sigma^2 \mathbf{V}(\varepsilon_{i})  ,  
\end{eqnarray*}
where we have used the fact that $\epsilon_{i}$ is mean zero and independent of $\mathbf{\tilde{X}}_i$. 
The left-hand side of \eqref{eqn:aux_foc_appendix} is thus bounded above with high probability by 

\[ \frac{ (1/3) \sigma^2 \lambda^2 \| \boldsymbol{\beta} \|_{\infty}}{ \sqrt{(1/3) \sigma^2 \lambda^2} \sqrt{\boldsymbol{\beta}^{\top}  \boldsymbol{\beta}  } } = \sqrt{1/3}  \cdot \sigma  \cdot \lambda \cdot \Big \| \frac{\boldsymbol{\beta}}{ \sqrt{\boldsymbol{\beta}^{\top}  \boldsymbol{\beta}  } }  \Big \|_{\infty}.  \]

This means that the event in \eqref{eqn:FOC_appendix} occurs with high probability if 
\begin{equation} \label{eqn:zeros_appendix}
\sqrt{1/3}  \cdot \sigma  \cdot \lambda \cdot \Big \| \frac{\boldsymbol{\beta}}{ \sqrt{\boldsymbol{\beta}^{\top}  \boldsymbol{\beta}  } }  \Big \|_{\infty} \leq  \frac{1}{n^{1/4}} C_{\textrm{sim}}.
\end{equation}
Using the definition of $C_{\textrm{sim}}$, the event in \eqref{eqn:zeros_appendix} occurs if and only if

\[ n \leq 9 \cdot \Big \| \frac{\boldsymbol{\beta}}{ \sqrt{\boldsymbol{\beta}^{\top}  \boldsymbol{\beta}  } }  \Big \|^{-1/4}_{\infty} \left( q_{1-\alpha} \right)^{2} \cdot \left( d + \left( \| \boldsymbol{\beta} \|_{1}   + (2/\lambda) \right)^2 \right)^{2}.  \]

Thus, a sample size smaller than the right-hand side of the equation above implies that, with high probability, $\boldsymbol{\beta}=0_{d \times 1}$ will be a solution to the $\sqrt{\text{LASSO}}$ problem. 

\subsection{Gaussian distributions with similar prediction errors, but large Wasserstein distance}\label{subsection:gaussian_example} 

Suppose $(\mathbf{X},Y) \sim N_{d+1}(0, \mathbb{I}_{d+1} )$. Because $(\mathbf{X},Y)$ are, by assumption, independent and have mean zero, the prediction error of any linear predictor $\mathbf{X}^{\top} \boldsymbol{\gamma}$ equals the variance of $Y$ plus the variance of $\mathbf{X}^{\top} \boldsymbol{\gamma}$; that is
\[ \mathbb{E}[(Y-\mathbf{X}^{\top} \boldsymbol{\boldsymbol{\gamma}})^2]=1+ \| \boldsymbol{\gamma} \|_2^2. \]
The prediction error scales with $\|\boldsymbol{\gamma}\|_2$, so it makes sense to restrict this norm. Let us focus on predictors for which $\|\boldsymbol{\gamma}\|_2 = 1$. 

Consider now a random vector $(\Tilde{\mathbf{X}}, \Tilde{Y})$. Assume $\Tilde{\mathbf{X}} = \mathbf{X} + \mathbf{V}$ where $\mathbf{V} \sim  N_d(0, \sigma^2_v \mathbb{I}_d )$ is an independent source of measurement error. Set $\Tilde{Y} = Y$. For any $\boldsymbol{\gamma}$ such that $\|\boldsymbol{\gamma}\|=1$:
\[ \mathbb{E}\left[(\tilde{Y}-\Tilde{\mathbf{X}}^{\top}\boldsymbol{\gamma})^2\right]=2 + \sigma^2_v. \] 
Thus, in this example 
\[ \left( \mathbb{E}\left[(Y-\mathbf{X}^{\top} \boldsymbol{\boldsymbol{\gamma}})^2\right] - \mathbb{E}\left[(\tilde{Y}-\Tilde{\mathbf{X}}^{\top}\boldsymbol{\gamma})^2 \right] \right) = \sigma^2_v. \]
Consequently, the difference in prediction errors equals $\sigma^2_v$.  

Let $\mathbb{P}$ denote the distribution of $(\mathbf{X},Y)$ and, analogously, let $\Tilde{\mathbb{P}}$ denote the  distribution of $(\Tilde{\mathbf{X}},\Tilde{Y})$. The distance between $\mathbb{P}$ and $\Tilde{\mathbb{P}}$ can be considerably large when measured using the standard $d$-dimensional Wasserstein metric. In fact, algebra shows that 
        $$ \mathcal{W}_2(\mathbb{P},\Tilde{\mathbb{P}}) = ( \sqrt{1  + \sigma_v^2} - 1) d^{1/2}~, $$
Thus, when $d$ is large, the standard $d$-dimensional Wasserstein distance suggests that $\mathbb{P}$ and $\Tilde{\mathbb{P}}$ are very different from one another. This stands in contrast with the magnitude of the difference in prediction errors associated to $\mathbb{P}$ and $\Tilde{\mathbb{P}}$ which is equal to $\sigma^2_v$.  

In this example, one can further show that if we set $\rho(\cdot)= \|\cdot \|_1$, then for any $\sigma > 0$ we have 
$$ \widehat{\mathcal{W}}_{2, \rho, \sigma}(\mathbb{P},\Tilde{\mathbb{P}}) \le \sigma_v~.$$

Therefore, the example further shows that two distributions can be close in $\rho$-MSW metric, even when their standard $d$-dimensional Wasserstein distance is large. 

\subsection{Alternative criterion for choosing $\delta$}\label{subsection:blanchet_delta}

In this section we present a brief discussion of the differences between our recommendation for selecting $\delta$ (which specifically targets out-of-sample performance) and the recommendation in \cite{blanchet2019robust}. We show that if we use the same criterion as in \cite{blanchet2019robust}, our optimal $\delta$ would be upper bounded by their recommendation. As we explain below, this has to do with the fact that our $\rho$-MSW balls are larger than those based on the standard Wasserstein metric.

\textit{Distributionally robust representation:} First, it is worth mentioning that both our paper and  \cite{blanchet2019robust} present a distributionally robust representation of the $\sqrt{\text{LASSO}}$ and related estimators. The key difference is that we define our class of testing distributions using the $\rho$-MSW metric (which we have denoted by $\widehat{\mathcal{W}}_{r, \rho, \sigma}$),  instead of the Wasserstein metric. Based on our Remark 2, our balls are larger than the Wasserstein balls used in \cite{blanchet2019robust}.  

Another difference, relative to their results, is that \cite{blanchet2019robust} take the same testing and training distributions of the outcome variable; see their Proposition 2, Equation (14), Theorem 1. Thus, to make our results comparable to them we set $\sigma=0$ and use the notation $\widehat{\mathcal{W}}_{r, \rho, 0}$. We denote the Wasserstein metric used in \cite{blanchet2019robust} by $\mathcal{W}^B$.

\textit{Criterion:} \cite{blanchet2019robust}  recommend $\delta^*_n$ as the $1-\alpha$ quantile of the profile function $R_n(\boldsymbol{\beta})$: 
\begin{align}\label{eq:blanchet}
R_n(\boldsymbol{\beta})=\min\{ \mathcal{W}^B(\P_n, \Q):\, \boldsymbol{\gamma} \mapsto \E_{\Q}[|Y- \mathbf{X}^\top \boldsymbol{\gamma}|^r] \text{ has minimizer } \boldsymbol{\beta}\},
\end{align}
see \cite[eq.\ (16) and Section 4.2]{blanchet2019robust}. Assuming a linear regression model with Gaussian errors, $Y= \mathbf{X}^T \boldsymbol{\beta}^* + e$, and an appropriate normalization of the covariates $\mathbf{X}$, it can be shown that 
$$\sqrt{R_n(\boldsymbol{\beta}^*)}\le \frac{\pi}{\pi-2} \frac{\Phi^{-1}(1-\alpha/2d)}{\sqrt{n}},$$
with probability asymptotically larger than $1-\alpha$, as $n \to \infty$, see \cite[Theorem 7 and Remark 1]{blanchet2019robust}. This aligns with the recommendation of \cite{belloni2011square}.

Consider a modification of their profile function based on the $\rho$-MSW metric as follows:
$$ \Tilde{R}_n(\boldsymbol{\beta})=\min\{ \widehat{\mathcal{W}}_{r, \rho, 0}(\P_n, \Q):\, \boldsymbol{\gamma} \mapsto \E_{\Q}[|Y- \mathbf{X}^\top \boldsymbol{\gamma}|^r] \text{ has minimizer } \boldsymbol{\beta}\}~.$$
If we define by $\Tilde{\delta}_n^*$ the $1-\alpha$ quantile of the modified profile function $\Tilde{R}_n(\boldsymbol{\beta})$, the event $ \{{R}_n(\boldsymbol{\beta}) \le z \}$ is included in the event   $ \{ \Tilde{R}_n(\boldsymbol{\beta}) \le z \}$ because for any $\Q$ such that $\boldsymbol{\beta}$ is a minimizer of $\E_{\Q}[|Y- \mathbf{X}^\top \boldsymbol{\gamma}|^r]$ we have that $\widehat{\mathcal{W}}_{r, \rho, 0}(\P_n, \Q) \le \mathcal{W}^B(\P_n, \Q)$ due to Remark 2. This implies that
$$ \P(\Tilde{R}_n(\boldsymbol{\beta}) \le z ) \ge \P({R}_n(\boldsymbol{\beta}) \le z )~, $$
which allows us to conclude  $\Tilde{\delta}_n^* \le \delta_n^*$.

\section{Additional Simulations} \label{sec:additional_sims}

Suppose now that the training data consists of $n$ i.i.d.\ draws from a Gaussian, homoskedastic, linear regression model. In other words
\[ Y_i = \mathbf{X}_i^{\top} \boldsymbol{\beta} + \sigma_{\varepsilon} \varepsilon_i, \]
where $\varepsilon_i \sim \mathcal{N}(0,1)$ and $\mathbf{X}_i \sim \mathcal{N}_d(\mathbf{0}, \mathbb{I}_d)$, with $\varepsilon_{i} \bot \mathbf{X}_i$. The parameters controlling the simulation design are $(\boldsymbol{\beta}, \sigma_{\epsilon}, d)$. 

We are interested in comparing the out-of-sample performance of a linear predictor that uses coefficients estimated via the $\sqrt{\text{LASSO}}$ ($r=2$), with other popular regularization procedures (Ridge regression and the LASSO). We use the standard tuning parameter for the $\sqrt{\text{LASSO}}$ in \cite{belloni2014pivotal} 
\[ \delta_n \equiv  (1.1) \cdot \Phi^{-1} \left( 1 - \frac{\alpha}{2d} \right) \cdot n^{-1/2},  \]
and we take $\alpha=.05$. For Ridge regression, we use the approximately optimal oracle recommendation in Corollary 6 of \cite{hastie2022surprises}.\footnote{In our set-up this equals $\sigma_{\epsilon} d / \| \mathbf{\beta}\|^2$.} For LASSO, we use the oracle recommendation discussed in \cite{belloni2011square}, which equals $\delta_n$ multiplied by the unknown paramater $\sigma_{\epsilon}$.\footnote{We implement the LASSO using the Matlab function \texttt{lasso}. }

We consider different sizes of the training data $n_{\textrm{train}} \in \{2,50,100,150, \ldots, 2000\}$. We set the size of the testing data to be $n_{\textrm{testing}}=1000$, and we benchmark the performance of each of the regularized estimators relative to the root mean-squared prediction error of a predictor based on ordinary least-squares (OLS). When $d>n_{\textrm{train}}$ we use the ``ridgless'' estimator in \cite{hastie2022surprises} as a benchmark.  

For the testing data, we consider two different distributions. First, we consider $n_{\textrm{testing}}$ new draws from the true data generating process. Second, we perturb the true data generating process according to the worst-case distribution derived in  Corollary 1 with $\delta_n$ equal to the tuning parameter used by the $\sqrt{\text{LASSO}}$. According to the first theorem in the paper, the $\sqrt{\text{LASSO}}$ offers robustness against these types of perturbations. 

The simulation results provide two interesting findings. First, when the testing distribution and the DGP are the same, the out-of-sample prediction error of the Ridge, LASSO, and $\sqrt{\text{LASSO}}$ estimators are larger than the OLS estimator. In this case, neither of the penalized estimators has any attractive property in terms of out-of-sample performance. Second,  when the testing distribution and the DGP are different---and, in particular, the testing distribution is adversarial---then the OLS estimator has a larger out-of-sample prediction error in comparison to the penalized estimators. Moreover, the $\sqrt{\text{LASSO}}$ estimator reports the lowest out-of-sample prediction error among the estimators.  Importantly, in this case, the $\sqrt{\text{LASSO}}$  has clearly superior performance (up to 25\%) relative to optimally tuned Ridge and LASSO.

\begin{figure}[t!]
    \begin{center}
     \begin{subfigure}[b]{0.45\textwidth}
        \centering
        \includegraphics[width=\textwidth]{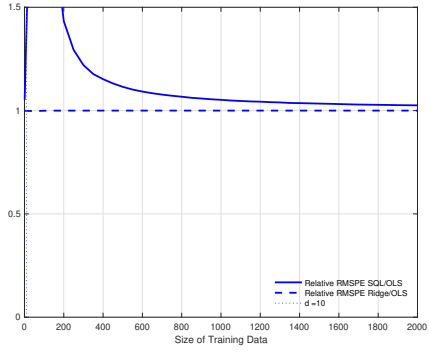}
    \caption{ Ridge: $d=10$ }
    \end{subfigure}
    \hfill
    \begin{subfigure}[b]{0.45\textwidth}
        \centering
        \includegraphics[width=\textwidth]{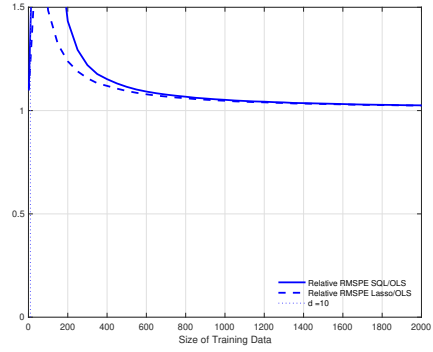}
    \caption{ LASSO: $d=10$}     
    \end{subfigure}  
    \vskip\baselineskip
    \begin{subfigure}[b]{0.45\textwidth}
        \centering
        \includegraphics[width=\textwidth]{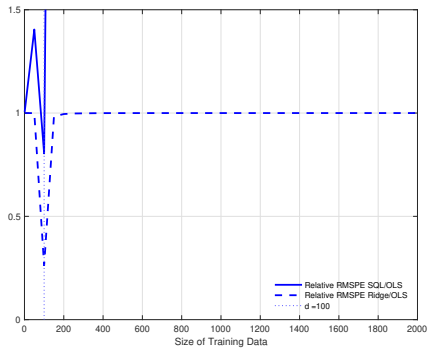}
    \caption{ Ridge: $d=100$ }
    \end{subfigure}
    \hfill 
    \begin{subfigure}[b]{0.45\textwidth}
        \centering
        \includegraphics[width=\textwidth]{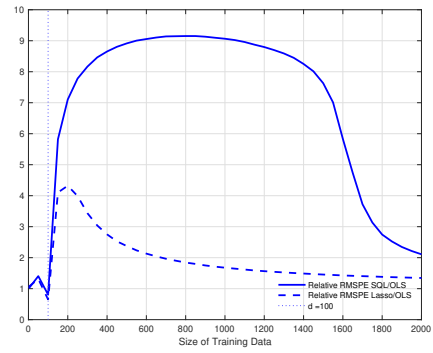}
    \caption{ LASSO: $d=100$} 
    \end{subfigure}
    
    \end{center}
    \caption{\small Relative out-of-sample prediction error of the Ridge, LASSO, and $\sqrt{\text{LASSO}}$ estimators with respect to the OLS estimator. No perturbation in testing data. $\sqrt{\text{LASSO}}$ estimator is the solid line. A dashed line is used for the other estimators. $\boldsymbol{\beta} = (1,\ldots,1)^{\top}$ (vector of $d$ ones), $\sigma_\varepsilon = 1$.}
    \label{fig:training_testing}
    
\end{figure}

\begin{figure}[t!]
    \begin{center}
     \begin{subfigure}[b]{0.45\textwidth}
        \centering
        \includegraphics[width=\textwidth]{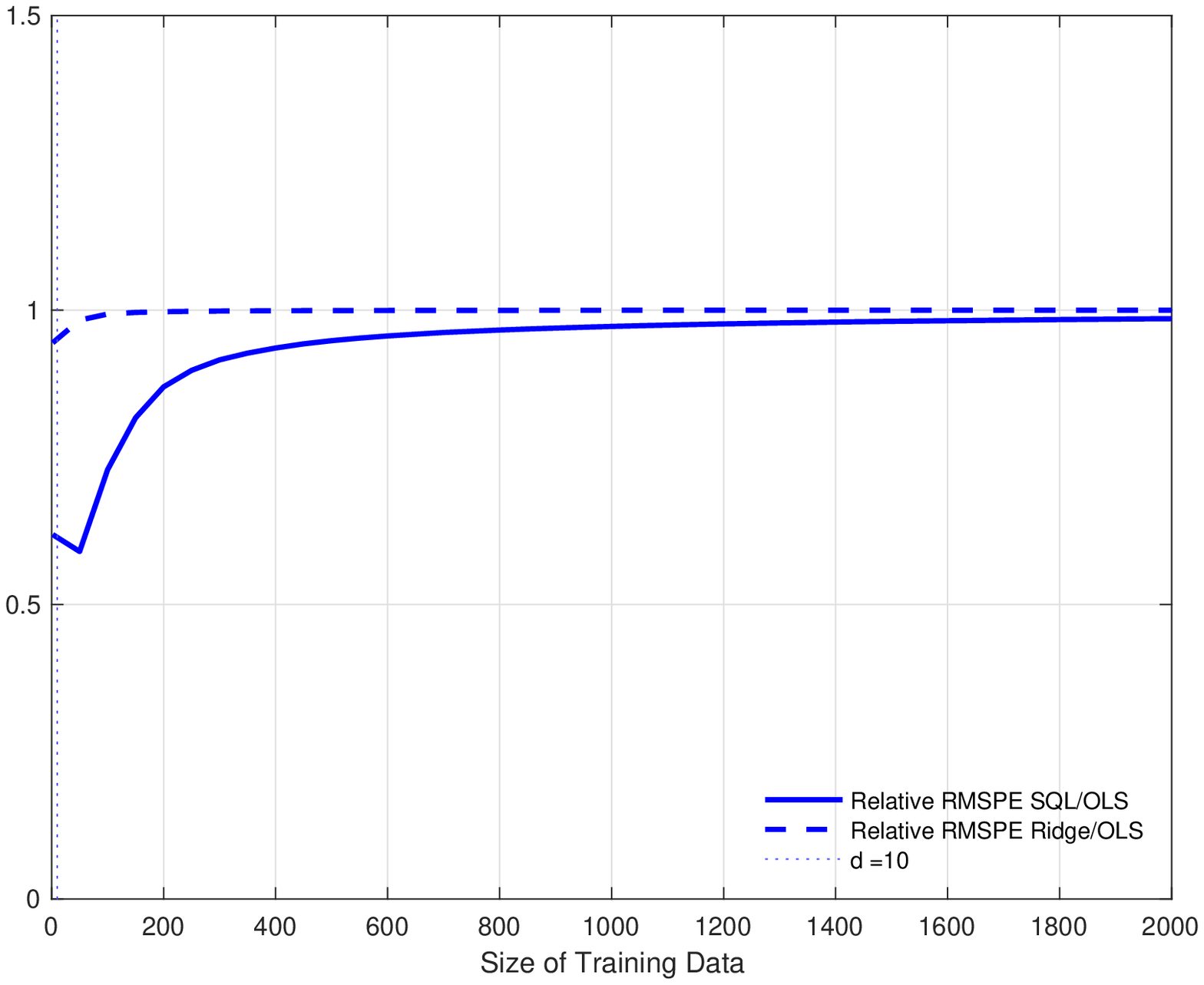}
    \caption{ Ridge: $d=10$ }
    \end{subfigure}
    \hfill
    \begin{subfigure}[b]{0.45\textwidth}
        \centering
        \includegraphics[width=\textwidth]{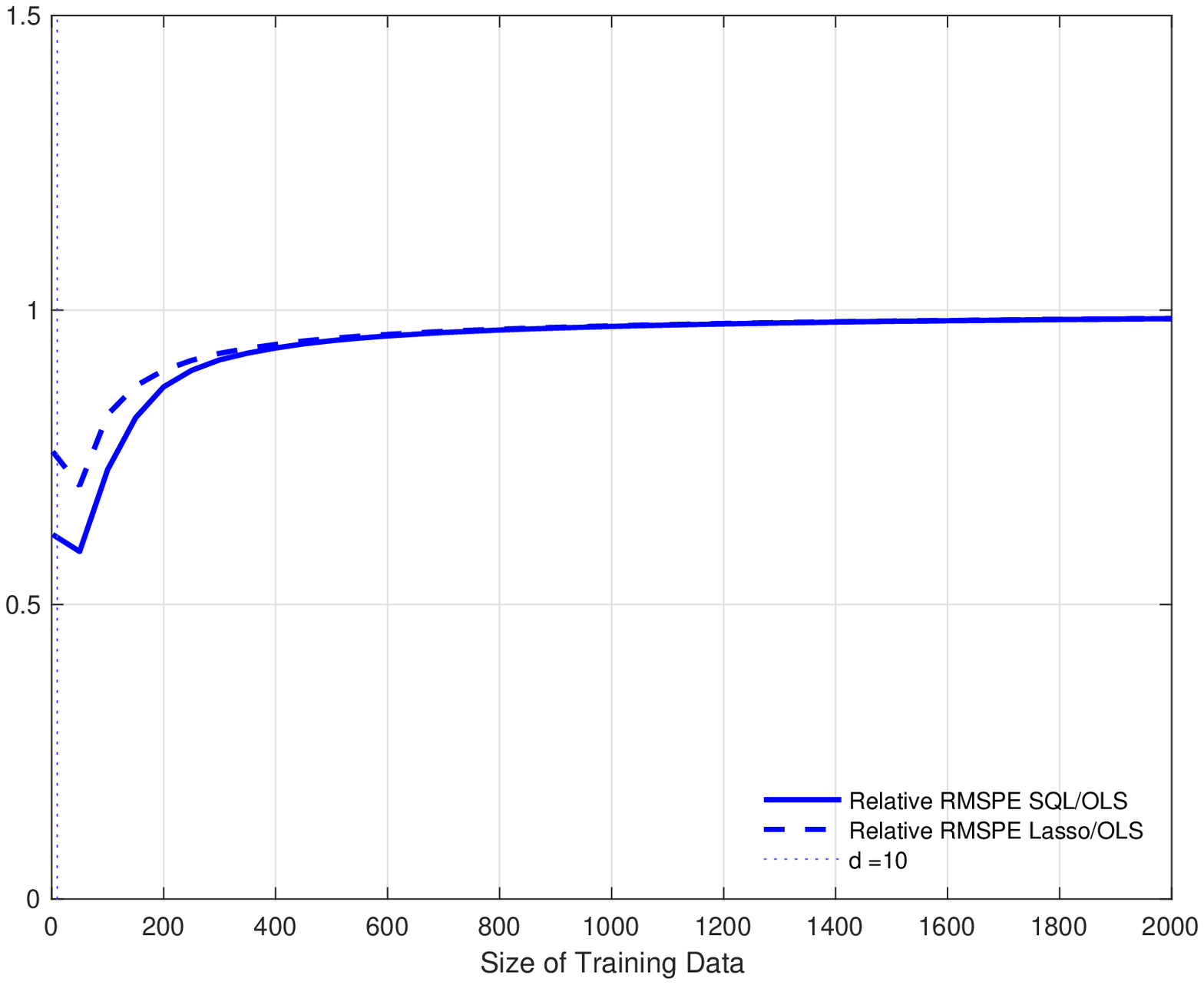}
    \caption{ LASSO: $d=10$}     
    \end{subfigure}  
    \vskip\baselineskip
    \begin{subfigure}[b]{0.45\textwidth}
        \centering
        \includegraphics[width=\textwidth]{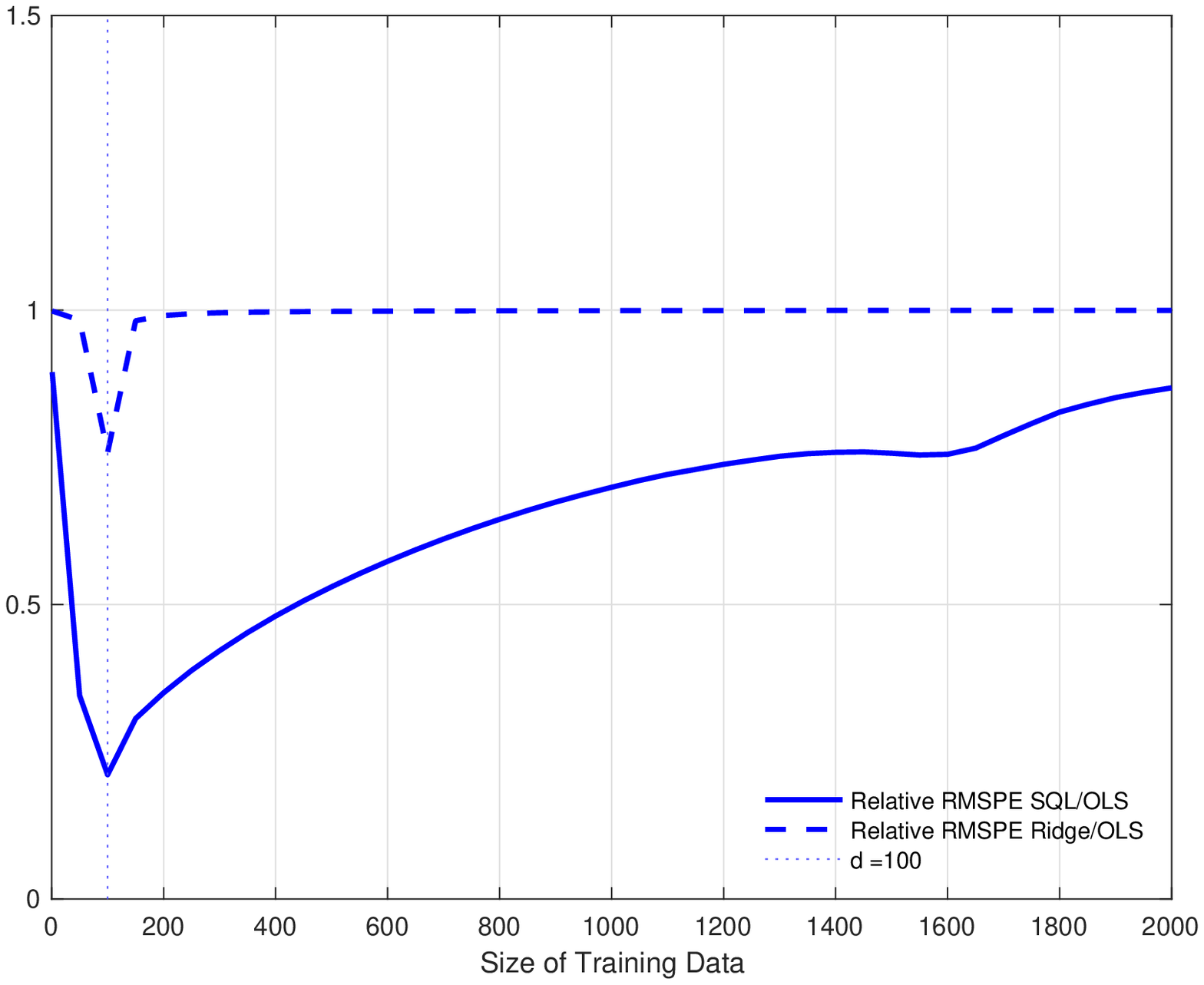}
    \caption{ Ridge: $d=100$ }
    \end{subfigure}
    \hfill 
    \begin{subfigure}[b]{0.45\textwidth}
        \centering
        \includegraphics[width=\textwidth]{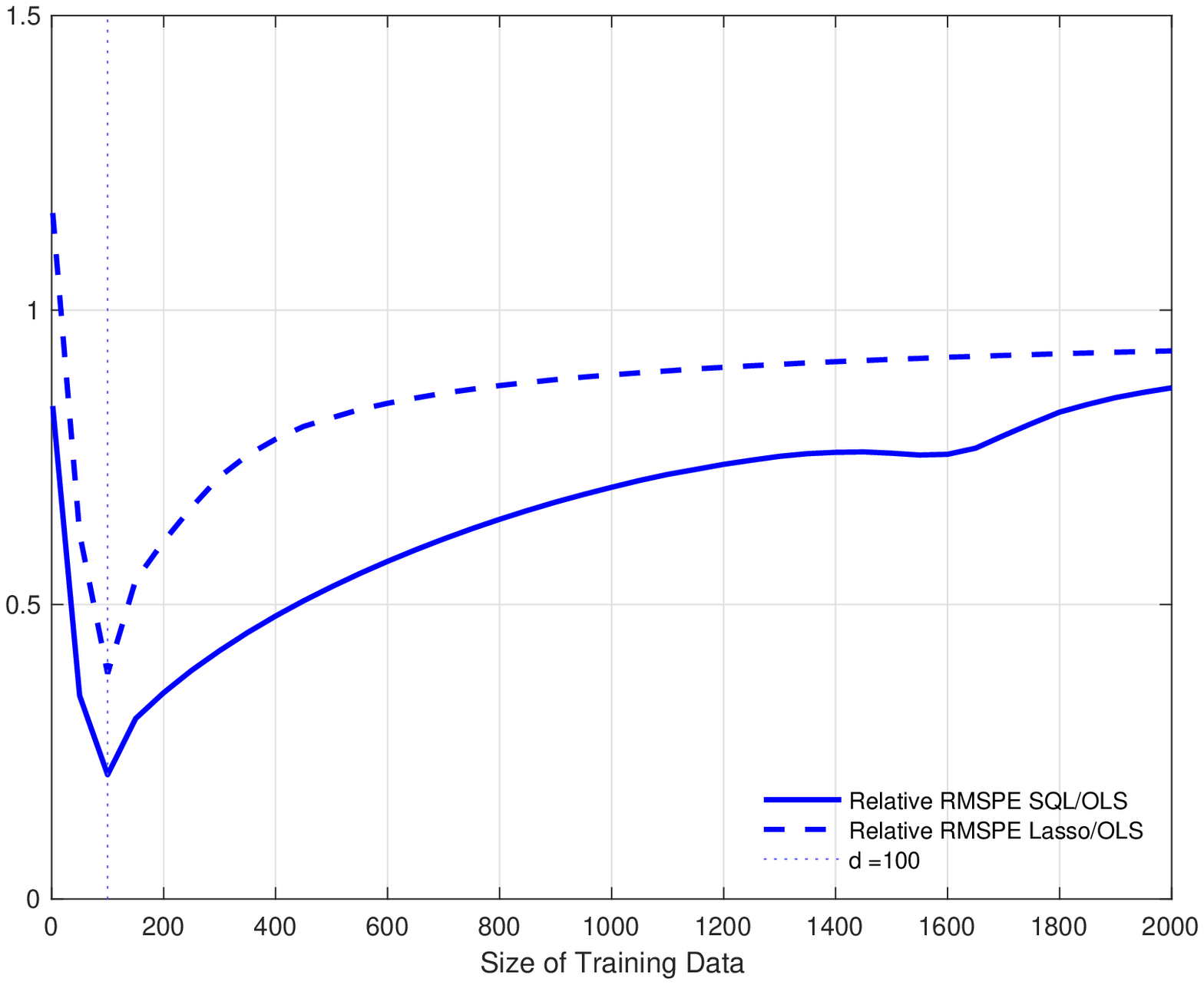}
    \caption{ LASSO: $d=100$} 
    \end{subfigure}
    
    \end{center}
    \caption{ \small Relative out-of-sample prediction error of the Ridge, LASSO, and $\sqrt{\text{LASSO}}$ estimators with respect to the OLS estimator. Perturbed testing data. $\sqrt{\text{LASSO}}$ estimator is the solid line. A dashed line for the other estimators. $\boldsymbol{\beta} = (1,\ldots,1)^{\top}$ (vector of $d$ ones), $\sigma_\varepsilon = 1$.}
    \label{fig:training_testing}
    
\end{figure}

As we discussed in Section \ref{sec:simulation}, the oracle recommendation for the regularization parameter of $\sqrt{\textrm{LASSO}}$ (based on our analysis of the $\rho$-MSW metric) can be more than 10 times larger than the standard recommendation in \cite{belloni2011square}. This raises the question of whether the out-of-sample performance of the Lasso reported in Panel d) of Figure \ref{fig:training_testing} can be improved by also using a larger regularization parameter. Lemma 2 in \cite{tian2018selective}---which shows that the $\sqrt{\textrm{LASSO}}$ and the LASSO share an explicit reparameterization of their solution paths conditional on the data---imply that this is indeed possible. In fact, Lemma 2 in \cite{tian2018selective} shows that, for each data realization and each possible regularization parameter, $\delta_{\textrm{SQL}}$, one can find a regularization parameter for the LASSO, $\delta_{\textrm{LASSO}}$, such that both the LASSO and the $\sqrt{\textrm{LASSO}}$ estimators coincide. As a consequence, using $\delta_{\textrm{LASSO}}$, guarantees that the out-of-sample performance of the two procedures must coincide.  

We consider the same Gaussian, homoskedastic, linear regression model described at the beginning of this section. We set $d=100$ and set $\boldsymbol{\beta} = (1,\ldots,1)^{\top}$ and $\sigma_{\epsilon}=1$. We consider different training sizes $n_{\textrm{train}} \in \{200,250,300, \ldots, 2000\}$. For each data realization we implement the formula in Lemma 2 in \cite{tian2018selective} to obtain a new regularization parameter $\delta_{\textrm{LASSO}}$. For the $\sqrt{\textrm{LASSO}}$ we once again use the tuning parameter in \cite{belloni2011square}. The testing distribution is the worst-case distribution derived in Corollary 1. 

Panel a) in Figure \eqref{fig:Lemma2_Tian} reports the ratio of $\delta_{\textrm{LASSO}}$ relative the oracle regularization parameter for the LASSO. For each sample size $n_{\textrm{train}}$, the figure reports the average ratio across data realizations. The figure shows that in order for the LASSO to have the same out-of-sample performance as the $\sqrt{\textrm{LASSO}}$, the new regularization parameter needs to be, on average, 10 times larger than the standard tuning parameter. Panel b) in Figure \eqref{fig:Lemma2_Tian} confirms that the new regularization parameter indeed aligns the out-of-sample performance of both procedures.  

\begin{figure}[t!]
    \begin{center}
     \begin{subfigure}[b]{0.45\textwidth}
        \centering
        \includegraphics[width=\textwidth]{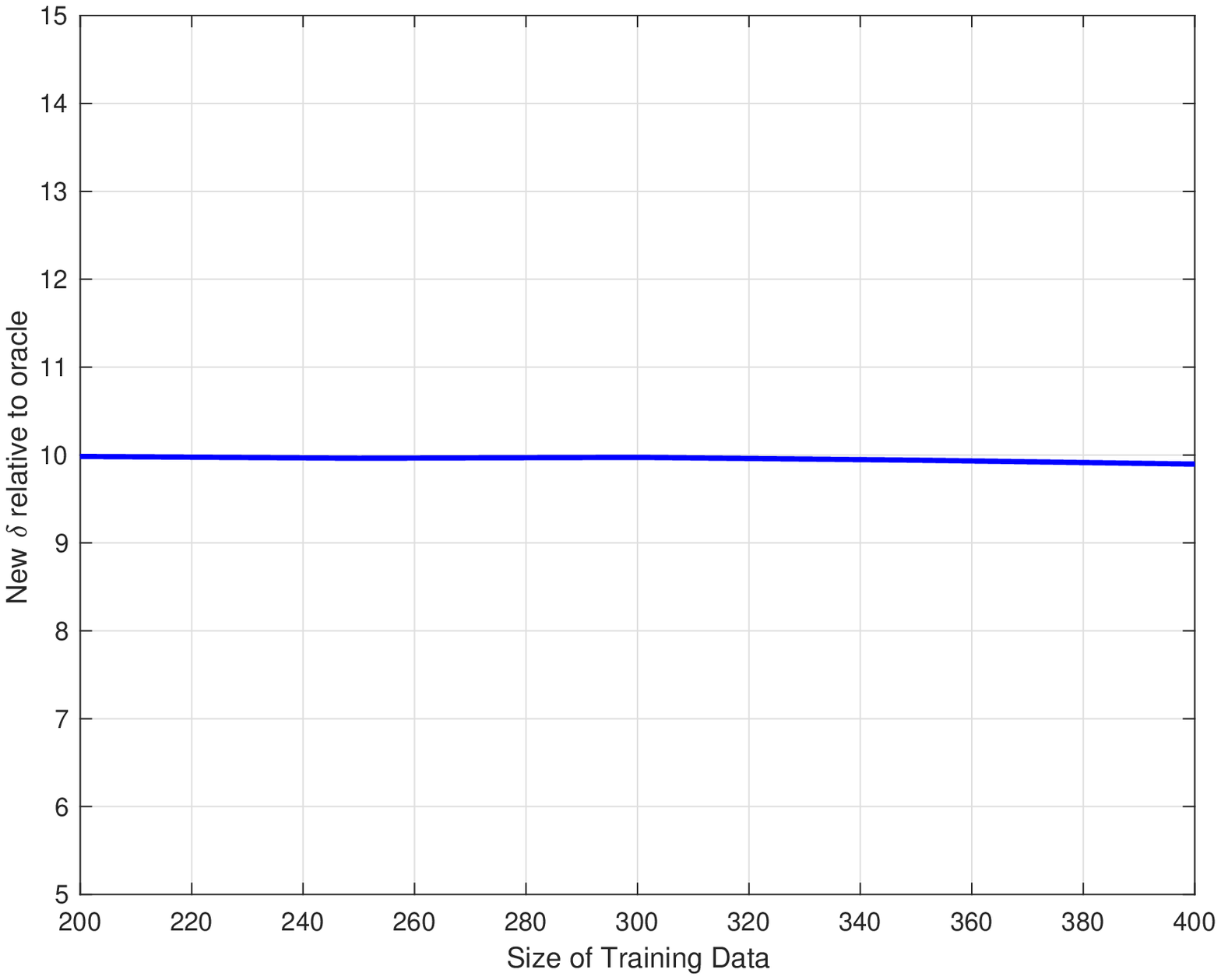}
    \caption{ $\delta_{\textrm{LASSO}}/\delta_{\textrm{oracle}}$}
    \end{subfigure}
    \hfill
    \begin{subfigure}[b]{0.45\textwidth}
        \centering
        \includegraphics[width=\textwidth]{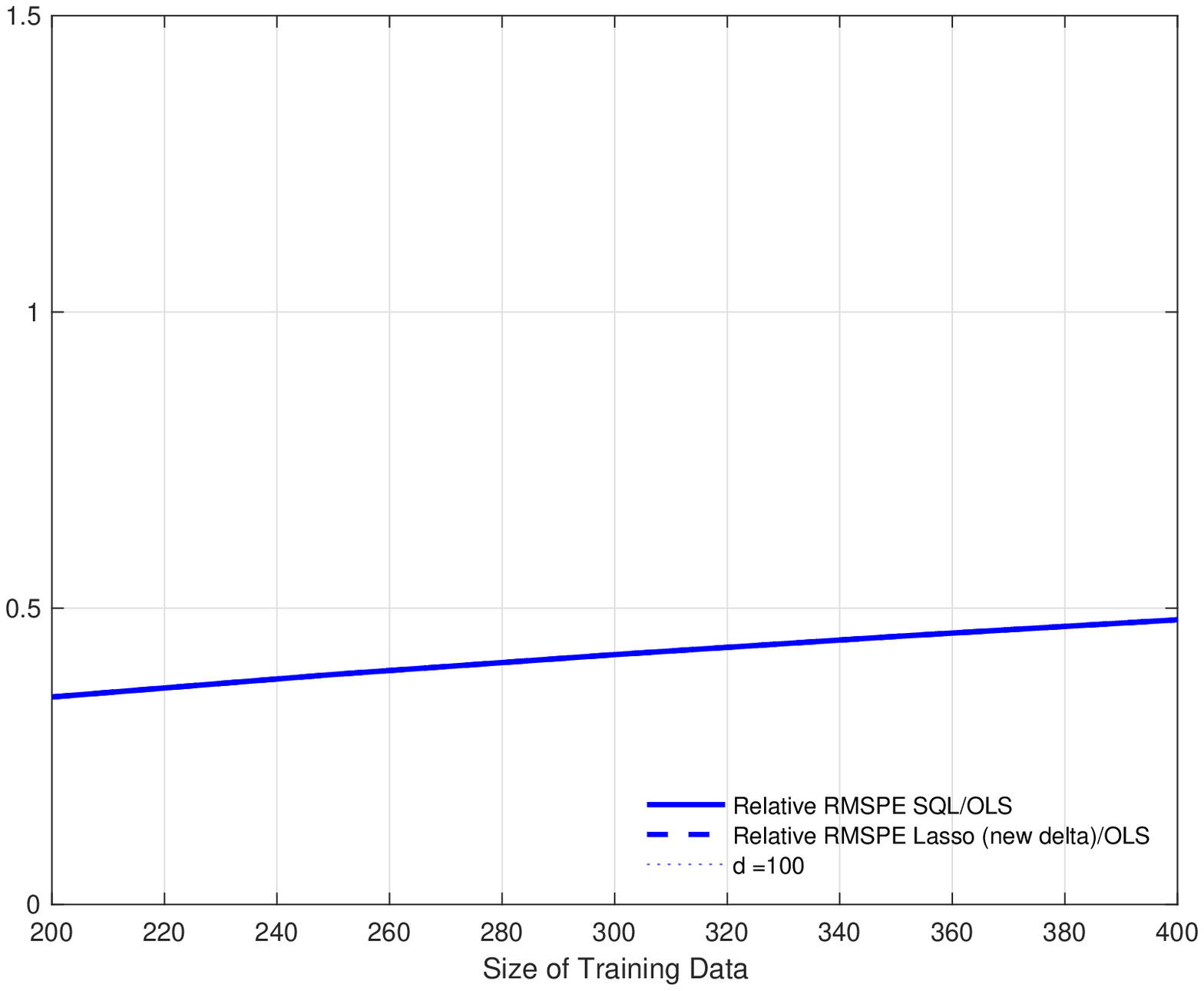}
    \caption{ LASSO vs. $\sqrt{\textrm{LASSO}}$ }     
    \end{subfigure}  

    \end{center}
    \caption{\small Panel a) reports the regularization parameter, $\delta_{\textrm{LASSO}}$, such that both the LASSO and $\sqrt{\textrm{LASSO}}$ coincide; following Lemma 2 in \cite{tian2018selective}. Panel b) reports the relative out-of-sample prediction for both the $\sqrt{\textrm{LASSO}}$ and the LASSO, but with the latter using the regularization parameter from Lemma 2 in \cite{tian2018selective}.}
    \label{fig:Lemma2_Tian}
    
\end{figure}

\end{document}